\documentclass[english]{amsart}
\pdfoutput=1

\usepackage{graphicx}
\usepackage{amsmath,amsfonts,amsthm}
\usepackage{amssymb}
\usepackage{multirow,array}
\usepackage{comment}
\usepackage{enumitem}
\usepackage[utf8]{inputenc}
\usepackage[T1]{fontenc}
\usepackage{emptypage}
\usepackage{hyperref}
\usepackage{color}
\usepackage{mathrsfs}
\usepackage{url}
\usepackage{mathrsfs}%
\usepackage[T1]{fontenc}
\usepackage{marvosym}
\usepackage{babel}
\usepackage{amsmath,amsfonts,amsthm}%
\usepackage{amsmath}%
\usepackage{amsfonts}%
\usepackage{amssymb}%
\usepackage{graphicx}

\usepackage{shuffle}

\usepackage{mathrsfs}

\usepackage{tikz}
\usepackage{mhequ}

\usepackage{geometry}

\usetikzlibrary{arrows,backgrounds}

\usetikzlibrary{decorations}
\usetikzlibrary{positioning}
\usetikzlibrary{shapes}

\tikzset{
root/.style={circle,fill=black!50,inner sep=0pt, minimum size=3mm},
        dot/.style={circle,fill=black,inner sep=0pt, minimum size=1.5mm},
         bluedot/.style={circle,fill=blue,inner sep=0pt, minimum size=1.5mm},
        reddot/.style={circle,fill=red,inner sep=0pt, minimum size=1.5mm},
        var/.style={circle,fill=black!10,draw=black,inner sep=0pt, minimum size=3mm},
        kernel/.style={semithick,shorten >=2pt,shorten <=2pt},
        kernel1/.style={thick},
        kernels/.style={snake=zigzag,shorten >=2pt,shorten <=2pt,segment amplitude=1pt,segment length=4pt,line before snake=2pt,line after snake=5pt,},
		kernels1/.style={snake=zigzag,segment amplitude=0.5pt,segment length=2pt},
		rho1/.style={dotted,semithick},
        rho/.style={densely dashed,semithick,shorten >=2pt,shorten <=2pt},
           testfcn/.style={dotted,semithick,shorten >=2pt,shorten <=2pt},
        renorm/.style={shape=circle,fill=white,inner sep=1pt},
        labl/.style={shape=rectangle,fill=white,inner sep=1pt},
        xic/.style={very thin,circle,fill=symbols,draw=black,inner sep=0pt,minimum size=1.2mm},
        xi/.style={very thin,circle,fill=blue!10,draw=black,inner sep=0pt,minimum size=1.2mm},
        xix/.style={crosscircle,fill=blue!10,draw=black,inner sep=0pt,minimum size=1.2mm},
	xib/.style={very thin,circle,fill=blue!10,draw=black,inner sep=0pt,minimum size=1.6mm},
	xie/.style={very thin,circle,fill=green!50!black,draw=black,inner sep=0pt,minimum size=1mm},
	xid/.style={very thin,circle,fill=symbols,draw=black,inner sep=0pt,minimum size=1.6mm},
	xibx/.style={crosscircle,fill=blue!10,draw=black,inner sep=0pt,minimum size=1.6mm},
	edgetype/.style={very thin,circle,draw=black,inner sep=0pt,minimum size=5mm},
	nodetype/.style={very thick,circle,draw=black,inner sep=0pt,minimum size=5mm},
	kernels2/.style={very thick,draw=connection,segment length=12pt},
clean/.style={thin,circle,fill=black,inner sep=0pt,minimum size=1mm},	 not/.style={thin,circle,fill=symbols,draw=connection,fill=connection,inner sep=0pt,minimum size=0.5mm},
	>=stealth,
        }

\newcommand{\vep}{\varepsilon}
\newcommand{\tX}{\tilde{X}}
\newcommand{\tY}{\tilde{Y}}
\newcommand{\tBX}{\tilde{\mathbf{X}}}

\newcommand{\MP}{\mathcal{P}}
\newcommand{\tx}{\tilde{x}}
\newcommand{\tZ}{\tilde{Z}}
\newcommand{\MH}{\mathcal{H}}
\newcommand{\MF}{\mathcal{F}}
\newcommand{\MT}{\mathcal{T}}
\newcommand{\MM}{\mathcal{M}}

\newcommand{\ML}{\mathcal{L}}
\newcommand{\ty}{\tilde{y}}

\newcommand{\MV}{\mathcal{V}}
\newcommand{\tBY}{\tilde{\mathbf{Y}}}
\newcommand{\tBZ}{\tilde{\mathbf{Z}}}

\newtheorem{thm}{Theorem}[section]

\newtheorem{lem}[thm]{Lemma}
\newtheorem{coro}[thm]{Corollary}
\newtheorem{rem}[thm]{Remark}
\newtheorem{prop}[thm]{Proposition}
\newtheorem{example}[thm]{Example}
\newtheorem{defi}[thm]{Definition}

\let\oldtocsection=\tocsection

\let\oldtocsubsection=\tocsubsection

\let\oldtocsubsubsection=\tocsubsubsection

\renewcommand{\tocsection}[2]{\hspace{0em}\oldtocsection{#1}{#2}}
\renewcommand{\tocsubsection}[2]{\hspace{1em}\oldtocsubsection{#1}{#2}}
\renewcommand{\tocsubsubsection}[2]{\hspace{2em}\oldtocsubsubsection{#1}{#2}}

\begin{document}

\title{Differential equations driven by rough paths with jumps}

\author{Peter K. Friz}
\address{P.K. Friz, Institut f\"ur Mathematik, Technische Universit\"at Berlin, and Weierstra\ss --Institut f\"ur Angewandte Analysis und Stochastik, Berlin, Germany}
\email{friz@math.tu-berlin.de (corresponding author)}

\author{Huilin Zhang}
\address{H. Zhang, Institut f\"ur Mathematik, Technische Universit\"at Berlin, and Institute of Mathematics, Shandong University, Jinan, China}
\email{huilinzhang2014@gmail.com}

\subjclass[2010]{Primary 60H99; Secondary 60H10}

\keywords{Rough paths with jumps, It\^o stochastic and rough differential equations with jumps, general semimartingales, limit theorems}

\begin{abstract}
We develop the rough path counterpart of It\^o stochastic integration and - differential equations driven by general semimartingales. This significantly enlarges the classes of (It\^o / forward) stochastic differential equations treatable with pathwise methods. A number of applications are discussed.
\end{abstract}

\maketitle

\tableofcontents

\newcommand{\HC}{\mathcal {H}_\mathbb{C}}
\newcommand{\OHEC}{(\Omega,\,\HC,\,\EC)}
\newcommand{\E}{\mathbb{E}}
\newcommand{\R}{\mathbb{R}}
\newcommand{\OHE}{(\Omega,\,\mathcal{H},\,\E)}
\newcommand{\C}{\mathbb{C}}
\newcommand{\X}{\mathbb{X}}
\newcommand{\BX}{\mathbf{X}}
\newcommand{\K}{\mathbb{K}}
\newcommand{\oc}{\mathcal {C}}
\newcommand{\FC}{\mathscr{C}}
\newcommand{\B}{\mathbb{B}}
\newcommand{\BB}{\mathbf{B}}
\newcommand{\op}{\mathcal{P}}
\newcommand{\FD}{\mathscr{D}}
\newcommand{\oq}{\mathcal{Q}}
\newcommand{\oor}{\mathcal {R}}
\newcommand{\hc}{\hat{c}}
\newcommand{\BI}{\mathbf{1}}
\newcommand{\BZ}{\mathbf{Z}}
\newcommand{\I}{\mathcal{I}}
\newcommand{\Z}{\mathbb{Z}}
\newcommand{\bx}{\mathbf{x}}
\newcommand{\bz}{\mathbf{z}}
\newcommand{\s}{\mathbb{S}}
\newcommand{\A}{\mathbb{A}}
\newcommand{\LL}{\mathbb{L}}
\newcommand{\tiop}{\tilde{\mathcal{P}}}
\newcommand{\BBI}{\mathbf{I}}
\newcommand{\BY}{\mathbf{Y}}
\newcommand{\vertiii}[1]{{\left\vert\kern-0.25ex\left\vert\kern-0.25ex\left\vert #1
    \right\vert\kern-0.25ex\right\vert\kern-0.25ex\right\vert}}

\newcommand{\ru}[2]{\begin{tikzpicture}[scale=0.15,baseline=0.1cm]
        \node at (0,0)  [dot,label= {[label distance=-0.2em]below: \scriptsize  $ #1 $} ] (root) {};
         \node at (0,2)  [dot,label={[label distance=-0.2em]above: \scriptsize  $ #2 $}]  (up) {};
            \draw[kernel1] (root) to
     node {}  (up);
     \end{tikzpicture}}

\newcommand{\rlr}[3]{\begin{tikzpicture}[scale=0.15,baseline=0.1cm]
        \node at (0,0)  [dot,label= {[label distance=-0.2em]below: \scriptsize  $ #1 $} ] (root) {};
         \node at (1,2)  [dot,label={[label distance=-0.2em]above: \scriptsize  $ #2 $}]  (right) {};
         \node at (-1,2)  [dot,label={[label distance=-0.2em]above: \scriptsize  $ #3 $} ] (left) {};
            \draw[kernel1] (right) to
     node [sloped,below] {\small }     (root); \draw[kernel1] (left) to
     node [sloped,below] {\small }     (root);
     \end{tikzpicture}}

\newcommand{\ruu}{\begin{tikzpicture}[scale=0.25,baseline=0.1cm]
        \node at (0,-1)  [dot,label= {[label distance=-0.2em]right: \scriptsize  $ i_1 $} ] (root) {};
         \node at (0,1)  [dot,label={[label distance=-0.2em]right: \scriptsize  $ i_2 $}]  (u) {};
         \node at (0,3)  [dot,label={[label distance=-0.2em]right: \scriptsize  $ i_{n} $} ] (uu) {};
            \draw[kernel1] (u) to
     node [below] {\small }     (root); \draw[dotted,thick] (u) to
     node [below] {\small }     (uu);
     \end{tikzpicture}}


\setcounter{section}{-1}
\section{Introduction and notation}

In many areas of engineering, finance and mathematics 
one encounters equations of the form
\begin{equation} \label{equ:ODEintro}
dy_t = f(y_t) dx_t \ ,
\end{equation}
where $x$ is a multi-dimensional driving signal, $f$ a collection of nice driving vector fields.\footnote{At the price of replacing $f$ by $(f_0,f)$ and $x$ by $(t,x)$ formulation (\ref{equ:ODEintro}) immediately allows for a drift term.}
 For $x \in C^1$, this can be written as time-inhomogenous ODE
of the form $\dot y (t) = f(y(t)) \dot x(t)$ and there is no ambiguity in its interpretation. This is still the case for rectifiable drivers,
$$ x \in C^{1-var} \equiv C \cap V^1, $$
 i.e. continuous paths of locally finite $1$-variation, say on $[0,T]$, in which case there is perfect meaning to the (Riemann-Stieltjes) integral equation
\begin{equation} \label{equ:intODEintro}
   y_t = y(0) + \int_0^t f(y_s) dx_s \ .
\end{equation}
For $x \notin C^{1-var}$, it is helpful to distinguish between low regularity (e.g. $x \in C^{p-var}$ for some $p>1$) and lack of continuity (e.g. $x \in V^1$) before tackling the case of general (possibly discontinuous) driver with finite $p$-variation, i.e. $$x \in V^p, $$ for arbitrary $p< \infty$.
The first case, $x \in C^{p-var}$ includes the important class of {\it continuous semimartingales} (with $p>2$), and here already one one encounters a fundamental ambiguity in the interpretation of
(\ref{equ:ODEintro}), with It\^o- and Stratonovich interpretation being the most popular choices. Without semimartingale (or in fact: any probabilistic) structure, {\it rough paths}  \cite{Lyo98} provide a
satisfactory substitute, that deals with all $p < \infty$, provided $x \in \mathbf{C}_g^p$, the space of {\it continuous, geometric} $p$-rough paths: any ambiguity is then resolved by the additional information
contained in $x$, by the very nature of a rough path. (Probability is still used to construct a - random - rough path over some given stochastic process, see e.g. \cite{FV10}.)

\bigskip

A different phenomena, purely deterministic (and unrelated to rough path considerations), arises when one drops continuity of the driving signal, even in the case of finite $1$-variation. To wit, take $x \in D^1 \equiv D \cap V^1$, where $D$ denotes the space of c\`adl\`ag paths (so that $dx$ can be interpreted as Lebesgue-Stieltjes measure). We may consider the possible integral equations
\begin{eqnarray}
                   y_t = y_0+ \int_0^t f(y_s) dx_s \ ,  \label{equ:sillySDE} \\
                   y_t = y_0 + \int_0^t f(y^-_s) dx_s \ , \label{equ:fwdSDE} \\
		  y_t = y_0 + \int_0^t f(y) \diamond dx_s \ , \label{equ:geoSDE}
\end{eqnarray}
with the meaning of (\ref{equ:geoSDE}) best desribed by words: ``replace every jump $\Delta x_s = x_s - x_{s-}$ by a straight line over some artificial extra time interval and solve the resulting continuous differential equations; finally disregard all extra time for the solution'' (this is {\it Marcus' canoncial solution} in the finite variation context; it has the advantage of preserving the chain-rule.) The first equation (\ref{equ:sillySDE}) deserves no further attention, as in general it does not admit solutions (a simple counter-example is given in Section \ref{sec:BVODEs}). So we are left with
(\ref{equ:fwdSDE}), (\ref{equ:geoSDE}), solutions to which have been called \cite{Wil01} {\it forward} and {\it geometric} respectively.\footnote{In absence of jumps forward and geometric solutions coincide.} We also note that the (to stochastic analysts familiar) structure of (\ref{equ:fwdSDE}), with c\`agl\`ad integrand and c\`adl\`ag integrator, is by no means necessary and it is only consequent in a purely deterministic development to discard of this assumption: it is enough to have $x \in V^1$ to study a generalization of (\ref{equ:fwdSDE}), based on a {\it $\ell$eft-point} Riemann-Stieltjes integral, which we write as
\begin{equation}
     y_t = y_0 + \int_0^t f(y_s)^\ell dx_s \ . \label{equ:fwdellSDE} \\
\end{equation}

One can go beyond $1$-variation regularity by suitable rough path considerations. At least conceptually, this is easier in the geometric case, driven by {\it c\`adl\`ag, geometric} $p$-rough paths
$$
       x \in \mathbf{D}_g^p \equiv \mathbf{D}_g \cap \mathbf{V}_g^p \ ,
$$
for preservation of the chain rule implies that the algebraic setting of geometric rough path is still suitable. Without going in full detail,
$\mathbf{D}_g^ p $ 
here is the space of c\`adl\`ag path on $[0,T]$ with values in the step-$[p]$ niltpotent group (over $\R^d)$, $\mathbf{V}_g^p$ is the space of path in the same group, of finite $p$-variation w.r.t. the Carnot-Caratheodory distance \cite{FV10,FH14}). We also note that ``c\`adl\`ag'' is little more than a convention here, since the solution theory - in the spirit of Marcus' canonical solutions - effectively relies on the continuous theory. Evenso, it is a subtle matter to identify the correct rough path metrics of such geometric solutions (as was recently seen \cite{CF17x}, a $p$-variation rough path variants of the Skorohod M1 metric) and its applicability to general semimartingale driver, in the spirit of \cite{KPP95}.  A fairly complete study of this ``geometric theory'' is carried out in the compagnion paper \cite{CF17x}.

\bigskip

The ``forward theory'' on the other hand, topic of the present paper, comes with differents challenges. First, the lack of a chain-rule makes it impossible to work in the geometric rough path setting. For $p \in [2,3)$ this just means that the state space can be identified with $\R^d \oplus ( \R^{d} )^{\otimes 2} $, dropping the well-known geometricity (a.k.a. ``first order calculus'') condition \cite[(2.5)]{FH14}. For general $p < \infty$, however, we need to use the full Hopf algebraic formalism of {\it branched rough paths}. 
That is,
$$
      x \in \mathbf{V}^p = \mathbf{V}^p ([0,T],G_{[p]}(\MH^*)) \ ,
$$
the space of path on $[0,T]$ with values in the step-$[p]$ Butcher group, again with a suitable $p$-variation condition. For example, as pointed out in the works of Gubinelli and Hairer--Kelly, \cite{Gub10, HK15}, the seemingly suitable state space $\R^d \oplus ( \R^{d} )^{\otimes 2} \oplus ( \R^{d} )^{\otimes 3}$ is {\it not} sufficient to understand level-$3$ rough paths, interacting with non-linear differential equations, in absence of a chain-rule. The correct state space, in the general case, was understood in \cite{Gub10}, together with a description in terms of trees, rather then  words ($\leftrightarrow$ ``linear trees'', tensors) over $\{ 1, ... , d\}$, hence the terminology ``branched''.
(We have organised Sections 2-4 in a way that allows to bypass the complexity of branched rough paths, simply by skipping Section 4, at the price of accepting the roughness restriction $p<3$.)
Secondly, the analysis, starting with rough integration, has to be carried out from first principles, i.e. cannot be reduced to the well-developed continuous case. To this end, we develop some variants of the ``sewing lemma'' (e.g. \cite[Ch. 4]{FH14} and the references therein) with non-continuous a.k.a. {\it non-regular controls}  - which in turn offers (pathwise) expansions of integrals and solutions to (possibly stochastic) differential equations. 
This also implies that our notion of solution delicately depends on the fine-structure of the jumps (left-right limits exist everywhere). On a technical level, the presence of (big) jumps puts an a priori stop to
the interval on which Picard iteration can be used to construct maximal solutions, this is dealt with by treating (finitely many) big jumps by hand. Our analysis comes with stability estimates in terms
 of $p$-variation rough path metrics, which readily can be framed in (purely deteministic) limit theorems with respect to $p$-variation rough path variants of the Skorohod J1 metric. An immediate application then concerns discrete approximations (higher order Euler or Milstein-type schemes) for rough differential equations, underpinning the numerical nature of the entire theory, as previously done by Davie
 \cite{Dav07}, $2 \le p < 3$ and then \cite{FV10}, any $p< \infty$. As will be seen (Section \ref{sec:discrete}) all this follows immediately from the results of this paper,
 and then of course in greater generality, allowing for jumps in the limiting rough paths.

 \bigskip

The final part of this paper (Section \ref{sec:SDE}) is devoted to the case of {\it random} $p$-rough paths. General (c\`adl\`ag) semimartingales, with their It\^o-lift, are seen to give rise to such random rough paths, and important estimates
 such as the BDG inequality remain valid for the homogenous $p$-rough path norm. This allows us to finally reconcile the ``general'' theory of (c\`adl\`ag) semimartingales and their It\^o theory, with a matching
 theory of general (and then c\`adl\`ag) rough path. The abstract main result of that section is Theorem \ref{r.rde}.  Amongst others, we shall see that many classical stability results (``limit theorems'') in the general It\^o semimartingale theory, notably the classical UCV/UT conditions  (Kurtz-Protter \cite{KP91}, \cite{KP96} and Jakubowski, Memin and Pages \cite{JMP89}) thus admit a perfect explanation from a rough path point of view, see also \cite{CL05}. In essence, UCV/UT implies tightness of $p$-variation rough path norms, which is the key condition
 (in a random rough path setting, but without any implicit semimartingale assumptions) for limit theorem of random rough differential equations. This is especically interesting in the case of It\^o SDEs (interpreted
 as random RDEs) driven by a convergent sequence of semimartingales which does not satisfy UCV/UT. This is a typical situation in homogenization theory, see e.g. the works of Kelly and Melbourn
 \cite{Kel16, KM16, KM17}.

 \bigskip

 Checking the desired rough path $p$-variation tightness in Theorem \ref{r.rde} can be non-trivial. In Section \ref{sec:cat} we thus present a catalogue of the relevant (general) techniques we are aware of. Amongst others,
  we present a Besov-type criterion for discrete rough path approximation that presents a definite improvement of the main result in \cite{Kel16}, namely a ``$2+$ instead of $6+$ moments assumption'' relative
   to the H\"older scale used in Kelly \cite{Kel16} (cf. also Erhard-Hairer \cite{EH17x} applied in the rough path setting). This is almost optimal, as seen from martingale examples in which the UCV/UT conditions
   exhibits ``$2$ moments'' as optimal. (It would be desirable, though this is by no means the purpose of this paper, to build a similar Besov theory for discretizations of regularity structures in the sense of
   \cite{EH17x}.) We note however that Besov spaces are ill-suited to deal with (limiting) differential equations with jumps, a point also made in \cite{PT16}. In particular, the results of Sections \ref{sec:RDE} and \ref{sec:bRP}) are out of reach of regularity structures and its recent Besov ramifications \cite{Hai14, HL17}.

 \bigskip
 {\bf Further applications.}
 The results of this paper seem to have a variety of applications beyond what has been said so far.  A simple, but potentially far-reaching, remark is that simple recursions, both $x$ and $y$ are now vector-valued sequences,
 $$
       y_{i+1} = y_i + f(y_i) (x_{i+1} - x_i)
 $$
 are within the scope of our theory. To wit, such a system can first be viewed as differential equation of type (\ref{equ:fwdSDE}), by identifiying $(x_i$) with a piecewise c\`adl\`ag path with jumps $x_{i+1} - x_i$,
 but then also as rough differential equation by a canonical lifting procedure of the sequence $(x_i)$ to a (branched) $p$-rough path, any $p<\infty$. The interest in this is to obtain estimates, indexed by the choice of $p$ (over which can later be optimised!), on such recursions that do not depend on the {\it depth}, i.e. the cardinality of the index set $\{i : i = 1,...,N \}$. (In contrast, the a priori smooth dependence of $y_N$ as function of $x$ will be very poorly behaved for large $N$.) We note that this structure is of interest in {\it deep learning networks}, explored in forthcoming work.
 If one replaces the above explicit recursion with certain ODEs, one is back in the realm of continuous rough path theory, see \cite{FHL16} for a related discussion, also departing from discrete signals.

 \bigskip
Some readers will be especially interested in how the present results improve (beyond what was said above) on existing stochastic analysis. Of course, many of the ``usual rough path comments'' (see e.g.
\cite{FH14} and Lyons' original work \cite{Lyo98}) apply: with probability one we can simultaneously solve all differential equations over a given driving  semimartingale. (The content of this remark is in the fact
that there are uncountably many different differential equations.) This has notably consequences for the construction of (stochastic) flows, support theorems (see e.g. \cite{Sim03} for an application of this idea for
Levy processes in the Young regime), and situations with anticipating starting points or coefficient vector fields. Then of course, not all interesting stochastic dynamical systems are
semimartingales, e.g. rough volatility as example of a nonlinear systems forced by a Gaussian random processes. The construction of (non-semimartingale) symmetric Markov process with jumps is also
classical. In both cases, such continuous non-semimartingale processes have been successfully used in a rough path framework (e.g. \cite[Ch. 15, 16]{FV10} and the references therein) - with the present work it
is possible to extend this systematically to situations with jumps. Numerical algorithms, with adaptive step-size, are another application area. Unlike the It\^o  (or It\^o- F\"ollmer) integral, the rough integration
allows for any choice of partitions (anticipating or not) without ambiguity of the resulting limit. Next, stochastic filtering - a longstanding stability question (resolved in \cite{CDFO13}) can now be studied in the
case of, say, observations that are also subjected to noise with jumps. It can also be expected that our considerations, one way or another, lead the way to a pathwise understanding of certain stochastic partial
differential equations with jump noise. Yet another field of potential applications is mathematical finance. Recall that (It\^o) stochastic integrals of the form $\pi = \int Y d S$ carry the interpretation of P \& L (profit
and loss) process, given a (previsible) trading strategy $Y$ and asset price processes which constitute the vector-valued semimartingale $S$. In the case of discrete-time financial markets, there is a perfect
pathwise interpretation of $\pi$ which is lost upon passage to continuous time. While restricted to ``Markovian'' investment strategies $Y = f(S^-)$ and some slight generalization thereof (see Section
\ref{sec:Rint}) we have now have a toolbox that provides pathwise meaning to $\pi$, and also can deal simultaneously with (uncountable) families of investment strategies, given some underlying $S$. As such,
they may constitute natural mathematical tools in the field of robust finance \cite{PP16}.

\bigskip

\bigskip

{\bf Comments on literature. }
Concerning the history of rough differential equations with jumps: the Young case $p \in [1,2)$ was already studied in \cite{Wil01}. Rough integration against c\`adl\`ag rough path, $p \in [2,3)$, was introduced in \cite{FS17}, as was the notion of {\it c\`adl\`ag geometric $p$-rough path}, any $p<\infty$. In the latter setting, limit theorems and applicability to semimartingales (canonical solutions in the sense of Marcus) was obtained in \cite{CF17x}. This left open, topic of the present paper,  a ``forward'' theory of rough differential equations, capable of recovering and extending the known stability theory of It\^o stochastic differential equations driven by c\`adl\`ag semimartingales.

For the sake of completeness, we note that {\it continuous} semimartingales as {\it continuous} rough path were studied in \cite{CL05,FV08b, FV10}. In \cite{FHL16} the authors consider piecewise linear, axis-directed approximations of continuous semimartingale and so obtain an It\^o SDE in the Wong-Zakai limit; in \cite{LY16} It\^o SDEs are obtained as averaged Stratonovich solutions. (No such construction works in presence of jumps. In particular, this does not allow to go from \cite{CF17x} to any of the results of this paper.)

The literature on It\^o stochastic integration against - and differential equations driven by general semimartingale is vast, see e.g. \cite{JS03, Pro05} and the references in there. (Marcus canonical solutions, which provided the motivation for \cite{CF17x}, are discussed e.g. in \cite{KPP95, App04} but play less a role for this work.)

\subsection{Notations} \label{sec:not}

A {\it partition} of $[0,T]$ is of the form $\op = \{ 0 = t_0 < t_1 < ... < t_N = T \}$, $N \in \mathbb{N}$. We write both $[t_{i-1},t_i] \in \op$ and $t_i \in \op$. A {\it path} is a function from $[0,T]$ (or any other interval clear from the context) into a metric space  $(E,d)$. Such a path is called {\it regulated} if both left and right limits exist for all points and we write $X \in V^\infty \equiv V^\infty( [0,T],E)$ accordingly. The class of right-continuous paths with left-limits (a.k.a. c\`adl\`ag paths) is denoted by $D$, the class of continous paths by $C$; of course $C \subset D \subset V^\infty$.
For $p \in [1,\infty)$ we say $X \in V^p$, in words: {\it has finite $p$-variation}, if
\begin{equation} \label{equ:metricPnorm}
 \| X \|_{p,[0,T]} := \left( \sup_{\op} \sum_{[s,t]\in \op} d(X_s,X_t)^p   \right)^{\frac 1 p} < \infty,
\end{equation}
where $\sup$ is taken over all partitions $\op$ of $[0,T]$. (Here and below, $[0,T]$ is readily replaced by an arbitrary, not necessarily closed, interval.) One readily checks that $V^1 \subset V^p \subset V^\infty$. We also write
$$
      \| X \|_{\infty,[0,T]} := Osc(X;[0,T]) := \sup_{0 \le s \le t \le T}  d(X_s,X_t) \ \ \    \text{ and } \| X \|_{\sup,[0,T]} := \sup_{0 \le t \le T} d(X_0,X_t) \ .
$$
 A {\it control} is a function $\omega = \omega(s,t)$ from $\{ 0 \le s \le t \le T \}$ into $[0,\infty)$
 which is null on the diagonal and {\it super-additive} in the sense $\omega(s,t)+\omega(t,u) \le \omega(s,u)$.
 A control is {\it regular} if it is continuous (at least near the diagonal). Every $X \in V^p$ gives rise to a natural control
 $\omega_{X,p}(s,t) := \| X \|^p_{p,[s,t]}$. It is regular
if and only if $X$ is also continuous, i.e. $X \in C^{p-var} = C \cap V^p$, see e.g. \cite{FV10}.

We also define $p$-variation, any $p \in (0,\infty)$, of functions $\Xi$ from $\{ 0 \le s \le t \le T \}$ into a normed space (in general, $\delta\Xi_{s,u,t}:=\Xi_{s,t}-\Xi_{s,u}-\Xi_{u,t} \ne 0$)
\begin{equation}  \label{equ:doubleRPnorm}
      \| \Xi \|_{p,[0,T]} := \left( \sup_{\op} \sum_{[s,t]\in \op} | \Xi_{s,t}|^{p}   \right)^{\frac 1 p} < \infty,
\end{equation}
and further set
$$
      \| \Xi \|_{\infty,[0,T]} := \sup_{0 \le s \le t \le T} | \Xi_{s,t}| \ \ \  \text{ and } \| \Xi \|_{\sup,[0,T]} := \sup_{0 \le t \le T} | \Xi_{0,t}| \ .
$$
A (level-$2$) {\it $p$-rough path} (over $\R^d)$ is a path $\BX=\BX_{t}$ with values and increments $\BX_{s,t}:=\BX_s^{-1}\star \BX_t := (X_{s,t}, \X_{s,t})$ in $\R^d \oplus \R^{d \times d} =:G$, a (Lie) group equipped with multiplication $(a,M) \star (b,N) = (a + b, M + a \otimes b + N)$, inverse $(a,M)^{-1}:= (-a, -M + a\otimes a),$ and identity $(0,0),$ of finite $p$-variation condition, $p \in [2,3)$, either in the sense (``{\it homogenuous rough path norm}'')
\begin{equation} \label{equ:tripleenorm}
                   ||| \BX |||_{p,[0,T]} := \| X \|_{p,[0,T]} +   \| \X \|^{1/2}_{p/2,[0,T]} < \infty \ .
\end{equation}
or, equivalently, in terms of the {\it inhomogenous rough path norm}\footnote{It is an artefact of the non-geometric level-$2$ case that $G$ equals (as set) $\R^d \oplus \R^{d \times d}$. In general, $G$ is a non-linear group, only embedded in a linear space. Definitions (\ref{equ:tripleenorm}),(\ref{equ:doublenorm}) provide the correct multiscale view on rough paths. The mild clash of notations, (\ref{equ:doublenorm}) vs. (\ref{equ:doubleRPnorm}), will not cause any confusion.}
\begin{equation} \label{equ:doublenorm}
                   \| \BX \|_{p,[0,T]} := \| X \|_{p,[0,T]} +   \| \X \|_{p/2,[0,T]} < \infty \ .
\end{equation}
Many later estimates will be expressed in term of the {\it (inhomogenous) rough path distance}
\begin{equation} \label{equ:doublenormdist}
                   \| \BX ; \tilde \BX  \|_{p,[0,T]} := \| X - \tilde X \|_{p,[0,T]} +   \| \X -\tilde \X \|_{p/2,[0,T]}  \ .
\end{equation}
We also set
$$
                    \| \BX \|_{\infty,[0,T]} := \| X \|_{\infty,[0,T]} +   \| \X \|_{\infty,[0,T]}  \ \ \  \text{ and } \ \ \ \| \BX \|_{\sup,[0,T]} := \| X \|_{\sup,[0,T]} +   \| \X \|_{\sup,[0,T]} \ .
$$
Furthermore, equip the group $G$ with the following (continuous) mappings induced by power series,
\begin{eqnarray*}
\log (a,M):= (a, M-\frac12 a\otimes a), \ \ \  \exp (a,M) = (a, M+ \frac12 a\otimes a ),
\end{eqnarray*}
and one can check that $G$ is indeed a Lie group and in fact a homogenous group. 
It can be seen \cite{HS90} (see also appendix) that $G$ admits a left-invariant metric such that $d((0,0);(a,M)) \sim |a| + |M|^{1/2}$, under which $(G,d)$ is a Polish space, and
such that \footnote{It would be consistent to denote the right-hand side of (\ref{equ:triplenorm}) by $ \| \BX \|_{p,[0,T]}$, in the sense of
(\ref{equ:metricPnorm}) which however clashes with (\ref{equ:doublenorm}). Again, this will cause no confusion. In particular $ \| \BX ; \BY \|_{p,[0,T]}  = \| \BX - \BY \|_{p,[0,T]} $ has no correspondence in the metric setting of (\ref{equ:metricPnorm}).}

\begin{equation} \label{equ:triplenorm}
                   ||| \BX |||_{p,[0,T]} \asymp \left( \sup_{\op} \sum_{[s,t]\in \op} d(\BX_s,\BX_t)^p   \right)^{\frac 1 p} \ . 
\end{equation}
Indeed, for any partition $\op,$ one has
\begin{eqnarray*}
\sum_\op (|X_{s,t}|^p + |\X_{s,t}|^\frac p 2) &=& \sum_\op (|X_{0,t}-X_{0,s}|^p + |\X_{0,t}-\X_{0,s}- X_{0,s}\otimes X_{s,t}|^{\frac p 2})\\
                                              &\asymp & \sum_\op \left(d(\BX_{0,s}, \BX_{0,t})\right)^p,
\end{eqnarray*}
which implies the consistency of our rough path metric with the homogeneous metric $d$ on the Lie group.
We further introduce the {\it left-} and {\it right jumps} of a rough path by setting (both limits exist)
\begin{eqnarray*}
                   \Delta^-_t \BX & := &  \BX_{t-,t} := \lim_{s \to t} \BX_{s,t} \ , \\
                   \Delta^+_t \BX & := &  \BX_{t,t+} := \lim_{s \to t} \BX_{t,s} \ .
\end{eqnarray*}

Everything here extends to the case of (branched) rough paths of arbitrary roughness $p < \infty$. Details (including notation) are left to Section \ref{sec:bRP} and Appendix \ref{app:leftinvmetric}.

\bigskip

\noindent\textbf{Acknowledgment:}
P.K.F. acknowledges partial support from the  ERC, CoG-683164, and DFG research unit FOR2402. H.Z. thanks the Institut f\"ur Mathematik, TU Berlin, for its hospitality. Both authors thank Ilya Chevyrev for discussions and  feedback on an earlier version.

\section{A miniature in the bounded variation case} \label{sec:BV}

In order to detangle the effects of jumps with other rough path considerations, it is useful to start with the bounded variation case. A basic object here is the ``left-point'' Riemann-Stieltjes integral,
\begin{equation} \label{into:BVmin}
     \int_0^T y^\ell_r dx_r :=     \lim_{| \op | \to 0}  \sum_{(u,v)\in \op} y_u x_{u,v} \ ,
\end{equation}
as  frequently studied in the context of semimartingales, usually under additional assumptions (e.g. $x$ c\`adl\`ag) required from a martingale perspective. With an outlook to a fully deterministic rough theory given later in the paper, we make no additional regularity assumption, do however point out the (analytic) consequence of such assumptions in terms of a stronger convergence. A novelty with regard to classical discussions of BV integration, which appears useful in its own right, is the use of mild controls which allows a formalism (sewing lemma etc, cf. \cite{Gub04, FH14}) similar to what is now mainstream in rough integration.

\subsection{Mild sewing lemma}

The following definition (cf. \cite{DN98}) is crucial to understand the convergence of Riemann-Stieltjes sums as in (\ref{into:BVmin}) above. The following definition takes an abstract view on this,
of course $\Xi_{u,v}  \equiv y_u x_{u,v}$ is a good (first) example to have in mind.


\begin{defi}(\textbf{MRS v.s. RRS})\label{MRS,RRS}
For a partition $\op$ of $[0,T]$ and any $[u,v]\in \op,$ $\Xi_{u,v}  $(e.g. $y_u x_{u,v}$) takes value in $\R^d.$
We call the Riemann sum $ \sum_{[u,v]\in \op }\Xi_{u,v} $ converges to $K$ in the sense of

\begin{itemize}

\item Mesh Riemann-Stieltjes (MRS) sense(classical one): if for any $\vep>0,$ there exists $\delta>0,$ such that any $ |\op|<\delta, $ one has $|\sum_{[u,v]\in \op }\Xi_{u,v}-K|<\vep.$

\item Refinement Riemann-Stieltjes (RRS) sense (Hildebrandt \cite{H38},\cite{H63}): if for any $\vep>0,$ there exists $\op_\vep$ such that for any refinement $ \op \supset \op_\vep, $ one has $|\sum_{[u,v]\in \op }\Xi_{u,v}-K|<\vep.$

\end{itemize}
We denote this limit $K$ by (MRS resp. RRS) $\mathcal{I}\Xi_{0,T}$.

\end{defi}

\begin{rem}
MRS and RRS limits, if they exist, are unique. MRS implies RRS convergence, with the same limit. Last but not least, one has additivity in the sense that, for any $s <  u < t$ in $[0,T]$,
$$
    \I \Xi_{s,t} = \I \Xi_{s,u} + \I \Xi_{u,t}.
$$
(In the RRS case, the last equality is obtained by insisting that $u \in \op_\vep$.)

\end{rem}

Recall that a control $\omega = \omega(s,t)$ is null on the diagonal, and superadditive in the sense of
$
\omega(s,u)+\omega(u,t)\leq \omega(s,t)$ for all $s<u<t$.

\begin{defi}(\textbf{mild control})
We call a function $\sigma(s,t)$ from the simplex $\{ (s,t):0\leq s \leq t \leq T\}$ to nonnegative real numbers a mild control, if it is increasing in $t,$ decreasing in $s,$ null on the diagonal, and for any $\delta>0,$ there exist only finite $u\in[0,T]$ such that
$$
\sigma(u-,u+):= \lim_{t\downarrow u , s \uparrow u} \sigma(s,t) > \delta.
$$
We call $\sigma(s,t)$ a mild control for some path $x$ if $d(x_s, x_t) \leq \sigma(s,t)$ for any $s,t\in [0,T].$ I
\end{defi}

\begin{lem}\label{controls are mild c.}
(i) Every control $\omega$ is a mild control. More generally, for any function $f: [0,\infty) \to [0,\infty)$ which is strictly increasing, $(s,t) \mapsto f(\omega(s,t))$ is a mild control. (ii) Let $y$ be a regulated path,
(i.e., left and right limit exist at any point) with values in a metric space $(E,d)$. Then
$$
\sigma(s,t):= Osc ( {y;[s,t] )} \equiv \sup_{u,v\in[s,t]} d(y_u, y_v )
$$
defines a mild control.
\end{lem}
\begin{proof}
For the first claim, suppose $\omega$ is a control. The monotonicity is obvious. For any $\delta>0,$ if there are infinitely countable $\{u_i\}_i \subseteq [0,T],$ indeed one could have $\omega(0,T)> K$ for any $K>0.$ For the second claim, one has
\begin{eqnarray*}
\sigma(l-,l+) &:=& \lim_{t\downarrow l , s \uparrow l} \sup_{u,v\in[s,t]} d(y_u, y_v )\\
              &\leq & \lim_{t\downarrow l , s \uparrow l} \sup_{u,v\in[s,t]} (d(y_u,y_l)+ d(y_l,y_v) )\\
              &\leq & 2\left( d(y_{l-},y_l)+ d(y_{l+},y_l) \right).
\end{eqnarray*}
Since $y$ is regulated, for any $\delta>0,$ there are only finite $l\in[0,T]$ such that $d(y_{l-},y_l)+ d(y_{l+},y_l)> \delta,$ which completes the proof.

\end{proof}

%
%


\begin{lem}\label{mild control}
Suppose $\sigma(s,t)$ is a mild control. Then for any $\vep>0,$ there exists a partition $\op,$ such that for any $(s,t)\in \op,$ one has
$$
\sigma(s+,t-)< \vep.
$$
In particular, if $\sigma(s,t)$ is left-continuous in $t,$ then $\sigma(s+,t)< \vep.$

\end{lem}

\begin{proof}
Fix a $\vep>0,$ by the definition of mild controls, there exists a finite increasing set $\{ s_i\}_{i=1}^N,$ such that $\sigma(s_i-,s_{i}+)\geq \vep, $ and for any other $s\in [0,T]\setminus\{ s_i\}_{i=1}^N,$ $\sigma(s-,s+)< \vep.$ For any $s\in [0,T]\setminus\{ s_i\}_{i=1}^N,$ there exists a $\delta>0,$ such that $\sigma(s-\delta,s+\delta)< \vep,$ and for each $s_i,$ by monotonicity, there exists $\delta_i < \min_{i=1}^{N-1} |s_{i+1}-s_i|,$ such that $\sigma(s_i+, s_i+\delta_i)<\vep,$ $\sigma(s_{i}-\delta_i,s_i-)< \vep.$ Indeed, for any $\vep_i<\min_{i=1}^{N-1} |s_{i+1}-s_i|$, one has $\sigma(s_i+\vep_i, (s_i+\vep_i)+)< \vep,$ $\sigma((s_i-\vep_i)-  , s_i-\vep_i )$ which implies the existence of such $\delta_i.$ Then one obtains a open cover $\{(s-\delta,s+\delta)|s\in[0,T]\setminus\{ s_i\}_{i=1}^N \}\bigcup \{(s_i-\delta_i,s_i+\delta_i)\}_{i=1}^N$ for $[0,T],$ so there exists a finite cover. Take all endpoints of this cover and $\{ s_i\}_{i=1}^N$ to make a partition, denoted as $\op:=\{0=u_0<u_1<,...,<u_M=T\}.$ Then it follows that
\[
\left\{
\begin{array}{ll}
\sigma(u_j,u_{j+1})< \vep & \text{if } u_j,u_{j+1} \not\in \{ s_i\}_{i=1}^N\\
\sigma(u_j+,u_{j+1})   < \vep & \text{if } u_j \in \{ s_i\}_{i=1}^N\\
\sigma(u_j,u_{j+1}-) < \vep & \text{if } u_{j+1} \in \{ s_i\}_{i=1}^N,
\end{array}\right.
\]
which implies our result by monotonicity.

\end{proof}

\begin{lem}\label{regular path}
Suppose that $x$ from $[0,T]$ to $(E,d)$ is regulated. Then for any $\vep>0,$ there exists a partition $\op$, such that for any interval $(s,t)\in \op,$
$$
Osc(x,(s,t)) \equiv \sup_{u,v\in(s,t)}d(x_u,x_v) < \vep.
$$
In particular, if $x$ is right-continuous, $Osc(x,[s,t))< \vep.$
\end{lem}

\begin{proof}
Let $\sigma(s,t):= \sup_{u,v\in[s,t]} d(x_u,x_v),$ and according to Lemma \ref{controls are mild c.}, it is a mild control. Then this is a corollary of the Lemma \ref{mild control}.

\end{proof}




\begin{thm}(\textbf{mild sewing})\label{bv sewing}
Suppose $\Xi_{s,t}$ is a mapping from simplex $\{ (s,t):0\leq s < t\leq T \}$ to a Banach space $(E,\|\cdot\|) .$
Let $\delta\Xi_{s,u,t}:=\Xi_{s,t}-\Xi_{s,u}-\Xi_{u,t}.$ Assume $\delta\Xi$ satisfies
$$
\|\delta\Xi_{s,u,t}\| \leq \sigma(s,u) \omega(u,t),
$$
where $\sigma$ is a mild control and $\omega$ is a control. Then the following limit exists in the RRS sense,
$$
\mathcal{I}\Xi_{0,T}:=RRS-\lim_{|\mathcal{P}|\rightarrow 0} \sum_{(u,v)\in \mathcal{P}} \Xi_{u,v},
$$
and one has the following local estimate:
$$
\| \mathcal{I}\Xi_{s,t}-\Xi_{s,t}\|\leq   \sigma(s,t-) \omega(s+,t),
$$
Furthermore, if $\omega(s,t)$ is right-continuous in $s,$ i.e.
$$
\omega(s+,t):=\lim_{u \downarrow s} \omega(u,t)= \omega(s,t),\ \ \forall \ 0\leq s < t \leq T,
$$
or $\sigma(s,t)$ is left-continuous in $t,$ the convergence holds in the MRS sense.

\end{thm}

\begin{proof}
For a fixed $\vep>0,$ according to the proof of Lemma \ref{mild control}, there exists a partition $\op_0=\{0=s_0<s_1<,...,<s_N=T\},$ such that for any $(s_i,s_{i+1}),$ one has
$$
\sigma(s_i +,s_{i+1}-)< \vep.
$$
For each $s_i,$ since $\omega$ is also a mild control, there exists a $t_i\in (s_i,s_{i+1}),$ such that
$$
\omega(s_i+,t_i)< \frac{\vep}{N\sigma(0,T)}.
$$
Let $\op_\vep:= \{s_i\}_{i=0}^{N} \bigcup \{t_i\}_{i=1}^{N-1}.$ Then for any refinement $\op$ of $\op_\vep,$ and for any $(s,t) \in \op_\vep,$ denote $\op|_{[s,t]}=\{s=\tau_0< \tau_1<,...,< \tau_n = t\}.$ Without loss of generality, assume $n\geq 2.$ If $s  =t_j, $ for some $j=1,...,N-1,$ by inserting terms like $\sum_{\op\setminus\{\tau_1,\tau_2,...,\tau_k\}|_{[s,t]}}\Xi_{u,v},$ one obtains,
\begin{eqnarray*}
|\sum_{\op|_{[s,t]}}\Xi_{u,v} - \Xi_{s,t}| &\leq& \sum_{k=1}^{n-1} |\delta\Xi_{s,\tau_k,\tau_{k+1}}| \leq  \sum_{k=1}^{n-1} \sigma(s,\tau_k) \omega(\tau_k,\tau_{k+1})\\
& \leq& \sigma(s_j+,s_{j+1}-)\omega(s+,t) \leq  \vep \omega(s+,t).
\end{eqnarray*}

If $s=s_j,$ one has
\begin{eqnarray*}
|\sum_{\op|_{[s,t]}}\Xi_{u,v} - \Xi_{s,t}|&=& | (\sum_{\op|_{[s,t]}}\Xi_{u,v} -\Xi_{s,\tau_1}-\Xi_{\tau_1,t} ) + \left(\Xi_{s,\tau_1}+\Xi_{\tau_1,t}- \Xi_{s,t} \right)|\\
 &\leq & \sum_{k=2}^{n-1} \sigma(\tau_1,\tau_k)\omega(\tau_k,\tau_{k+1}) + \sigma(s,\tau_1)\omega(\tau_1,t)\\
 &\leq & \vep \omega(s+, t) + \frac{\vep}{N}  .
 \end{eqnarray*}
It follows that
\begin{eqnarray*}
|\sum_{\op}\Xi_{u,v}- \sum_{\op_\vep}\Xi_{u,v} | &\leq & \sum_{(s,t)\in \op_\vep} \vep \omega(s+,t)+ N \frac{ \vep}{N}=:C\vep ,
\end{eqnarray*}
which implies the RRS convergence. For the inequality, it follows from the algebra identity. Indeed, for any partition $\op=\{s=u_0<u_1<,...,<u_m=t\} $ of $[s,t],$ one has
$$
|\sum_\op \Xi_{u,v}- \Xi_{s,t}| = |\sum_{i=0}^{m-1} \delta \Xi_{s,u_i,u_{i+1}}| \leq \sum_{i=1}^{m-1} \sigma(s,u_i) \omega(u_i,u_{i+1})\leq \sigma(s,t-)\omega(s+,t).
$$
Now suppose $\sigma(s,t)$ left continuous in $t.$ By the above proof, for any $\vep>0,$ there exists a partition $\op_1$ on $[0,T],$ such that for any refinement $\tiop $ of $\op_1, $ one has $|\sum_{\tiop} \Xi_{u,v} - \I \Xi_{0,T} |< \vep.$ By Lemma \ref{mild control}, there exists a partition $\op_2$ such that for any $(s,t)\in \op_2,$ $\sigma(s+,t)< \vep.$ Then let $\op_3:= \op_1 \vee \op_2,$ and one has
$$
|\sum_{\op_3}\Xi_{u,v}- \I \Xi_{0,T}|< \vep \text{ and } \sigma(s+,t)< \vep , \ \text{for any }(s,t)\in \op_3.
$$
Then for any partition $\op$ with $|\op|<|\op_3|, $ let $\bar{\op}:= \op \vee \op_3,$ and one obtains
\begin{eqnarray*}
| \sum_\op \Xi_{u,v} - \I \Xi_{0,T}| & \leq& |\sum_{\bar{\op}} \Xi_{u,v} - \I\Xi_{0,T}| + |\sum_\op \Xi_{u,v}-\sum_{\bar{\op}} \Xi_{u,v} |\\
& \leq& \vep + \sum_{\tau \in \op_3 \atop (u, v)\in \op} |\delta \Xi_{u,\tau,v}| \leq \vep + \sum_{\tau \in \op_3 \atop (u, v)\in \op} \sigma(u,\tau)\omega(\tau,v)\\
&< &\vep + \vep \omega(0,T).
\end{eqnarray*}
If $\omega(s,t)$ is right continuous in $s, $ suppose $\op_\vep$ as above, and one has $|\sum_{\op_\vep}\Xi_{u,v}- \I \Xi_{0,T}|< \vep.$ Furthermore, by right-continuity of $\omega$ and the definition of $\op_\vep,$ one has for any $\tau \in \op_\vep,$
\[
\left\{
\begin{array}{ll}
\omega(\tau,t_i)<\frac{\vep}{N \sigma(0,T)},& \text{ if }\tau = s_i \\
\sigma(s_i+, \tau)< \vep, & \text{ if }\tau = t_i .
\end{array}
\right.
\]
Then for any partition $\op$ with $|\op|< |\op_\vep|,$ one has
\begin{eqnarray*}
| \sum_\op \Xi_{u,v} - \I \Xi_{0,T}| & \leq& |\sum_{\op \vee \op_\vep} \Xi_{u,v} - \I\Xi_{0,T}| + |\sum_\op \Xi_{u,v}-\sum_{\op \vee \op_\vep} \Xi_{u,v} |\\
& \leq& \vep + \sum_{\tau \in \op_\vep \atop (u, v)\in \op} |\delta \Xi_{u,\tau,v}| \leq \vep + \sum_{\tau \in \op_\vep \atop (u, v)\in \op} \sigma(u,\tau)\omega(\tau,v)\\
&< &\vep + \vep \omega(0,T)+ \vep.
\end{eqnarray*}

\end{proof}

The following helps to understand when certain limits of Riemann-Stieltjes type sums are identical.

\begin{prop}(\textbf{mild sewing for pure jumps})\label{bv pure jumps}
Suppose $\Xi_{s,t}$ is a mapping from simplex $\{ (s,t):0\leq s < t\leq T \}$ to a Banach space $(E,\|\cdot\|) .$
Assume
$$
\| \Xi_{s,t}\| \leq g(s) \omega(s,t),
$$
where $g(s)$ is a positive function on $[0,T],$ which satisfies for any $\delta>0,$ there are only finite $s\in[0,T]$ such that $g(s)>\delta$, and $\omega$ is a control, which is right-continuous in the sense of
$$
\omega(s,s+):= \lim_{t\downarrow s} \omega(s,t)=0.
$$
Then the following limit exists in the MRS sense,
$$
  \lim_{|\mathcal{P}|\rightarrow 0} \sum_{(u,v)\in \mathcal{P}} \|\Xi_{u,v}\|=0.
$$

\end{prop}

\begin{proof}
For any $\vep>0,$ and any partition $\op,$ one has

\begin{eqnarray*}
\sum_{(s,t)\in \mathcal{P}} \| \Xi_{s,t} \| &\leq& \sum_{(s,t)\in \mathcal{P}} g(s ) \omega(s,t)1_{[g(s)>\vep]}+\sum_{(s,t)\in \mathcal{P}} g(s ) \omega(s,t)1_{[g(s)\leq \vep]}\\
&\leq& \sum_{(s,t)\in \mathcal{P}} g(s ) \omega(s,t)1_{[g(s)>\vep]}+ \vep \omega(0,T).
\end{eqnarray*}
Since there are only finite points such that $g(s)>\vep,$ by the right-continuity of $\omega,$ the first summation converges to null as $|\op|\rightarrow 0.$
\end{proof}

\subsection{Left-point Riemann-Stieltjes integration}

Recall that $V^1 = V^1([0,T])$ is the space of paths which are of bounded variation on $[0,T]$ while $V^\infty$ is the space of regulated paths, i.e. for which a left and right limit exist at any point.


\begin{prop}(\textbf{integration in bounded variation case})\label{bv integration}

\begin{enumerate}

\item If $x \in V^1([0,T], \R^d), y \in V^\infty([0,T], \mathcal{L}(\R^d, \R^e)), $ the Riemann sum $\sum_{(u,v)\in \op} y_u x_{u,v} \in \R^e$ converges in the RRS-sense and we write
    $$
    \int_0^T y^\ell_r dx_r := RRS-\lim_{|\op|\rightarrow 0}  \sum_{(u,v)\in \op} y_u x_{u,v}.
    $$
Furthermore, one has
$$
|\int_s^t y_{r}^\ell dx_r - y_s x_{s,t}| \leq   \|y \|_{\infty,[s,t)} \|x\|_{1,(s,t]}.
$$

\item If $x \in V^1([0,T], \R^d), y \in V^\infty([0,T], \mathcal{L}(\R^d, \R^e)), $ and furthermore $y$ is c\`agl\`ad, or $x$ c\`adl\`ag, the Riemann sum $\sum_{(u,v)\in \op} y_u x_{u,v} $ converges in the MRS-sense. \\

\item If $x \in D^1([0,T], \R^d), y \in V^\infty([0,T], \mathcal{L}(\R^d, \R^e)), $ then
$$
\int_0^T (y^-_r)^\ell dx_r= \int_0^T (y^+_r)^\ell dx_r= \int_0^T y^\ell dx_r.
$$

\item If $x \in D^1([0,T], \R^d), y \in V^\infty([0,T], \mathcal{L}(\R^d, \R^d)), $ then one has
\begin{equation} \label{equ:clLSint}
\int_0^T y^\ell_r dx_r = \int_{(0,T]} y^- d\mu_x,
\end{equation}
with the righthand-side uses Lebsgue-Stieltjes integration, $\mu_x ([a,b]) = x(b) - x(a)$.

\end{enumerate}

\end{prop}

\begin{proof}

Set $\sigma(s,t):= \sup_{u,v\in [s,t]}|y_{u,v}|$ and $\omega(s,t):= \|x\|_{1,[s,t]}.$ Then according to Theorem \ref{bv sewing}, one has the RRS convergence, together with the estimate stated in (1).
For the second result, by Lemma \ref{vari.continuity} in the Appendix, we know if $x$ is right-continuous, so is $\|x\|_{1,[s,t]}$ in both $s$ and $t.$ Then the result follows by Theorem \ref{bv sewing} again. The third result follows by Proposition \ref{bv pure jumps}. For the last result, we recall that according to the classical result(see e.g. \cite{DN98}) that if $y$ is left-continuous, then $\int_{(0,T]}   y d\mu_x$ equals to Riemann-Stieltjes integral, which agrees with our form of integral.

\end{proof}

\begin{example} To appreciate the added generality of $\ell$-integration relative to (\ref{equ:clLSint}), say, consider $x_t$ on $[0,T]$ is constant except at $t=\tau>0$. In this case
one has
$$
\int_0^T y^\ell dx = y_{\tau-} \Delta^-_\tau x + y_\tau  \Delta^+_\tau x .
$$
\end{example}
\begin{rem} One can also define
$
\int_{s+}^t y^\ell_r dx_r := \lim_{u \downarrow s}\int_u^t y^\ell_r dx_r 
$
and readily finds (similarly for $\int_s^{t-} , \int_{s-}^{t},  \int_{s-}^{t+}$ and so on)
$$
\int_{s+}^t y^\ell_r dx_r= \int_s^t y^\ell_r dx_r - y_s \Delta^+_s x.
$$
\end{rem}

%

%
%
%


\subsection{ODEs driven by BV paths} \label{sec:BVODEs}

Consider a driving BV driving signal $x \in V^1$. We are interested in differential equations of the type
$$
      dy = f( y) dx.
$$
The rigorous meaning is in terms of an integral equations, however a naive attempt, say
$$
      y_t = y_0 + \int_{(0,t]} f(y_r) dx_r,
$$
with c\`adl\`ag driver $x$, immediately leads to problems. (Take $y_0=1$, then $f(y)=y$ and $x$ a Heaviside function with jump at time $t=1$ so that $dx=\delta_1$. Evaluation of the equation at $t=1$, gives the contradiction $y_1 = 1 + y_1$.) The ``better'' formulation, familiar to stochastic analysts, is of the form
\begin{equation} \label{equ:classicalBVSDE}
      y_t = y_0 + \int_0^t f(y^-_r) dx_r
\end{equation}

which however tacitly assumes $x$ to be c\`adl\`ag. This brings us to the equation we actually treat,
\begin{equation} \label{equ:ourBVSDE}
      y_t = y_0 + \int_0^t f(y)^\ell dx_r
\end{equation}
and which contains (\ref{equ:classicalBVSDE}) as special case when $x$ is indeed c\`adl\`ag. It is instructive to write down explicit solutions, which is possible in the scalar, linear case.

\begin{example}(\textbf{Linear equation in the bounded variation case})\label{linear equation}

The explicit and unique solution to (\ref{equ:ourBVSDE}) in dimension $1$ and with $f(y)=y$ is given by
$$
y_t:= y_0 \exp(x_t^c- x_{0}^{+}) \prod_{s\in(0,t]}(1+ \Delta_s^-x)  \prod_{s\in[0,t)}(1+ \Delta_s^+x).
$$
Remark that this reduces to the ``Doleans-Dade'' form (see e.g. Ch. 0 in \cite{RY99}) whenever $x$ is  c\`adl\`ag.
\end{example}


We checks by hand in the above example that $y$ depends continuously on $x$. In general, one has the following. (This is a special case of Theorem \ref{young ODE}, Theorem \ref{global sol. for regular} below.) 

\begin{thm}(\textbf{ODEs driven by BV paths}) \label{thm:ODE}
Given $f\in C^1(\R^e, \mathcal{L}(\R^d,\R^e)),$ the space of functions with continuous first order derivatives, and $x\in V^1([0,T],\R^d),$ there exists a unique
solution $y \in V^1([0,T],\R^e)$ solves the differential equation on some time interval $[0,t_0),$
$$
y_t = y_0 + \int_0^t f(y_r)^\ell dx_r.
$$
Furthermore, if $f \in C^1_b, $ then a global solution exists. In addition, if $\ty$ solves the same equation driven by $\tx,$ then one has the following local Lipschitz estimate,
$$
\|y-\ty  \|_{1,[0,T]} \leq C ( \|x-\tx\|_{1,[0,T]} + |y_0-\ty_0|  ),
$$
where $C$ depends on $f$ and $L$, whenever $\|x\|_{1,[0,T]},\|\tx\|_{1,[0,T]} < L$. Moreover, if $x\in D^1,$ then $y \in D^1.$

\end{thm}

\section{General rough integration} \label{sec:Rint}

We now deal with paths of finite $p$-variation, dealing first with the Young case, $p \in [1,2)$ and then the (level-$2$) rough case $p \in [2,3)$. (The case $p \ge 3$ requires branched rough paths, see Section \ref{sec:bRP}.)

\subsection{Sewing lemma with non-regular controls}



\begin{lem}\label{partition}

Suppose $\omega_i(s,t),i=1,2$ are controls. Then for any $\vep>0,$ there exists a partition $\mathcal{P}'$ such that for any $\{s,t\}\in \mathcal{P}',$
$$
\omega_1(s,t-) \wedge  \omega_2(s+,t)  < \vep,
$$
where $\omega_1(s,t-):=\lim_{r\uparrow t}\omega_1(s,r)$ and $\omega_2(s+,t)$ is defined similarly. In particular, for any $\tau\in (s,t),$ one has
$$\omega_1(s,\tau)\omega_2(\tau,t)<(\omega_1(0,T)\vee \omega_2(0,T))\vep.$$

\end{lem}

\begin{proof}

Consider $x^i_t:=\omega_i(0,t): [0,T]\rightarrow \R^+,$ nondecreasing by superadditivity of $\omega_i,i=1,2.$
Then according to Lemma \ref{regular path}, for fixed $\vep>0,$ there exist $\mathcal{P}^i,i=1,2$
such that
\[
\forall(s^i,t^i) \in \mathcal{P}^i, Osc(x^i,(s^i,t^i))<\vep.
\]
Let $\mathcal{P}:=\mathcal{P}^1\vee \mathcal{P}^2,$ and one has $\forall (s,t) \in \mathcal{P},$
$$
\max_{i=1,2}Osc(x^i,(s,t))< \vep,
$$
Then for any $(s,t)\in \mathcal{P},$ take a new endpoint $ u \in (s,t)$ and one obtains a new partition $\mathcal{P}',$ which satisfies our need. Indeed, for odd intervals $(s,t)\in \mathcal{P}',$ one has $\max_{i=1,2}Osc(x^i,(s,t])< \vep.$ Then for any $s<u<v\leq t,$
$$
\omega_2(u,v)\leq x^2_v-x^2_u \leq Osc(x^2,(s,t])<\vep.
$$
and even intervals likewise.
\end{proof}

\begin{thm}(\textbf{Young sewing })\label{young sewing}
Suppose $\Xi_{s,t}$ is a mapping from simplex $\{ (s,t):0\leq s < t\leq T \}$ to a Banach space $(E,\|\cdot\|) .$
Define $\delta\Xi_{s,u,t}:=\Xi_{s,t}-\Xi_{s,u}-\Xi_{u,t}.$ Assume $\delta\Xi$ satisfies
$$
\|\delta\Xi_{s,u,t}\| \leq \omega_1^{\alpha_1}(s,u) \omega_2^{\alpha_2}(u,t),
$$
where $\omega_1,\omega_2$ are controls on the same simplex and $\alpha_1+\alpha_2>1.$
Then the following limit exists uniquely in the RRS sense,
$$
\mathcal{I}\Xi_{0,T}:=RRS-\lim_{|\mathcal{P}|\rightarrow 0} \sum_{(u,v)\in \mathcal{P}} \Xi_{u,v},
$$
and one has the following local estimate:
$$
\| \mathcal{I}\Xi_{s,t}-\Xi_{s,t}\|\leq C \omega_1^{\alpha_1}(s,t-) \omega_2^{\alpha_2}(s+,t),
$$
with $C$ depending only on $\alpha_1+\alpha_2$.
Furthermore, if $\omega_2(s,t)$ is right-continuous in $s,$ i.e.
$$
\omega_2(s+,t):=\lim_{u \downarrow s} \omega_2(u,t)= \omega_2(s,t),\ \ \forall \ 0\leq s < t \leq T,
$$
or $\omega_1(s,t)$ is left-continuous in $t,$ the convergence holds in the MRS sense.

\end{thm}

\begin{rem} In a continuous setting, the sewing lemma usually (e.g. \cite[Ch. 4]{FH14}) comes with the assumption $\|\delta\Xi_{s,u,t}\| \leq \omega^{\alpha_1+ \alpha_2} (s,t)$. This is not sufficient to deal with jumps.
\end{rem}

\begin{proof}
For the first convergence part, one needs to fix any three indexes $\alpha_1',\alpha_2',\theta$ such that
$$
1<\theta<\alpha_1'+\alpha_2',0<\alpha_1'< \alpha_1, 0<\alpha_2'< \alpha_2.
$$
For any fixed $\vep>0,$ according to Lemma \ref{partition} one may choose a partition $\mathcal{P}$ of $[0,T],$ such that for any $(s,t)\in \mathcal{P},$ and $u\in(s,t),$

\begin{eqnarray*}
\|\delta \Xi_{s,u,t}\| &\leq & \omega_1^{\alpha_1}(s,u) \omega_2^{\alpha_2}(u,t) \leq \vep C' \omega_1^{\alpha_1'}(s,u)\omega_2^{\alpha_2'}(u,t)\\
&\leq& \vep C' \omega_1^{\alpha_1'}(s,t)\omega_2^{\alpha_2'}(s,t) = \vep C' \omega_3^\theta(s,t),
\end{eqnarray*}
where $C'$ depends on $\max_{i=1,2}\omega_i(0,T),\min_{i=1,2}(\alpha_i-\alpha_i'),$ and $\omega_3(s,t):=\omega_1^{\frac{\alpha_1'}{\theta}}(s,t)\omega_2^{\frac{\alpha_2'}{\theta}}(s,t).$ It is obvious that the above inequality holds if $s,u,t$ is replaced by any three points in $(s,t).$ Furthermore, One may check that $\omega_3$ is also super-additive, so we can apply Young's argument w.r.t. $\omega_3.$ Indeed, for any other refinement $\mathcal{P}'$ of $\mathcal{P},$ and any $(s,t)\in \mathcal{P},$ without loss of generality, say that $\mathcal{P}'|_{[s,t]}$ has $r$ intervals with $r\geq 2$. There exist two intervals $[u',u],[u,u'']$ in $\mathcal{P}'|_{[s,t]}$ such that
$$
\omega_3(u-,u+)\leq \frac{2}{r-1}\omega_3(s,t).
$$
One obtains that
\begin{eqnarray*}
\|\sum_{\mathcal{P}'|_{[s,t]}\backslash \{u\}}\Xi_{\tau,\nu}-\sum_{\mathcal{P}'|_{[s,t]}}\Xi_{\tau,\nu}\| &=& \|\delta\Xi_{u-,u,u+}\|
\leq \vep C' \omega_3^\theta(u-,u+)\\
&\leq& \vep C' \left(\frac{2}{r-1}\right)^\theta \omega_3^\theta(s,t).
\end{eqnarray*}
Iterating this process on $\mathcal{P}'|_{[s,t]},$ one has
$$
\|\sum_{\mathcal{P}'|_{[s,t]}}\Xi_{\tau,\nu}-\Xi_{s,t}\| \leq \vep \sum_{r=2}^\infty \left(\frac{2}{r-1}\right)^\theta \omega_3^\theta(s,t),
$$
and it follows that
$$
\|\sum_{ \mathcal{P}'}\Xi_{\tau,\nu}-\sum_{\mathcal{P}}\Xi_{s,t}\| \leq C' \vep \omega_3^\theta(0,T),
$$
where $C'$ is changed by multiplying a constant from the above. The convergence part is finished and the uniqueness follows by considering refinements of partitions. Indeed, if $\{\op_n\}_n,\{\op'_n\}_m$ are two sequences of partitions such that $\sum_{\op} \Xi_{s,t}$ converges, then let $\tilde{\op}_n:= \op_n \vee \op'_n,$ which is a sequence of refinements of both $\{\op_n\}_n$ and $\{\op'_n\}_n.$ By the definition of RRS convergence, the limits agree.
For the inequality part, suppose $\{\mathcal{P}_n\}_n$ a sequence of partition on $[s,t]$ such that the limit above exists. Choose any constant $1<\beta <\alpha_1+\alpha_2.$ For any $(u,v)\in \mathcal{P}_n$ and $u<\tau<v,$ one has
$$
\|\delta\Xi_{u,\tau,v}\| \leq \left(\omega_1^{\frac{\alpha_1}{\beta}}(u,v-) \omega_2^{\frac{\alpha_2}{\beta}}(u+,v)\right)^\beta.
$$
Note that for a control $\omega(s,t),$ $\omega(s+,t)$ and $\omega(s,t-)$ are also controls. One may apply Young's argument to $\omega_4(u,v):=\omega_1^{\frac{\alpha_1}{\beta}}(u,v-) \omega_2^{\frac{\alpha_2}{\beta}}(u+,v)$ and obtain
$$
\|\sum_{\mathcal{P}_n} \Xi_{u,v}-\Xi_{s,t}\|\leq C \omega_4^\beta(s,t)=C\omega_1^{\alpha_1}(s,t-) \omega_2^{\alpha_2}(s+,t).
$$
where $C$ is a generic constant. The inequality follows by taking the limit.\\

 Now suppose $\omega_2(s,t)$ right-continuous in $s$ and we are going to show the MRS-convergence. Fix some $\delta>0$ such that $\delta < (\alpha_2+\alpha_1-1) \wedge \alpha_2.$ For any $\vep>0,$ according to the convergence which we have shown and our right-continuity assumption, one may choose a partition $\op_\vep$ such that for any refinement $\bar{\op}$ of $\op_\vep,$
$$
\|\sum_{\bar{\mathcal{P}} }\Xi_{s,t}-\mathcal{I}\Xi_{0,T}\| < \vep,
$$
and for any interval $(s,t)\in \op_\vep,$
$$
\omega_2(s,v) < \vep^{\frac {1} {\delta}}, \ \ \forall v\in (s,t).
$$
Then for any $\op$ with $|\op| <|\op_\vep|$ and any interval $(s,t)\in \op,$ there exists at most one endpoint $\tau$ of $ \op_\vep$ such that $\tau\in [s,t],$ which implies $\omega_2(\tau,t) < \vep^{\frac 1 \delta}.$ Now let $\tilde{\op}:=\op \vee  \op_\vep,$ and one has

\begin{eqnarray*}
\| \sum_{ \mathcal{P} }\Xi_{s,t}-\mathcal{I}\Xi_{0,T}\| &\leq & \| \sum_{\tilde{\mathcal{P}} }\Xi_{s,t}- \mathcal{I}\Xi_{0,T}\|+\| \sum_{\tilde{\mathcal{P}} }\Xi_{s,t}- \sum_{\mathcal{P} }\Xi_{s,t}\| \\
&\leq & \vep + \sum_{\tau \in \op_\vep} \|\Xi_{s,t}-\Xi_{s,\tau}-\Xi_{\tau,t}\| \\
&\leq & \vep + \sum_{\tau \in \op_\vep} \omega_1^{\alpha_1}(s,\tau) \omega_2^{\alpha_2 }(\tau,t)\\
&\leq & \vep + \sum_{\tau \in \op_\vep} \omega_1^{\alpha_1}(s,t) \omega_2^{\alpha_2- \delta}(s,t) (\vep^\frac 1 \delta)^\delta\\
&\leq & (\omega_5(0,T) +1)\vep,
\end{eqnarray*}
where $\omega_5(s,t):= \omega_1^{\alpha_1}(s,t) \omega_2^{\alpha_2- \delta}(s,t)$ is super-additive. The case that $\omega_1(s,t)$ is left-continuous in $t$ likewise.

\end{proof}


\subsection{Young integration with jumps}

By sewing, Theorem \ref{young sewing}, we immediate obtain a proof of

\begin{prop}(\textbf{integration in Young's case})\label{Young integral}
Let $x\in V^p([0,T],\R^d),y\in V^q([0,T],\mathcal{L}(\R^d,\R^n))$ with $\frac1p+\frac1q>1$.
Then the limit ($\mathcal{P}$ paritition of $[0,T]$)
$$
\lim_{|\mathcal{P}|\rightarrow 0} \sum_{(u,v)\in \mathcal{P}}y_s x_{u,v}, =: \int_0^T y_r^{\ell} dx_r
$$
exists in RRS sense and we have the local estimate
$$
|  \int_s^t y_r^{\ell} dx_r -y_s x_{s,t}|\leq  C_{p,q} \|y\|_{q,[s,t)}\|x\|_{p,(s,t]}.
$$
If $x$ is right-continuous (i.e. c\`adl\`ag), then the above limit exists in MRS sense and we will write $\int y^- dx$ for the integral, notation consistent with the BV case (\ref{equ:clLSint}).
\end{prop}

At this stage, it seems appropriate to spend a moment to verify that our so-defined Young integral behaves in reasonable ways. For instance, with $x, y$ in $V^p,V^q$ respectively, in the Young regime $1/p+1q>1$ one readily checks a product formula of the form
$$
y_T x_T = y_0 x_0 + \int_0^T y^\ell_r dx_r + \int_0^T x^\ell_r dy_r + \sum_{t\in[0,T]} (\Delta_t^- y \Delta_t^- x + \Delta_t^+ y \Delta_t^+ x) \ .
$$
More generally, an It\^o-type formula can be given, of the form
\begin{eqnarray*}
f(x_T)=f(x_0) + \int_0^T Df(x_t)^\ell d x_t + \sum_{t\in[0,T)}\left( (f(x_{t+})-f(x_t)-Df(x_t)\Delta_t^+x \right)\\
+\sum_{t\in(0,T]}\left( (f(x_{t})-f(x_{t-})-Df(x_{t-})\Delta_t^-x )\right).
\end{eqnarray*}
It holds true for $x\in V^p([0,T],\R^d)$ and $f\in Lip^{p+\epsilon}_{loc}(\R^d, \mathcal{L}(\R^d,\R^e))$, any $\epsilon>0$, and actually for $f \in C^1$ in case $p=1$, using the results of Section \ref{sec:BV}). If one specializes to the case when $x$ is c\`adl\`ag, this form is consistent with the It\^o-formula for BV martingales (see e.g. \cite{Pro05}), though not a consequence of it.

\subsection{Sewing lemmas for rough paths and pure jumps}

We now state the sewing lemma necessary for genuine rough integration with jumps. Roughly speaking, $N$ corresponds to the number of ``levels'', so that $N=1$ was sufficient for the Young case.

aaa

\begin{thm}(\textbf{generalized sewing})\label{general sew} \label{thm:genSew}
$\Xi,\delta\Xi$ as defined in Theorem \ref{young sewing} with
$$
\|\delta\Xi_{s,u,t}\| \leq \sum_{j=1}^N \omega_{1,j}^{\alpha_{1,j}}(s,u) \omega_{2,j}^{\alpha_{2,j}}(u,t),
$$
where $\omega_{1,j}, \omega_{2,j}$ are controls and $\alpha_{1,j}+\alpha_{2,j}>1$ for all $j=1,...,N.$
Then the following limit exists uniquely in the RRS sense,
$$
\mathcal{I}\Xi_{0,T}:=RRS-\lim_{|\mathcal{P}|\rightarrow 0} \sum_{(u,v)\in \mathcal{P}} \Xi_{u,v},
$$
and one has the following estimate:
$$
\| \mathcal{I}\Xi_{s,t}-\Xi_{s,t}\|\leq C \sum_{j=1}^N \omega_{1,j}^{\alpha_{1,j}}(s,t-) \omega_{2,j}^{\alpha_{2,j}}(s+,t),
$$
with $C$ depending only on $\min_j\{\alpha_{1,j}+\alpha_{2,j}\}.$ Furthermore, if $\omega_{1,j}(s,t),j=1,...,N$ are left-continuous in $t$ or $\omega_{2,j}(s,t),j=1,...,N$ are right-continuous in $s,$ then the convergence holds in the MRS sense.

\end{thm}


\begin{proof}
Fix constants $\alpha_{1,j}',\alpha_{2,j}',\theta$ such that
$
1<\theta<\min_j\{\alpha_{1,j}+\alpha_{2,j}\}, \alpha_{1,j}'<\alpha_{1,j},\alpha_{2,j}'<\alpha_{2,j}.
$
For any $\vep>0,$ according to Lemma \ref{partition}, one may choose $\MP_j$ such that for any $(s,t)\in \MP_j,$ $u \in (s,t)$
$$
\omega_{1,j}^{\alpha_{1,j}-\alpha_{1,j}'}(s,u) \omega_{2,j}^{\alpha_{2,j}-\alpha_{2,j}'}(u,t) < C \vep,
$$
with $C$ depending on $\max_{j=1,...,N}\left(\omega_{1,j}(0,T)\vee \omega_{2,j}(0.T)\right)$ and $\min_{  j=1,...,N;i=1,2 }(\alpha_{i,j}-\alpha_{i,j}').$
Then let $\op:= \vee_{j=1}^N \op_j,$ and one has for any $(s,t) \in \op,$
\begin{eqnarray*}
\|\delta\Xi_{s,u,t}\| &\leq & \sum_{j=1}^N \omega_{1,j}^{\alpha_{1,j}}(s,u) \omega_{2,j}^{\alpha_{2,j}}(u,t)\\
&\leq & \vep C \sum_{j=1}^N \omega_{1,j}^{\alpha_{1,j}'}(s,u) \omega_{2,j}^{\alpha_{2,j}'}(u,t)    \  \leq  \ \vep C \omega_3^\theta(s,t),    \\
\end{eqnarray*}
with  $\omega_3 (s,t):=\sum_{j=1}^N \omega_{1,j}^{\frac{\alpha_{1,j}'}{\theta}}(s,t) \omega_{2,j}^{\frac{\alpha_{2,j}'}{\theta}}(s,t)$ superadditive. Then one may apply Young's argument as Theorem \ref{young sewing} and obtain that for any refinement $\op'$ of $\op,$
$$
\| \sum_{\op'}\Xi_{\tau,\nu}- \sum_{\op} \Xi_{s,t}\|< C \vep \omega_3^\theta(0,T),
$$
which implies the convergence part. The inequality part and the MRS-convergence part are also similar by applying basic inequality $\left(\sum_{i=1}^N a_i\right)^p\leq N^p \left(\sum_{i=1}^N a^p_i\right), p\geq 1,a_i\geq 0.$

\end{proof}

\begin{prop}(\textbf{level-2 rough integration})\label{level-2 rp}
Suppose $\BX=(X,\mathbb{X})$ is a rough path with finite $p$-variation for $p \in [2,3),$ and $(Y,Y')$ is a controlled rough path in the following sense,
$$
Y,Y' \in V^p , \ \ \ R_{s,t}:= Y_{s,t}- Y'_{s} X_{s,t} \in V^{\frac{p}{2}}.
$$
Define $\Xi_{s,t}= Y_s X_{s,t} + Y'_s \X_{s,t}.$ Then one has the following convergence and estimate
\begin{eqnarray}\label{local estimate for level-2 r.i.}
&&\int_0^T Y_s^\ell d\BX_s :=  RRS-\lim_{|\op | \rightarrow 0} \sum_{(u,v)\in \op} \Xi_{u,v}, \label{equ:level2Rint} \\
&&|\int_s^t Y_s^\ell d\BX_s-Y_s X_{s,t} - Y'_s \X_{s,t} | \leq C (\|R\|_{\frac{p}{2},[s,t)}  \|X\|_{ p,(s,t]} +\|Y'\|_{ p,[s,t)}  \|\X\|_{\frac{p}{2},(s,t]}),   \label{equ:level2Rintest}
\end{eqnarray}

where $C$ depends only on $p.$ In particular, if $\BX$ is c\`adl\`ag, then the convergence in (\ref{equ:level2Rint}) holds in MRS-sense and we write
$$
\int_0^T Y^-_s d\BX_s:= \text{MRS}- \lim_{|\op | \rightarrow 0} \sum_{(u,v)\in \op} \Xi_{u,v}.
$$

\end{prop}

\begin{rem}
For c\`adl\`ag rough paths, the RRS part of the above proposition was seen in \cite{FS17}.
\end{rem}

\begin{proof}

Indeed, one has
\begin{eqnarray*}
|\delta \Xi_{s,u,t}|&=& |-R_{s,u} X_{u,t}- Y'_{s,u} \X_{u,t}| \\
&\leq& \|R\|_{\frac{p}{2},[s,u]}  \|X\|_{ p,[u,t]} +\|Y'\|_{ p,[s,u]}  \|\X\|_{\frac{p}{2},[u,t]}\\
&=: &  \omega_{1,1}^{\frac 2 p }(s,u) \omega_{2,1}^{\frac1 p}(u,t) + \omega_{1,2}^{\frac 1 p }(s,u) \omega_{2,2}^{\frac 2 p}(u,t),
\end{eqnarray*}
which implies (\ref{equ:level2Rint}),(\ref{equ:level2Rintest}) by general sewing (Theorem \ref{general sew}). At last, c\`adl\`agness of $\BX$ implies right-continuity of $X_{s,t}$ and $\X_{s,t}$ in both $s,t$ which gives, as in Lemma \ref{vari.continuity},  right-continuity of $ \omega_{2,i} (i=1,2)$.

\end{proof}

\begin{rem}\label{integ.as crp} Let $Z$, as below, denote the indefinite rough integral. Then the pair $(Z,Y)$ defines a controlled rough path (w.r.t. $X)$, as is immediate from the local estimate.
\end{rem}

\begin{lem}(\textbf{preservation of local jump structure under rough integration})\label{preserve jumps}

Suppose $\BX=(X,\mathbb{X})$ is a rough path with finite $p$-variation for $p \in [2,3),$ and $(Y,Y')$ is a controlled rough path. Write
$$
Z_t:=\int_0^t Y_s^\ell d\BX_s.
$$
Then for any $t\in (0,T),$ one has
\begin{equation} \label{equ:RIJ}
\Delta_t^- Z = Y_{t-} \Delta_t^- X + Y'_{t-} \Delta_t^- \X, \ \ \Delta_t^+ Z = Y_{t} \Delta_t^+ X + Y'_{t} \Delta_t^+ \X,
\end{equation}
where we recall that $\Delta_t^+ \X:= \X_{t,t+}, \Delta_t^- \X:= \X_{t-,t}.$ In particular, if $\BX$ is right-continuous, i.e.
$(X_{s,s+} , \X_{s,s+})=0,$ so is $Z.$ Similar for the left-continuous case.
\end{lem}

aaa

\begin{proof}
Indeed, according to Proposition \ref{level-2 rp}, one has
\begin{eqnarray*}
Z_{t-} &= & \lim_{s\uparrow t } \lim_{|\op|\rightarrow 0} \sum_{(u,v)\in\op|_{[0,s]}} \Xi_{u,v}= \lim_{s\uparrow t } \lim_{|\op|\rightarrow 0} \left(\sum_{(u,v)\in\op|_{[0,s]}} \Xi_{u,v}+ \Xi_{s,t}-\Xi_{s,t}\right)\\
&= & \lim_{|\op|\rightarrow 0} \left(\sum_{(u,v)\in\op|_{[0,t]}} \Xi_{u,v} - \lim_{s\uparrow t } \Xi_{s,t}\right)=  Z_t - Y_{t-} X_{t-,t}- Y'_{t-} \X_{t-,t}.
\end{eqnarray*}
The similar argument holds for the right limit version.

\end{proof}

\begin{rem} In stochastic integration (written only here as $\cdot$), it is known that (see e.g. \cite{Pro05}), 
\begin{equation} \label{equ:SIJ}
\Delta ( H \cdot X) \equiv H \cdot \Delta X \ ,
\end{equation}
with probability one, whenever $X$ is a c\`adl\`ag semimartingale and $H$ is a suitable, left-continuous integrand process (such as the left-modification $Y^-$ of another c\`adl\`ag semimartingale $Y$). If applied to $H=X^-$, this implies (similar to the forthcoming Remark \ref{rem:cbrpl}) that $\Delta^- \X \equiv 0$ a.s. This shows that (\ref{equ:RIJ}) is precisely a deterministic generalization of the (\ref{equ:SIJ}).


\end{rem}
%
%

We have shown that if the driven path is right-continuous or the integrand is left-continuous, the convergence of (compensated) Riemann sum holds in the MRS sense. Furthermore, if the driven path is right-continuous(left-continuous if the potential Riemann sum is like $y_t x_{s,t}$), one could obtain the same limit regardless of which continuous version(left or right-continuous or mixed or the original one) of the integrand. This fact could be generalized by the following theorem.

\begin{thm}(\textbf{sewing lemma for pure jumps})\label{pure jumps}
Suppose $\Xi$ is a mapping from simplex $\{ (s,t):0\leq s<t \leq T\}$ to a Banach space $(E,\|\cdot\|),$ and
$$
\|\Xi_{s,t}\|\leq \sum_{j=1}^N \delta_{1}^{\alpha_{1,j}}(s ) \omega_2^{\alpha_{2,j}}(s,t),
$$
where $\delta_1$ is a positive pure jump function with finite summation of jumps, i.e. $\delta_1$ is non-zero only on a countable set $J \subseteq[0,T],$ and $\sum_{r\in [0,T]} \delta_1(r) < \infty,$ $\omega_2$ is a control and $\alpha_{1,j}+\alpha_{2,j}>1>  \alpha_{i,j}  $ for any $i=1,2, \ j=1,...,N.$ Furthermore, assume $\omega_2(s,t)$ is right-continuous in the following sense,
$$
\omega_2(s,s+):=\lim_{t\downarrow s}\omega_2(s,t)=0.
$$
Then the following limit exists in the MRS sense,
$$
\mathcal{I}\Xi_{0,T}:=MRS-\lim_{|\mathcal{P}|\rightarrow 0} \sum_{(u,v)\in \mathcal{P}} \| \Xi_{u,v} \|=0.
$$

\end{thm}

\begin{proof}

We only need to show the case $\| \Xi_{s, t}\| \leq \delta_1^{\alpha_1}(s ) \omega_2^{\alpha_2}(s,t)$ since this result holds in the MRS sense which implies we need not mess with partitions. Take any positive $1-\alpha_2 \leq \alpha'_1 < \alpha_1 $ and $\alpha_2':=1-\alpha_1' \leq  \alpha_2.$ For any $\vep>0,$ one has
$$
\sum_{(s,t)\in \mathcal{P}} \| \Xi_{s,t} \| \leq \sum_{(s,t)\in \mathcal{P}} \delta_1^{\alpha_1}(s ) \omega_2^{\alpha_2}(s,t)1_{[\delta_1(s)>\vep]}+\sum_{(s,t)\in \mathcal{P}} \delta_1^{\alpha_1}(s ) \omega_2^{\alpha_2}(s,t)1_{[\delta_1(s)\leq \vep]}.
$$
Since there are only finite $s\in[0,T]$ such that $1_{[\delta_1(s)>\vep]}=1,$ the first summation converges to null as $|\mathcal{P}|\rightarrow 0$ thanks to right-continuity of $\omega_2.$ For the second summation, one has
\begin{eqnarray*}
&&\sum_{(s,t)\in \mathcal{P}} \delta_1^{\alpha_1}(s ) \omega_2^{\alpha_2}(s,t)1_{[\delta_1(s)\leq \vep]}\\
&\leq & [\sum_{(s,t)\in \mathcal{P}} (\delta_1^{\alpha_1}(s ))^{\frac1 {\alpha_1'}} 1_{[\delta_1(s)\leq \vep]}]^{\alpha'_1} [\sum_{(s,t)\in \mathcal{P}} (\omega_2^{\alpha_2}(s,t ))^{\frac1 {\alpha_2'}}  ]^{\alpha'_2}\\
&\leq & \vep^{\alpha_1-\alpha'_1} [\sum_{s\in [0,T]}\delta_1(s)]^{\alpha'_1} \omega_2(0,T)^{\alpha_2},
\end{eqnarray*}
which implies our result since $\vep$ is arbitrary.
\end{proof}

The interest in the pure jump sewing lemma is a decisive understanding which RS-type approximation yield an identical limit. For instance, in the Young regime, with $X,Y \in V^p, p\in [1,2)$ and $X $ right-continuous, one readily deduces that
$$
\lim_{|\op| \rightarrow 0} \sum_\op Y_s X_{s,t} = \lim_{|\op| \rightarrow 0} \sum_\op Y^+_s X_{s,t} =\lim_{|\op| \rightarrow 0} \sum_\op Y^-_s X_{s,t},
$$
where all limit can be taken in the MRS sense.
(One only needs to check $\Xi_{s,t}^\pm= (Y_s-Y_s^\pm) X_{s,t}$ satisfy our assumption.)
As partially observed in \cite{FS17}, the rough case is more involved. Indeed, suppose $\BX=(X,\X),\ \  (Y,Y')$ as defined in
Example \ref{level-2 rp} with $\BX_{0,t}$ right-continuous, for the sake of argument. We constructed the rough integral as RRS-limit
of sums of $\Xi_{s,t}^1 =  Y_s X_{s,t}+ Y_s' \X_{s,t}$ and one {\it cannot } just replace $(Y,Y')$ by its left- resp. right-point modification, as is true in the Young case and also -
for probabilistic (!) reasons - in the stochastic integration case. The underlying deterministic fact, as revealed by Theorem \ref{pure jumps}, is that $\Xi^1$
can replaced, with liming rough integral, by
\begin{eqnarray*}
\Xi_{s,t}^{2\pm}&=& (Y_{s-}   +Y'_{s-} X_{s-,s})X_{s,t}+ Y'_{s\pm} \X_{s,t} , \\
\Xi_{s,t}^{3\pm}&=& (Y_{s+} - Y'_s X_{s,s+}  )X_{s,t} + Y'_{s\pm} \X_{s,t}  =Y_{s+}X_{s,t} + Y'_{s\pm} \X_{s,t}  .
\end{eqnarray*}
%
%
%
%
%
%
%
%
%
%
%
%
%
%
%
%
%
%
%
%
%

%



\subsection{Jump It\^o-type formula in level-$2$ rough path case}

In this section we extend the It\^{o}'s formula for level-$2$ rough paths, as given in \cite[Ch.5]{FH14} in the case of H\"{o}lder continuous paths, to the $p$-variation jump setting.
To this end, we adapt the notion of {\it reduced rough path} to the present setting, defined as pair $\BX = (X, [ \BX ])$ where $X \in V^p([0,T],\R^d)$, for some $p<3$, and $[\BX] \in V^{p/2}([0,T],\R^d \otimes \R^d)$. (Every genuine rough path gives rise to a reduced rough path by ``forgetting'' its area and by setting $[\BX]_t := X_{0,t} \otimes X_{0,t} - 2 \text{Sym} \X_{0,t} \equiv X_{0,t} \otimes X_{0,t} - 2 \s_{0,t}$.) Note the for continuous geometric rough paths (so that $\Delta^\pm X \equiv$ and $[ \BX ] \equiv 0$) everything reduces to a standard chain-rule.

\begin{thm}(\textbf{It\^{o}'s formula for rough paths})
Given a (reduced) $p$-rough path $\BX$ and $F \in Lip_{loc}^{p+\epsilon}$, with $\epsilon>0,$ one has the following identity,
\begin{eqnarray*}
F(X_T)- F(X_0)&= &\int_0^T DF(X_t)^\ell d \BX_t + \frac12 \int_0^T D^2F(X_t)^\ell d [\BX]_t\\
& +& \sum_{0 < t \leq T} \left( F(X_t)-F(X_{t-})- DF(X_{t-})\Delta_t^- X - \frac12 D^2F(X_{t-}) (\Delta_t^- X \otimes \Delta_t^- X) \right)\\
& +& \sum_{0 \leq t < T} \left( F(X_{t+})-F(X_{t})- DF(X_{t})\Delta_t^+ X - \frac12 D^2F(X_{t}) (\Delta_t^+ X \otimes \Delta_t^+ X) \right),
\end{eqnarray*}
where $\int_0^T D^2F(X_t)^\ell d [\BX]_t$ is defined as Young's integral and
$$
\int_0^T DF(X_t)^\ell d \BX_t := RRS- \lim_{|\op|\rightarrow 0} \sum_{[u,v] \in \op } \left(   DF(X_u) X_{u,v} + D^2F(X_u) \s_{u,v}    \right).
$$
In particular, if $\BX$ is c\`adl\`ag, then one has the following form
\begin{eqnarray}\label{ito formula}\nonumber
F(X_T)- F(X_0)&= &\int_0^T DF(X_t)^- d \BX_t + \frac12 \int_0^T D^2F(X_t)^- d [\BX]_t\\
& +& \sum_{0 < t \leq T} \left( F(X_t)-F(X_{t-})- DF(X_{t-})\Delta_t^- X - \frac12 D^2F(X_{t-}) (\Delta_t^- X \otimes \Delta_t^- X) \right).
\end{eqnarray}
Furthermore, if $\BX=(X,\X)$ is a rough path, then the above formula holds with $\int_0^T F(X_t)^\ell d \BX_t$ considered as rough integral and $[\BX]_t=X_{0,t}\otimes X_{0,t} - 2 \text{Sym}(\X).$

\end{thm}

\begin{proof}

First suppose $\s_{u,v}= \frac12 X_{u,v} \otimes X_{u,v}$ so that $[\BX]=0.$ Similar as the Young's case, for the algebra part, one has for any partition $\op:=\{0=:t_0<t_1<,...,< t_N:=T\}$ of $[0,T],$
\begin{eqnarray*}
&&F(X_T)-F(X_0) = \sum_{i=0}^{N-1} \left( F(X_{t_{i+1}}) - F(X_{t_{i }})- DF(X_{t_{i }})X_{t_i,t_{i+1}} - \frac12 D^2F(X_{t_{i }})X_{t_i,t_{i+1}}X_{t_i,t_{i+1}}  \right) + A_1 \\
&=& A_2 + \sum_{i=0}^{N-1} \left( F(X_{t_{i+1}-}) - F(X_{t_{i }+})- DF(X_{t_{i }})X_{t_i,t_{i+1}} - \frac12 D^2F(X_{t_{i }})X_{t_i,t_{i+1}}X_{t_i,t_{i+1}}  \right) +A_1\\
&=& A_2 + \sum_{i=0}^{N-1} \left( F(X_{t_{i+1}-}) - F(X_{t_{i }+})- DF(X_{t_{i }+})X_{t_i+,t_{i+1}-} - \frac12D^2F(X_{t_{i }+})X_{t_i+,t_{i+1}-}X_{t_i+,t_{i+1}-}  \right)\\
&& + \sum_{i=0}^{N-1} \big( DF(X_{t_{i }+})X_{t_i+,t_{i+1}-} + \frac12 D^2F(X_{t_{i }+})X_{t_i+,t_{i+1}-}X_{t_i+,t_{i+1}-} \\
&&- DF(X_{t_{i }})X_{t_i,t_{i+1}} - \frac12 D^2F(X_{t_{i }})X_{t_i,t_{i+1}}X_{t_i,t_{i+1}}  \big) +A_1\\
&=& A_2 + B_1 + \sum_{i=0}^{N-1} \left( DF(X_{t_{i }+})X_{t_i+,t_{i+1}-}-DF(X_{t_{i }})X_{t_i+,t_{i+1}-} - D^2F(X_{t_{i }})\Delta_{t_i}^+X X_{t_i+,t_{i+1}-} \right) \\
&& + \sum_{i=0}^{N-1}  \left( DF(X_{t_i}) X_{t_i+,t_{i+1}-}- DF(X_{t_i}) X_{t_i,t_{i+1}} \right) \\
& & + \sum_{i=0}^{N-1}  \left(\frac12 D^2F(X_{t_{i }+})X^{\otimes2}_{t_i+,t_{i+1}-} + D^2F(X_{t_{i }})\Delta_{t_i}^+X X_{t_i+,t_{i+1}-} -\frac12 D^2F(X_{t_{i }})X^{\otimes2}_{t_i,t_{i+1}}\right) +A_1\\
&=& A_2 + B_1 + B_2 + \sum_{i=0}^{N-1} DF(X_{t_i}) \left(-\Delta_{t_i}^+ X- \Delta_{t_{i+1}}^- X\right) + \sum_{i=0}^{N-1}  \big(\frac12 D^2F(X_{t_{i }+})X^{\otimes2}_{t_i+,t_{i+1}-} \\
&&+ D^2F(X_{t_{i }})\Delta_{t_i}^+X X_{t_i+,t_{i+1}-}     - \frac12  D^2F(X_{t_{i }})\left(\Delta_{t_i}^+ X + X_{t_i+,t_{i+1}-} + \Delta_{t_{i+1}-}^- X \right)^{\otimes2} \big) +A_1\\
&=& A_2 + B_1 + B_2 + \sum \left( -DF(X_{t_i})\Delta_{t_i}^+ X- \frac12D^2F(X_{t_{i }})(\Delta_{t_i}^+X)^{\otimes2}\right)\\
  &&+ \sum \left( -DF(X_{t_{i+1}-})\Delta_{t_{i+1}}^- X - \frac12 D^2F(X_{t_{i+1 }-})(\Delta_{t_{i+1}}^- X)^{\otimes2}\right)+ B_3 +A_1\\
&=&  A_2 + B_1 + B_2 +A_3 + B_3 + A_1,
\end{eqnarray*}
where
\begin{eqnarray*}
B_3:&=& DF(X_{t_{i+1}-})\Delta_{t_{i+1}}^- X - DF(X_{t_{i}})\Delta_{t_{i+1}}^- X- D^2F(X_{t_{i }+})\Delta_{t_{i+1}}^-X X_{t_i+,t_{i+1}-}\\
&+ & \frac12 D^2F(X_{t_{i+1}-})\left( \Delta_{t_{i+1}}^-X\right)^{\otimes2} - \frac12 D^2F(X_{t_{i}})\left( \Delta_{t_{i+1}}^-X\right)^{\otimes2}\\
&+ & D^2F(X_{t_{i}}) \Delta_{t_i}^+X   \Delta_{t_{i+1}}^-X.
\end{eqnarray*}
Then for the analytic part, one can choose a partition $\op:=\{0=:t_0<t_1<,...,< t_N:=T\}$ such that for any $[t_i,t_{i+1}],$ $X$ is continuous at either $t_i$ or $t_{i+1}, $ and furthermore $Osc(X,(t_i+,t_{i+1}-)) < \vep.$ By similar tricks as in Young's case, one can show $B_1,B_2,B_3$ converge to null and $A_2+ A_3$ converges to the jump part. At last, one only needs to check
$$
A_1= \sum_{i=0}^{N-1} \left(DF(X_{t_{i }})X_{t_i,t_{i+1}} + \frac12 D^2F(X_{t_{i }})X_{t_i,t_{i+1}}X_{t_i,t_{i+1}}\right)
$$
converges, which is implied by our sewing lemma, i.e. Theorem \ref{general sew}. Indeed, if $[\BX]=0,$ i.e. $\s_{s,t}= \frac12 X_{s,t}\otimes X_{s,t},$ let $\Xi_{u,v}:=DF(X_{t_{i }})X_{t_i,t_{i+1}} + D^2F(X_{t_{i }})\s_{t_i,t_{i+1}}.$ Then one can check that, by applying $D^2F(X) \left( \text{Sym}(X_{s,u} X_{u,t} ) \right)= D^2F(X)(X_{s,u} X_{u,t} ),$
$$
\delta \Xi_{s,u,t}= -D^2(X)_{s,u} \s_{u,t} - R_{s,u} X_{u,t},
$$
with $R_{s,t}= DF(X)_{s,t}- D^2F(X_s) X_{s,t}.$ For the general case, just set $\bar{\s} := \bar{\s} + \frac12 [\BX] $ and note that
$\bar{\s}$ corresponds again to the vanishing bracket situation to which the previous computation applies. It then suffices to note
that the Young's integral $\int_0^T D^2F(X_t) d[\BX]_t$ is well-defined.

\end{proof}


\begin{rem}(\textbf{{Relation to It\^o-F\"{o}llmer formula \cite{Foe81}}})   \label{rem:foellmer}  Let $X$ be a (for a moment, scalar) c\`adl\`ag path with {\it finite quadratic variation along $(\op_n)_n$ }where $(\op_n)_n$ is a fixed sequence of partitions, with mesh-size $|\op_n|\rightarrow 0$. More precisely, assume that 
$
\sum_{[s,t] \in \op^n} X_{s , t}^2 \delta_s
$
converges vaguely to a Radon measure, denoted by $[X,X] \equiv [X]$. Then, for $F \in C^2$,
\begin{eqnarray}\label{follmer}\nonumber
F(X_T)- F(X_0)&=& \int_0^T DF(X_{t-}) d X_t + \frac12 \int_0^T D^2F(X_{t-}) d [ X ]_t \\
 &&+ \sum_{0 < t \leq T} \left( F(X_t)-F(X_{t-})- DF(X_{t-})\Delta_t^- X - \frac12 D^2F(X_{t-}) \Delta_t^- X \Delta_t^- X \right),
\end{eqnarray}
where $\int_0^T F(X_{t-}) d X_t:= \lim_{n} \sum_{\op_n} F(X_u)X_{u,v}$ and $\int_0^T D^2F(X_{t-}) d [ X ]_t$ is well-defined, also as Riemann-Stieltjes
integral against $[X]_t \equiv [X] ( (0,t] )$. If $X$ is $d$-dimensional, the bracket definition extends component-wise, and
$[X^i,X^j]$ is defined by polarization. With the usual vector notation,  It\^{o}-F\"ollmer formula (\ref{follmer}) remains valid. We leave it to the reader
to verify  \eqref{ito formula} and \eqref{follmer} agree in their common domain of validity (i.e. $\BX$ c\`adl\`ag (reduced) $p$-rough path
on the one hand, $X$ and existence of $[X]$ on the other; examples show that neither condition implies the other.)
\end{rem}

\section{Differential equations driven by general rough paths}   \label{sec:RDE}

\subsection{Young's case}

We now study existence, uniqueness and stability of differential equations driven by $p$-variation signales, in the sense of an integral equation based on our forward integration. As usual, we start with $p \in [1,2)$ to detangle, in a first step, jump and rough path considerations. We note that in the case of a c\`adl\`ag driver, this Young regime of forward differential equations was also studied in \cite{Wil01}, by first solving the geometric equation (which can be done in a continuous setting), followed by ``fixing by hand'' the effect (geometric vs. forward ...) of jumps. The downside of that approach is that it yields no local estimates (easily extractable from our integral formulation) and no stability estimates (which turn out crucial in applications). For these reason, we (have to) take a direct route. Differences to familiar arguments (in absence of jumps; cf. \cite{Gub04, FH14}) are pointed out.

\begin{lem}(\textbf{stability of $p-$variation under smooth function}) \label{stable p-v under f}
Suppose $f\in C_b^2$ and $x,y \in V^p([0,T],\R^d) $ with $\|x\|_{p,[0,T]}, \|y\|_{p,[0,T]}< K.$ Then one has
$$
\|f(x)-f(y)\|_{p,[0,T]} \leq C_{p,K}\|f\|_{C_b^2} \left( |x_0- y_0| + \|x-y\|_{p,[0,T]} \right).
$$

\end{lem}

\begin{proof}
For any $s,t \in [0,T],$ by calculus and inserting term $\int_0^1 Df(y_t+ r (x_t-y_t))dr (x_s-y_s)$ one has
\begin{eqnarray*}
|f(x)_{s,t}- f(y)_{s,t}| &=& | \int_0^1 Df(y_t+ r (x_t-y_t))dr (x_t-y_t)- \int_0^1 Df(y_s+ r (x_s-y_s))dr (x_s-y_s) |\\
&\leq & C \|f\|_{C_b^2} \left( |(x-y)_{s,t}| + (|x_{s,t}|+ |y_{s,t}|)\|x-y\|_{\sup} \right).
\end{eqnarray*}
By taking $p-$variation and notice $\|x-y\|_{\sup} \leq |x_0-y_0|+ \|x-y\|_p,$ one obtains the result.

\end{proof}


\begin{thm}(\textbf{ODEs in Young's case}) \label{young ODE}
Suppose $x\in V^p([0,T],\R^d),\ p\in[1,2),$ is right-continuous at $t=0,$ and $f\in C^2(\R^e,\mathcal{L}(\R^d,\R^e)).$ Then for any $y_0\in \R^e,$ there exists $t_0>0,$ such that there exists a unique $y_t,$ satisfying
$$
y_t= \int_0^t f(y_r)^\ell dx_r  + y_0,\ \ \ \  t\in(0,t_0),
$$
where $\int_0^t f(y_r)^\ell dx_r$ as introduced in Proposition \ref{Young integral}.
Furthermore, if $x$ is c\`adl\`ag, then according to Lemma \ref{preserve jumps}, $y$ is also c\`adl\`ag.

\end{thm}

\begin{proof}
Here we only show the local result, and we give the global result in the level-$2$ rough path case where the same argument is also applied here.
Fix $p'\in (p,2),$ and denote
$$
\Omega_t:=\{Y\in W^{p'}([0,T],\R^e) | Y_0=y_0, \|Y\|_{p',[0,t]} \leq 1 \}.
$$
Define a priori mapping
$$
\begin{array}{llll}
\mathcal{M}_t:& \Omega_t &\longrightarrow & \Omega_t\\
& Y& \mapsto & Z_.:=\int_0^. f(Y_r)^\ell dx_r + y_0,
\end{array}
$$
with $t$ to be determined. Then according to the local estimate for the integral, one has
\begin{eqnarray*}
|Z_{s,t}| &\leq & \|f\|_\infty |x_{s,t}| + C\|f(Y)\|_{p',[s,t]} \|x\|_{p',[s,t]}\\
          & \leq& \|f\|_\infty |x_{s,t}| + C\|Df\|_\infty \|Y\|_{p',[s,t]} \|x\|_{p',[s,t]},
\end{eqnarray*}
where $C$ depends on $p'$ and $\|f\|_{\infty}:= \|f\|_{\infty,[-|y_0|-1, |y_0|+1]}$ so does $\|Df\|_\infty.$ Take the $p'$-variation and one obtains
\begin{eqnarray*}
\|Z\|_{p',[0,t]}^{p'} &\leq & \|f\|_{C^1}^{p'} (C+1) \|x\|_{p,[0,T]}^p Osc(x,[0,t])^{p'-p}.
\end{eqnarray*}
Thanks to the right continuity of at $t=0,$ one can choose $t_0$ small such that
$$
\|f\|_{C^2}^{p'} (C+1) \|x\|_{p,[0,T]}^p O{sc}(x,[0,t])^{p'-p} < 1.
$$
Then one obtains the invariant mapping $\mathcal{M}_{t_0}$ on $\Omega_{t_0}.$ Now we show $\mathcal{M}_{t_0}$ is a contraction. For any $Y,Y'\in \Omega_{t_0},$ one has
\begin{eqnarray*}
|Z_{s,u}-Z'_{s,u}| &=& |\int_s^u f(Y_r)^\ell dx_r - \int_s^u f(Y'_r)^\ell dx_r|\\
                   &=& |\int_s^u [f(Y_r)-f(Y'_r)]^\ell dx_r - [f(Y_s)-f(Y'_s)]x_{s,u}  + [f(Y_s)-f(Y'_s)]x_{s,u}|\\
                   &\leq & C \|f\|_{C^2} \| Y - Y' \|_{p',[s,u] }  \|x\|_{p',[s,u]} +  \|f\|_{C^1} |Y_s-Y'_s| |x_{s,u}|.\\
\end{eqnarray*}
where we apply Lemma \ref{stable p-v under f} in the last inequality. Take the $p$-variation again and one obtains for $t\leq t_0,$
\begin{eqnarray*}
\|Z-Z'\|^{p'}_{p',[0,t]} &\leq & (C+1) \|f\|_{C^2}^{p'} \| Y - Y' \|_{p',[0,t] }^{p'} \|x\|_{p,[0,t]}^p O{sc}(x,[0,t])^{p'-p}\\
& <& \| Y - Y' \|_{p',[0,t] }^{p'},
\end{eqnarray*}
which implies that the fixed point theorem holds and there exists a unique solution $y \in W^{p'}$ on $[0,t_0].$ The proof is finished by noticing
$$
\|y\|_{p,[0,t_0]} \leq \|f\|_{C^1} \|x \|_{p,[0,t_0]}.
$$

\end{proof}

\begin{rem}\label{explode}
(i). From the proof, Picard iteration needs $Osc(x)$ at $t=0$ relatively small w.r.t. the system. Otherwise the iteration may explode, which is shown in the following example. Define
\[
x_t:=\left\{
\begin{array}{ll}
-K_0, & t=0,\\
t, & t>0,
\end{array}\right.
\]
with $K_0>1.$ Consider the following linear equation
$
y_t= \int_0^t (y_r + 1) dx_r.
$
If one begins the Picard iteration from $y_t^{(0)}\equiv 0,$ one has
$$
y_t^{(1)}= \int_0^t 1 dx_r= K_0 +t \geq K_0, \ \ t>0,
$$
And so the Picard iteration explodes: after $n$ iterations we have
$$
y_t^{(n)} \geq \int_0^t  K_0^{n-1} dx_r = K_0^n +K_0^{n-1} t \geq K_0^n   \rightarrow \infty, \ \ \ n\rightarrow \infty.
$$
(ii). Even without small oscillation at the start point, one can still solve the equation on the whole time interval by fixing the original equation to a right-continuous version. Indeed, consider
$$
y_t = \xi + \int_0^t \phi(y_s)^\ell dx_s,
$$
where $x$ is a path with finite $p$-variation with $p\in [1,2).$ Suppose $y$ is a solution for the above equation. Then by our definition of integration, $y$ satisfies $y_{0+}= \xi + \phi(\xi) x_{0,0+} ,$ and
$$
\ty_t:=\left\{
\begin{array}{ll}
y_{0+}, & t=0 \\
y,& t>0,
\end{array}
\right.
$$
solves
$$
\ty_t= y_{0+} + \int_0^t \phi(\ty_s)^\ell d\tx_s, \ \ \ t\in [0,t^*],
$$
where
\[
\tx_t:=\left\{
\begin{array}{ll}
x_{0+}, & t=0 \\
x_t,& t>0.
\end{array}
\right.
\]
Conversely, if $\ty$ solves the later equation with initial value $y_{0+},$ then $y$ solves the former equation. Then furthermore by iterating this procedure, one can also have a global solution.
We will give general results about this argument in Theorem \ref{global sol.} and Theorem \ref{global sol. for regular}.

\end{rem}

%
%
%

\subsection{Level-2 rough path case}

In this part, we give solutions of differential equations driven by level-2 rough paths along with stability results for integrals and equations. The H\"{o}lder continuous case is given in \cite{FH14}. Compared with the H\"{o}lder continuous case, estimates here need to be more precise (or nasty) since, for example, $\|R^Y\|_{\frac p2,[s,t]}$ does not converge to null as $t \downarrow s.$ Now we set some notations for convenience. In this part $\BX=(X,\X)$ is a rough path with finite $p$-variation, not necessary c\`adl\`ag as before, and $p$ always belongs to $[2,3).$ Recall that we call a rough path $\BX$ is c\`adl\`ag if $\BX_{0,t}$ is c\`adl\`ag in $t.$ Denote the set of $\BX-$controlled rough paths (indeed $X-$ controlled) as $\MV_\BX^p([0,T],\R^e),$ which is a linear space according to the definition of controlled rough paths. We equip the linear space with norm
$$
\| Y,Y' \|_p:= \| Y_0 \| + \| Y'_0 \| + \|Y' \|_{p,[0,T]} + \| R^Y \|_{\frac p 2, [0,T]}, \ \ \ \forall (Y,Y') \in \MV_\BX^p,
$$
which makes $\MV_\BX^p([0,T],\R^d)$ a complete metric space. We remark that by definition $\|Y\|_p \leq C_p (\|Y'\|_\infty \|X\|_p + \|R^Y\|_{\frac p 2})$ and $\|Y'\|_\infty \leq |Y'_0|+ \|Y'\|_{p}.$  Recall that for any $(Y,Y')\in \MV_\BX^p,$ one has the estimate
$$
|\int_s^t Y_s^\ell d\BX_s-Y_s X_{s,t} - Y'_s \X_{s,t} | \leq C_p (\|R\|_{\frac{p}{2},[s,t]}\|X\|_{p,[s,t]} + \|Y'\|_{p,[s,t]}\|R\|_{\frac{p}{2},[s,t]}).
$$
Now consider equations like
\begin{equation}\label{level-2 equa}
Y_t=\int_0^t F(Y_s)^\ell d\BX_s + y_0,
\end{equation}
where $F\in C_b^3$ and the integral is introduced in Example \ref{level-2 rp}. We need to show some stability results before solving this equation. For $(Y,Y')\in \MV_{\BX}^p([0,T],\R^e) ,$ $(\tY,\tY') \in \MV_{\tBX}^{p}([0,T],\R^e),$ set $\Delta Y=Y-\tY,$ and $\Delta X, \Delta \X,  \Delta Y', \Delta R^Y$ likewise. Also set the following notations for convenience
\begin{eqnarray*}
&&M_{Y'}:=|Y_0'|+\|Y'\|_{p,[0,T]} , \Delta M_{Y'} := |\Delta Y_0'|+\|\Delta Y'\|_{p,[0,T]},\\
&&K_{Y}:= M_{Y'} + \|R^{Y}\|_{\frac p 2,[0,T]} , \Delta K_Y :=  \Delta M_{Y'} + \|\Delta R^Y \|_{\frac p 2,[0,T]},\\
&&\|\BX \|_p:= \|X\|_{p,[0,T]} + \|\X\|_{\frac p 2,[0,T]} , \| \Delta  \BX \|_p :=\|\BX;\tBX\|_p= \|\Delta X\|_{p,[0,T]} + \|\Delta \X \|_{\frac p 2,[0,T]}.
\end{eqnarray*}

\begin{lem}(\textbf{stability of rough integration})\label{stable under integral}
Suppose $(Y,Y')\in \MV_{\BX}^p([0,T],\R^e),$ $(\tY,\tY') \in \MV_{\tBX}^{p}([0,T],\R^e).$ Then by Remark \ref{integ.as crp}, $(I_{\BX}(Y),Y):=(\int_0^. Y^\ell d\BX, Y)$ is a controlled rough path(similar for $(I_{\tBX}(\tY),\tY)$). Furthermore one has the following local Lipschitz estimate
\begin{eqnarray*}
&& \|Y-\tY \|_{p,[0,T]} \leq C_p (\Delta M_{Y'} \|X\|_{p,[0,T]} +  M_{\tY'} \|\Delta X\|_{p,[0,T]} +\|\Delta R^Y \|_{\frac p 2, [0,T]})\\
&& \|R^{I_\BX(Y)} - R^{I_{\tBX}(\tY)}\|_{\frac p 2 ,[0,T]} \leq C_p (\|  \BX \|_{p,[0,T]}+1) (K_{\tY} \| \Delta  \BX \|_{p,[0,T]} + \| \BX \|_{p,[0,T]} \Delta K_Y).
\end{eqnarray*}

\end{lem}

\begin{proof}
For the first inequality, by inserting $\tY'_s X_{s,t}$ one has
\begin{eqnarray*}
|\Delta Y_{s,t}| &=& |Y'_s X_{s,t} + R^Y_{s,t} - \tY'_s \tX_{s,t} - R^{\tY}_{s,t}|\\
&\leq & |\Delta Y'_s| |X_{s,t}| + |\tY'_s| |\Delta X_{s,t}| +|\Delta R^Y_{s,t}|\\
& \leq & (|\Delta Y'_0|+\|\Delta Y'\|_p) |X_{s,t}| + (|  \tY'_0|+\|  \tY'\|_p)|\Delta X_{s,t}| +|\Delta R^Y_{s,t}|
\end{eqnarray*}
which implies the first estimate. For the second inequality, denote $\Delta \Xi_{s,t}=Y_s X_{s,t} + Y'_s \BX_{s,t}-(\tY_s \tX_{s,t} + \tY'_s \tilde{\X}_{s,t}).$ Then one has,
\begin{eqnarray*}
| R^{I_\BX(Y)}_{s,t} - R^{I_{\tBX}(\tY)}_{s,t} | &=& |\int_s^t Y^\ell_r d\BX_r - Y_s X_{s,t} -(\int_s^t \tY^\ell_r d\tBX_r - \tY_s \tX_{s,t})|\\
&\leq & |Y'_s\X_{s,t} - \tY'_s \tilde{\X}_{s,t} | + |I(\Delta \Xi)_{s,t} - \Delta \Xi_{s,t} |\\
&\leq & M_{\tY'} |\Delta \X_{s,t} | + \Delta M_{Y'} |\X_{s,t}| + C_p \sup_{\tau,u,\nu \in[s,t]} |\delta(\Delta \Xi)_{\tau,u,\nu}|,
\end{eqnarray*}
where the last inequality comes from the proof of the generalized sewing lemma with $\delta(\Delta \Xi)_{\tau,u,\nu}= R^Y_{\tau,u} X_{u,\nu} + Y_{\tau,u} \X_{u,\nu} -(R^{\tY}_{\tau,u} \tX_{u,\nu} + \tY_{\tau,u} \tilde{\X}_{u,\nu}).$ Now we only need to bound the $\frac p 2 -$variation of $A_{s,t} :=\sup_{\tau,u,\nu \in[s,t]} |\delta(\Delta \Xi)_{\tau,u,\nu}|$ by $(\|\BX \|_p+1) (K_{\tY} \| \Delta  \BX \|_p + \|\BX \|_p \Delta K_Y).$ Indeed, by inserting terms, one has
\begin{eqnarray*}
|\delta (\Delta \Xi)_{\tau,u,\nu}| \leq |\Delta Y_{\tau,u} \X_{u,\nu } | +|\tY_{\tau,u} \Delta \X_{u,\nu}| +|\Delta R^Y_{\tau,u} X_{u,\nu}| + |R^{\tY}_{\tau,u} \Delta X_{u,\nu}|.
\end{eqnarray*}
By applying the estimate for $\|\Delta Y\|_p$ we build, one has
\begin{eqnarray*}
\|A\|_{\frac p 2} &\leq& C_p (\|\Delta Y\|_p \|\X\|_{\frac p 2} + \|\tY \|_p \|\Delta \X\|_{\frac p 2} +\|\Delta R \|_{\frac p 2}\|X \|_p+ \| R_{\tY} \|_{\frac p 2}\|\Delta X \|_p )\\
&\leq & C_p [(K_{\tY}+\|\tY \|_p) \| \Delta  \BX \|_p + \|\BX \|_p (\Delta K_Y + \|\Delta Y\|_p)]\\
&\leq & C_p [(K_{\tY}+K_{\tY}(\|X\|_p +1) ) \| \Delta  \BX \|_p + \|\BX \|_p (\Delta K_Y+ \Delta K_Y \|\BX \|_p + K_{\tY} \| \Delta  \BX \|_p + \Delta K_Y  )]\\
&\leq & C_p (\| \BX \|_p+1) (K_{\tY} \| \Delta  \BX \|_p + \|\BX \|_p \Delta K_Y).
\end{eqnarray*}

\end{proof}

\begin{lem}(\textbf{stability of controlled r.p. under smooth function})\label{stable under smooth fun}
Suppose $\phi \in C_b^3(\R^e, \ML(\R^d, \R^e ) ),$ $(Y,Y')\in \MV_\BX^p([0,T],\R^e).$ Then $(\phi(Y),\phi(Y)'):=(\phi(Y), D\phi(Y)Y') \in \MV_\BX^p([0,T],\ML(\R^d, \R^e )),$ and one has
$$
\|\phi(Y)'\|_{p,[0,T]} + \|R^{\phi(Y)}\|_{\frac p 2, [0,T]} \leq C_p \|\phi \|_{C_b^2} K_Y(1+K_Y) (1+\|X\|_{p,[0,T]})^2.
$$
 Moreover, suppose $(\tY,\tY') \in \MV_{\tBX}^{p}([0,T],\R^e).$ Then one has estimates
\begin{eqnarray*}
&&\|\phi(Y)'-\phi(\tY)' \|  \leq  C_p \|\phi\|_{C_b^2} (M_{\tY'} +1) (|\Delta Y_0|+ \|\Delta Y'\| +\|\Delta R^Y \| + M_{\tY'} \|\Delta X\|  + \Delta M_{Y'} \|X\|  ),\\
&&\|R^{\phi(Y)}- R^{\phi(\tY)} \|  \leq   C_{p,K_Y,K_{\tY} } (1+\|X\|_p)^2  \|\phi \|_{C_b^3} (|\Delta Y_0 |+ \|\Delta R^Y \| + \Delta M_{Y'} \|X\| + M_{\tY'} \|\Delta X\| )
\end{eqnarray*}
where we omit the obvious subscripts for norms.

\end{lem}

\begin{proof}

We show the first inequality by showing both $\|\phi(Y)'\|_{p,[0,T]}$ and $\|R^{\phi(Y)}\|_{\frac p 2, [0,T]}$ bounded by the right hand side. By inserting $D\phi(Y_t)Y'_s,$ one has
\begin{eqnarray*}
|D\phi(Y_t)Y'_t- D\phi(Y_s)Y'_s| & \leq & \|\phi\| (|Y'_t-Y'_s| + |Y'_s| |Y_{t}-Y_s|)\\
& \leq & \|\phi \| (|Y'_{s,t}|+ M_{Y'}(M_{Y'}|X_{s,t}|+ |R^Y_{s,t}|))
\end{eqnarray*}
Since $M_{Y'}= |Y'_0|+ \|Y'\|_{p,[0,T]} \leq K_Y,$ one obtains
\begin{eqnarray*}
\|\phi(Y)'\|_{p,[0,T]} &\leq& C_p \|\phi \|( \|Y'\|_p+ M_{Y'} M_{Y'} \|X \|_p + M_{Y'} \|R^Y\|_{\frac p 2} )\\
                    &\leq& C_p \|\phi \| M_{Y'} (1+ M_{Y'} \|X \|_p + \|R^Y\|_{\frac p 2} )\\
                    &\leq & C_p \|\phi \| K_Y (1+K_Y ) (1+\|X \|_p)
\end{eqnarray*}
For the bound of $\|R^{\phi(Y)}\|_\frac p 2,$ by Taylor's expansion,
\begin{eqnarray*}
R^{\phi(Y)}_{s,t} &=& \phi(Y_t)-  \phi(Y_s) -D \phi(Y_s) Y_s' X_{s,t}\\
&= & \phi(Y_t)-  \phi(Y_s) -D \phi(Y_s) Y_{s,t} + D\phi(Y_s) R^Y_{s,t} \\
& = & \frac12 D^2 \phi(Y_u) Y_{s,t}^{\otimes 2} + D\phi(Y_s) R^Y_{s,t},
\end{eqnarray*}
with $u\in[s,t],$ which implies
\begin{eqnarray*}
\|R^{\phi(Y)}\|_\frac p 2 &\leq&  C_p \|\phi\| [(M_{Y'}\|X\|+ \|R^Y\|)^2 + \|R^Y\|]\\
&\leq&  C_p \|\phi\| K_Y [K_Y(1+\|X\|)^2 +1 ],
\end{eqnarray*}
which implies the first inequality. For the first local Lipschitz estimate, by inserting $D\phi(Y) \tY',$ one has
\begin{eqnarray*}
\|\phi(Y)'-\phi(\tY)' \|_p & \leq & C_p \|\phi \| \|\Delta Y' \|_p + \|\phi \| M_{\tY'} (\|\Delta Y\|_p +|\Delta Y_0| )\\
&\leq &  C_p \|\phi \|  [\|\Delta Y' \|_p + M_{\tY'} (\Delta M_{Y'} \|X\|_p +M_{\tY'} \|\Delta X\|_p + \|\Delta R^Y \|_\frac p2 + |\Delta Y_0|)]\\
&\leq &  C_p \|\phi \|  (M_{\tY'} +1) ( |\Delta Y_0| \|\Delta Y'\| +\|\Delta R^Y \| + M_{\tY'} \|\Delta X\|  + \Delta M_{Y'} \|X\|  ).
\end{eqnarray*}
For the second local Lipschitz inequality, by inserting $D\phi(Y_s)Y_{s,t}$ and $ D\phi(\tY_s) \tY_{s,t},$ one has
\begin{eqnarray*}
|R^{\phi(Y)}_{s,t}- R^{\phi(\tY)}_{s,t}| &=& |\phi(Y)_{s,t} - \phi(\tY)_{s,t} - D \phi(Y_s) Y_s' X_{s,t} -D \phi(\tY_s) \tY_s' \tX_{s,t} |\\
&\leq & |\phi(Y)_{s,t}-D\phi(Y_s)Y_{s,t} - (\phi(\tY)_{s,t}- D\phi(\tY_s) \tY_{s,t})| + |D\phi(Y_s)R^Y_{s,t}- D\phi(\tY_s)R^{\tY}_{s,t}|\\
&=&  |\int_0^1 \int_0^1 D^2\phi(Y_s + r_1 r_2 Y_{s,t}) Y_{s,t}^{\otimes 2}dr_1 dr_2 - \int_0^1 \int_0^1 D^2\phi(\tY_s + r_1 r_2 \tY_{s,t}) \tY_{s,t}^{\otimes 2}dr_1 dr_2|\\
&& + |D\phi(Y_s)R^Y_{s,t}- D\phi(\tY_s)R^{\tY}_{s,t}|\\
&=: & A+B.
\end{eqnarray*}
For $A,$ by inserting terms as in Lemma \ref{stable p-v under f}, one has
\begin{eqnarray*}
\|A\|_{\frac p 2} &\leq& C_p \|\phi \| [|\Delta Y|_\infty \|Y\|_p^2 + (\|Y\|_p+\|\tY \|_p) \|\Delta Y\|_p ]\\
                   &\leq& C_{p,K_Y,K_{\tY}  }(1+\|X\|_p)^2 \|\phi \| (|\Delta Y_0 | + \|\Delta Y\|_p )\\
                   &\leq& C_{p,K_Y,K_{\tY}  }(1+\|X\|_p)^2 \|\phi \| (|\Delta Y_0 | + \Delta M_{Y'} \|X\|_p + M_{\tY'} \|\Delta X\|_p + \|\Delta R^Y \|_{\frac p 2})
\end{eqnarray*}
For $B,$ by inserting $D\phi(\tY_s) R^Y_{s,t},$ one has
\begin{eqnarray*}
\|B \|_{\frac p 2} &\leq& C_p \|\phi \| (|\Delta Y|_\infty \|R^Y\|_p^2 +   \|\Delta R^Y \|_\frac p 2 )\\
                   &\leq& C_{p,K_Y  } \|\phi \| (|\Delta Y_0 | + \|\Delta Y\|_p + \|\Delta R^Y\|_\frac p 2 )\\
                   &\leq& C_{p,K_Y  } \|\phi \| (|\Delta Y_0 | + \Delta M_{Y'} \|X\|_p + M_{\tY'} \|\Delta X\|_p + \|\Delta R^Y \|_{\frac p 2}),
\end{eqnarray*}
which completes the proof.

\end{proof}


\begin{lem}(\textbf{invariance})\label{invariance}
Suppose $(Y,Y') \in \MV_{\BX}^p,$ $\phi \in C_b^3.$ Then $(Z,Z'):=(\int_0^. \phi(Y_r)^\ell d \BX_r, \phi(Y)) \in \MV_{\BX}^p,$ and one has estimates
\begin{eqnarray*}
&&\|\phi(Y) \|_{p,[0,T]} \leq  C_p \|\phi\|_{C_b^1} (M_{Y'} \|X\|_{p,[0,T]} + \|R^Y\|_{\frac p 2,[0,T]} )\\
&&\|R^{I_\BX(\phi(Y) )}\|_{\frac p 2,[0,T]} \leq C_{p  }\|\phi \|_{C_b^2} K_Y (1+K_Y) (1+\|X\|^2_p )\|\BX\|_p.
\end{eqnarray*}

\end{lem}

\begin{proof}
The first estimate comes from our estimate for $\|Y\|_p$ and Lipschitzness of $\phi.$ For the second estimate, according to the local estimate for rough integral, one has
\begin{eqnarray*}
\|R^{I_\BX(\phi(Y) )}\|_{\frac p 2} &\leq & C_p( \|D \phi(Y_s) Y'_s\|_\infty \|\X\|_{\frac p 2} + \|\phi(Y)'\|_p \|\X\|_{\frac p 2} + \|R^{\phi(Y)}\|_{\frac p 2} \|X\|_p)\\
&\leq & C_{p  }\|\phi \|_{C_b^2} K_Y (1+K_Y) (1+\|X\|^2_p )(\|\X\|_{\frac p 2} + \|X\|_p),
\end{eqnarray*}
where we apply Lemma \ref{stable under smooth fun} in the last inequality.

\end{proof}

\begin{lem}(\textbf{local contraction})\label{contraction}
Suppose $(Y,Y') \in \MV^p_{\BX}, $ $(\tY,\tY') \in \MV^p_{\tBX} $ and $\phi \in C_b^3.$ Then $(Z-\tilde{Z},Z'-\tilde{Z'}):=(I_{\BX}(\phi(Y))-I_{\tBX}(\phi(\tY)), \phi(Y)-\phi(\tY))$ has the following estimate,
\begin{eqnarray*}
 && \|\phi(Y)-\phi(\tY) \|_{p,[0,T]} \leq C_p \|\phi \|_{C_b^3}  (|\Delta Y_0|+ \Delta M_{Y'} \|X\|_p + M_{Y'} \|\Delta X\|_p +\|\Delta R^Y \|_{\frac p 2} )\\
 &&\|R^{I_{\BX}(\phi(Y))}- R^{I_{\tBX}(\phi(\tY))}\|_{\frac p 2, [0,T] } \leq C_{p,K_Y,K_{\tY} }\|\phi\|_{C_b^3}  \left(1+\|\BX \|_p \right)^3 \big(\|\Delta \BX \|_p \\
 &&+ \|\BX\|_p (\|\Delta Y, \Delta Y'\| + \Delta K_Y \|X\|_p+ K_{\tY} \|\Delta X\|_p)\big).
\end{eqnarray*}

\end{lem}

\begin{proof}

The first one follows easily from the Lemma \ref{stable p-v under f} and the estimate for $\|\Delta Y \|_p$ in Lemma \ref{stable under integral}. For the second one, applying Lemma \ref{stable under integral}, one has
\begin{eqnarray*}
\|R^{I_{\BX}(\phi(Y))}- R^{I_{\tBX}(\phi(\tY))}\|_{\frac p 2} \leq C_p (\|\BX\|_p +1 ) (K_{\phi(\tY)} \| \Delta \BX\|_p + \| \BX\|_p \Delta K_{\phi(Y)} ).
\end{eqnarray*}
By Lemma \ref{stable under smooth fun}, one has
\begin{eqnarray*}
K_{\phi(\tY)} & = & |\phi(\tY_0)'| + \|\phi(\tY)'\|_p + \|R^{\phi(\tY) }\|_{\frac p 2} \leq C_p \| \phi \|   K_{\tY}(1+K_{\tY} )(1+\|X\|_p )^2     \\
\Delta K_{\phi(Y)} &=& |\Delta \phi(Y)_0' | + \|\Delta \phi(Y)' \|_p + \|\Delta R^{\phi(\tY) }\|_{\frac p 2}\\
&\leq & C_{p,K_Y,K_{\tY} } \|\phi \| (1+\|X\|_p)^2 (|\Delta Y_0| + \Delta K_Y + \Delta K_Y \|X\|_p + K_Y \|\Delta X\|_p ).
\end{eqnarray*}
The estimate follows by plugging the later two inequalities into the first one.

\end{proof}

Now we are ready to solve equation \eqref{level-2 equa}.


\begin{thm}(\textbf{differential equation driven by level-2 rough paths})\label{solve level-2}
Suppose $\BX  $ is a level-2 rough path with finite $p$-variation, which is right-continuous at $t=0.$ Then for any $F \in C^3,$ there exists $t_1>0,$ such that the following equation has a unique solution on $[0,t_1],$
\begin{equation}
Y_t=y_0 + \int_0^t F(Y_s)^\ell d\BX_s .
\end{equation}
Furthermore, suppose $\tY$ solves the equation driven $\tBX.$ Then one has the following local Lipschitz estimate: for any rough paths $\BX,\tBX$ and $t$ such that $\|\BX\|_{p,[0,t]},\|\tBX\|_{p,[0,t]}< \delta_{p,\|F\|},$
where $\|F \|_{C^3}:= \|F\|_{C^3,[-K\cdot \mathbf{1},K\cdot \mathbf{1}]}$ with $K:=|Y_0|+ |F(Y_0)|+2,$ one has
\begin{eqnarray*}
\|Y-\tY \|_{p,[0,t ]} \leq C_p \|F \| (\|  \BX ; \tBX \|_{p,[0,t ]} + |  Y_0- \tilde{Y}_0|).
\end{eqnarray*}


\end{thm}

\begin{proof}

Since $\BX$ is right-continuous at $t=0,$ by Lemma \ref{vari.continuity}, $\omega_{\BX}(0,t)$ is right-continuous at $t=0.$ Without loss of generality, assume $\|\BX \|_{p,[0,T]} \leq 1$. For the existence and uniqueness part, we are going to build the fixed point theorem for the equation. Consider the closed set
$$
\Omega_t:= \{(Y,Y')\in \MV_{\BX}^p: Y_0=y_0,  Y'_0=F(Y_0), \|Y'\|_{p,[0,t]} \leq 1, \|R^Y\|_{\frac p 2, [0,t]} \leq \frac{1}{2 C_p\|F\|+1 } \},
$$
with $C_p$ described in Lemma \ref{invariance}, $\|F \|:= \|F\|_{C^3,[-K\cdot \mathbf{1},K\cdot \mathbf{1}]}$ and $t $ to be determined. Define a mapping on
$\mathcal{M}_t: \Omega_t \longrightarrow  \Omega_t$ given by
$$
 (Y,Y') \mapsto (Z,Z') := ( \int_0^. F(Y_r)^l d\BX_r + y_0, F(Y) ).
$$
According to Lemma \ref{invariance}, one has
\begin{eqnarray*}
\|F(Y)\|_{p,[0,t]} &\leq & C_p \|F\| (1+ F(y_0)) \|X \|_{p,[0,t]} + \frac 12\\
\|R^Z\|_{\frac p 2,[0,t]} &\leq & C_p \|F \| (3+F(y_0) )^2 \| \BX \|_{p,[0,t]}.
\end{eqnarray*}
where $K:=|Y_0|+ (|F(Y_0)|+1)+1.$ By choosing $t=T_1$ small, one obtains $C_p \|F\| (1+ F(y_0)) \|X \|_{p,[0,t]}< \frac 12$ and $C_p \|F \| (3+F(y_0) )^2 \| \BX \|_{p,[0,t]}<\frac{1}{2 C_p\|F\|+1 } ,$ which implies the invariance. For the contraction part, according to Lemma \ref{contraction}, for $(Y,Y'),(\tY,\tY') \in \Omega_t,$
\begin{eqnarray*}
\|F(Y)-F(\tY)\|_{p,[0,t]} &\leq & C_p \|F\| (\|\Delta Y\|_p \|X\|_{p,[0,t]} + \|\Delta R^Y \|_{\frac p 2,[0,t]})   \\
\|R^Z-R^{\tilde{Z}} \|_{\frac p 2,[0,t]} &\leq & C_{p,F }   \| \BX \|_{p,[0,t]}(\|\Delta Y \|_{p,[0,t]} + \|\Delta R^Y \|_{\frac p 2,[0,t]} ).
\end{eqnarray*}
To obtain the contraction, one can use the same strategy as in Young's case, that is proving a contraction under a weaker topology($p'$-variation with $p'>p$) and then using the local estimate for rough integral to show the solution is in the original space $V^p$. Here we follow another strategy. Define an equivalent norm on $\MV_{\BX}^p,$
$$
\|Y,Y'\|^{(\delta)}_{p,[0,t]}:= |Y_0|+|Y'_0|+ \| Y' \|_{p,[0,t]} + \delta \|R^Y \|_{\frac p 2,[0,t]}, \ \ \ \delta>1.
$$
Then one obtains that for any $(Y,Y'),(\tY,\tY') \in \Omega_t,$
\begin{eqnarray*}
\|Z-\tZ,Z'-\tZ' \|^{(\delta)}_{p,[0,t]} &\leq&  C_{p,F} [(1+\delta )\|\BX \|_{p,[0,t]} \|\Delta Y'\|_{p,[0,t]} + \frac{1+\delta \|\BX\|_{p,[0,t]}}{\delta} \delta \|\Delta R\|_{\frac p 2,[0,t]}]\\
&\leq&  C_{p,F} ((1+\delta )\|\BX \|_{p,[0,t]} \vee \frac{1+\delta \|\BX\|_{p,[0,t]}}{\delta} )\|\Delta Y , \Delta Y' \|^{(\delta)}_{p,[0,t]}.
\end{eqnarray*}
Choose $\delta $ large and $t=T_2$ small such that $ C_{p,F} ((1+\delta )\|\BX \|_{p,[0,t]} \vee \frac{1+\delta \|\BX\|_{p,[0,t]}}{\delta} )< 1,$ and let $t_1=T_1 \wedge T_2, $ and then the fixed point theorem is built. Now suppose $Y,\tY$ solve equations driven by $\BX, \tBX$ on $\Omega_{t_1}, \tilde{\Omega}_{t_1}$ respectively. We have implicitly assumed $\BX,\tBX,K_Y,K_{\tY}$ bounded locally as above. Since $(I_{\BX}(F(Y)),F(Y))$ is still a controlled rough path, one has
\begin{eqnarray*}
|(Y-\tY)_{s,t} | &=& |I_{\BX}(F(Y))_{s,t}-I_{\tBX}(F(\tY))_{s,t} \mp F(Y_s)X_{s,t} \pm F(\tY_s)\tX_{s,t}| \\
&\leq & |R^{I_{\BX}(F(Y))}_{s,t}-R^{I_{\tBX}(F(\tY))}_{s,t}| + |F(Y_s)X_{s,t}-F(\tY_s)\tX_{s,t}|, \ \ s<t<t_1.
\end{eqnarray*}
Applying Lemma \ref{contraction} again, one obtains for $t<t_1,$
\begin{eqnarray*}
&&\|\Delta Y  \|_{p,[0,t]}  +   \|\Delta R^{I_{\BX}(F(Y))}   \|_{\frac p 2, [0,t]} \\
&\leq& C_{p} \|\Delta R^{I_{\BX}(F(Y))}  \|_{\frac p 2, [0,t]} + C_p \|F \| ( \| \Delta X\|_{p,[0,t]}+ \|X\|_{p,[0,t]} (|\Delta Y_0| + \|\Delta Y  \|_{p,[0,t]}))\\
&\leq& C_p \|F\| (\|\Delta \BX \|_{p,[0,t]}+ |\Delta Y_0|) + C_{p,F} \|\BX \|_{p,[0,t]} (\|\Delta Y  \|_{p,[0,t]}  +   \|\Delta R^{I_{\BX}(F(Y))} ),
\end{eqnarray*}
so one may choose $t=t_2$ small enough that $C_{p,F} \|\BX \|_{p,[0,t_2]} \leq \frac12,$ which implies our conclusion.

\end{proof}

As mentioned in the Young case, the (local existence) argument requires $\|\BX\|_{p,[0,t_1]} < \delta_{p,F}$, for a constant $\delta_{p,F}$ by making $t_1>0$ small enough. On the other hand,
if there is a ``large'' jump of $\BX$ at $t=\tau \in (0,T)$, the Picard iteration may get stuck on $[0,\tau)$. Fortunately, there are only finitely many large jumps (since $\BX \in \mathbf{V}^p$, hence regulated) which
may be handled "by hand'' such as to obtain a global solution. For c\`adl\`ag  one so obtains


%

\begin{thm}(\textbf{global solution for D.E. driven by c\`adl\`ag rough paths})\label{global sol.}
Suppose $\BX$ is a level$-2$ c\`adl\`ag rough path on $[0,T]$ with finite $p$-variation and $F \in C_b^3.$ Then there exists a unique c\`adl\`ag solution $Y_t \in V^p([0,T],\R^e)$ solving the following equation
\begin{equation}
Y_t=\int_0^t F(Y_s)^- d\BX_s + y_0,
\end{equation}
in the sense that the integration is considered as Proposition \ref{level-2 rp} with $(Y,F(Y))$ as controlled rough paths. Furthermore, suppose $\|\BX\|_{p,[0,T]},\|\tBX\|_{p,[0,T]}<L,$ and $Y,\tY$ solve equations driven by $\BX, \tBX$ respectively on $[0,T].$ Then one has the following local estimate,
$$
\|Y-\tY\|_{p,[0,T]} \leq M^2 (C_{p,F}(1\vee L))^{M+1} ( \|\BX ;\tBX \|_{p,[0,T]} + |Y_0-\tY_0| ),
$$
where $M$ is a constant bounded by $C_{p,F} L^p +1.$

\end{thm}

\begin{proof}

For the existence, according to Theorem \ref{solve level-2} and the remarks afterwards, 
there exists a constant $\delta_1,$ depending on $p$ and $F,$ such that given any $\xi \in \R^d, $ if $\|\BX \|_{p, [u,v)}:=\lim_{t\uparrow v}\|\BX \|_{p, [u,t]}< \delta_1 ,$ then there exists a unique solution for equation
\begin{equation}
Y_t=\xi + \int_u^t F(Y_r)^\ell d\BX_r, \ \ \ t\in [u,v).
\end{equation}
According to Lemma \ref{regular path} and Lemma \ref{partition}, there exists a partition $\op:=\{0=t_0<t_1<...<t_N=T\} $ of $[0,T],$ such that for any $(t_i, t_{i+1}),$ $\omega_{\BX}(t_i+,t_{i+1}-)< \delta_1^p.$ By right-continuity, one can solve the above equation on every $[t_i,t_{i+1} )$ given any initial condition denoted as $\xi_{i}^0,i=0,...,N-1.$ Now we only need to give the value of $\{\xi_{i}^0\}_{i=0}^N$ according to our equation. Given $\xi_i,i=0,...,N-1,$ denote the solution on every $[t_i,t_{i+1})$ as $Y^{\xi_i}_t,t\in [t_i,t_{i+1}).$ Naturally, let $\xi_0^0:=y_0 $ and define $\xi_{i+1}^0,i=0,...,N-1,$ by
$$
\xi_{i+1}^0:= Y^{\xi_i^0}_{t_{i+1}-} + F(Y^{\xi_i^0}_{t_{i+1}-}) X_{t_{i+1}-,t_{i+1}} +DF(Y^{\xi_i^0}_{t_{i+1}-}) F(Y^{\xi_i^0}_{t_{i+1}-}) \X_{t_{i+1}-,t_{i+1}}.
$$
Define $Y_t:= Y_t^{\xi_i^0}, t\in [t_i,t_{i+1})$ with $\xi_N^0$ as $Y_T.$ Then one can check $Y_t$ satisfies equation \eqref{level-2 equa}. The uniqueness follows from the local uniqueness. For the local Lipschitzness part, according to Theorem \ref{solve level-2}, there exists a constant $\delta_2$ depending on $p$ and $F, $ such that as long as $\|\BX \|_{p,[u,v)}\vee \|\tBX\|_{p,[u,v)}<\delta_2<1,$ then $M_Y$ and $M_{\tY}$ are bounded by $2 $ and the local Lipschitz estimate holds, i.e.
\begin{eqnarray}\label{local est}
\|Y-\tY \|_{p,[u,v)} \leq C_p \|F \|(\|\Delta \BX \|_{p,[u,v)} + |\Delta Y_u|).
\end{eqnarray}

Since $\omega_\BX(0,t),\omega_{\tBX}(0,t)$ are increasing and bounded by $L^p,$ by patching two partitions with at most $([\frac{L^p}{\delta_2^p}]+2)$ endpoints, one may build a partition $\oq:=\{0=s_0<s_1<...<s_M=T\} $ of $[0,T],$ such that $\omega_\BX(s_i+,s_{i+1}-)\vee \omega_{\tBX}(s_i+,s_{i+1}-) < \delta_2^p$ and $M< 2\frac{L^p}{\delta_2^p} +1, $ which implies, by right-continuity, one has the above local estimate on every $[s_i,s_{i+1}).$ We claim that one furthermore has
\begin{eqnarray}\label{local est2}
\|Y-\tY \|_{p,[s_i,s_{i+1}]} \leq C_p \|F \| (\|F\|\vee 1 ) (L \vee 1) (\|\Delta \BX \|_{p,[s_i,s_{i+1}]} + |\Delta Y_{s_i}|).
\end{eqnarray}
Indeed, this follows from the fact $\|g(\cdot)\|_{p,[u,v]} = \|g(\cdot)\|_{p,[u,v)}+ |g(v)-g(v-)|$ and inequalities
\begin{eqnarray*}
&&|\Delta F(Y_{s_{i+1}- })X_{s_{i+1}-,s_{i+1}} + \Delta DF(Y_{s_{i+1}- })F(Y_{s_{i+1}- })\X_{s_{i+1}-,  s_{i+1} }|\\
&\leq & \|F\| (|\Delta Y_{s_{i+1}- } |(|X_{s_{i+1}-,s_{i+1}}|+ |\X_{s_{i+1}-,s_{i+1}}|)+  |\Delta X_{s_{i+1}-,s_{i+1}}|+ |\Delta \X_{s_{i+1}-,s_{i+1}}| )\\
&\leq & \|F\|( (|\Delta Y_{s_i}|+|\Delta Y|_{p,[s_i,s_{i+1})})L + \|\Delta \BX \|_{p,[s_i,s_{i+1}]} )\\
&\leq & C_p \|F\|(\|F\|\vee 1 ) (L \vee 1) (|\Delta Y_{s_i} | + \|\Delta \BX \|_{p,[s_i,s_{i+1}]} ).
\end{eqnarray*}
where we apply estimate \ref{local est} in the last inequality. Specially, one obtains
$$|\Delta Y_{s_{i+1}}| \leq |\Delta Y_{s_{i}}| + \|\Delta Y \|_{p,[s_i,s_{i+1}]} \leq  C_p (1\vee \|F\|)^2 (L \vee 1) ( \|\Delta \BX \|_{p,[s_i,s_{i+1}]}+ |\Delta Y_{s_i} |).
$$
Denote $\alpha:=    (1\vee \|F\|)^2  (L\vee 1),$ one has the uniform bound for initial conditions,
\begin{eqnarray*}
|\Delta Y_{s_{i+1}}|^p &\leq& (2 C_p \alpha)^p ( \|\Delta \BX \|_{p,[s_i,s_{i+1}]}^p + |\Delta Y_{s_i} |^p )\\
 &\leq& (2 C_p \alpha  )^{2p} (\|\Delta \BX \|_{p,[s_i,s_{i+1}]}^p +\|\Delta \BX \|_{p,[s_{i-1},s_{i}]}^p + |\Delta Y_{s_{i-1}}|^p )\\
  &\leq& ... \leq   (2 C_p \alpha)^{pM}(\sum_{i=0}^{M-1}\|\Delta \BX \|_{p,[s_i,s_{i+1}]}^p + |\Delta Y_0|^p),
\end{eqnarray*}
which implies $|\Delta Y_{s_{i+1}}|< (C_p \alpha)^{M} (\|\Delta \BX \|_{p,[0,T]} + |\Delta Y_0| ).$ Plug this bound to estimate \eqref{local est2}, one obtains for any $i=0,...,M-1,$
$$
\|Y-\tY \|_{p,[s_i,s_{i+1}]} \leq (C_p \alpha)^{M+1} \|F\| (\|\Delta \BX \|_{p,[0,T]} + |\Delta Y_0|).
$$
Then the estimate follows from the fact $\|g\|_{p,[0,T]} \leq M \sum_{i=0}^{M-1} \|g\|_{p,[s_i,s_{i+1}]}.$

\end{proof}

%
%
%

In a final step, we show how to discard of the c\`adl\`ag assumption.

\begin{thm}(\textbf{global solution for D.E. driven by general rough paths})\label{global sol. for regular}
Suppose $\BX$ is a level$-2$ rough path on $[0,T]$ with finite $p$-variation and $F \in C_b^3.$ Then there exists a unique solution $Y_t \in V^p([0,T],\R^e)$ solving Equation \eqref{level-2 equa} in the sense that the integration is considered as Example \ref{level-2 rp} with $(Y,F(Y))$ as controlled rough paths.

\end{thm}

\begin{proof}

We apply the same strategy as the c\`adl\`ag case, solving the equation locally and extending it to the global solution. Define $$Y_{0+}:= \xi + F(\xi)X_{0,0+} + DF(\xi)F(\xi) \X_{0,0+},$$ and $\tBX:=(\tX, \tilde{\X})$ with
\[
\tX_t:=\left\{
\begin{array}{ll}
X_{0+}, & t=0 \\
X_t,& t>0,
\end{array}
\right.
\ \ \
\tilde{\X}_{s,t}:=\left\{
\begin{array}{ll}
\X_{0+,t}, & s=0 \\
\X_{s,t},& s>0.
\end{array}
\right.
\]
Then one can check that $\tBX_{0,t}$ is right-continuous at $t=0$ and $\tBX$ satisfies the Chen's relation. Then there is a unique solution to the following RDE driven by $\tBX$ on some $[0,t^*),$
$$
\tY_t = Y_{0+} + \int_0^t F(\tY_r)^\ell d\tBX_r,
$$
which defines a solution of \eqref{level-2 equa} by $Y_t=\tY_t,\ t\in (0,t^*).$ The converse also holds. Hence we proved the local uniqueness and existence. For the global solution, one may apply Lemma \ref{partition}, to obtain the partition $\op=\{0=t_0<t_1<...<t_M=T\}$ such that $\omega_{\BX}(t_i+,t_{i+1}- )$ small enough to support a unique solution of
$$
\tY_t = \tilde{ \xi}_i^0 + \int_0^t F(\tY_r)^\ell d\tBX_r^i, \ \ \ i=0,...,M,
$$
where $\tBX^i$ is defined on $[t_i,t_{i+1})$ by
\[
\tX_t^i:=\left\{
\begin{array}{ll}
X_{t_i+}, & t=t_i, \\
X_t,& t>t_i,
\end{array}
\right.
\ \ \
\tilde{\X}_{s,t}^i:=\left\{
\begin{array}{ll}
\X_{t_i+,t}, & s=t_i \\
\X_{s,t},& s>t_i.
\end{array}
\right.
\]

Then define $\tilde{ \xi}_0^0:= Y_{0+}$ and for $i=0,...,M-1,$
\begin{eqnarray*}
  \xi_{i+1}^0&:=& \tY^{\tilde{ \xi}_i^0}_{t_{i+1}-} + F(\tY^{\tilde{ \xi}_i^0}_{t_{i+1}-}) X_{t_{i+1}-,t_{i+1}} +DF(\tY^{\tilde{ \xi}_i^0}_{t_{i+1}-}) F(\tY^{\tilde{ \xi}_i^0}_{t_{i+1}-}) \X_{t_{i+1}-,t_{i+1}}  ,\\
  \tilde{ \xi}_{i+1}^0&:=& \xi_{i+1}^0 + F(\xi_{i+1}^0) X_{t_{i+1},t_{i+1}+} + DF(\xi_{i+1}^0 ) F(\xi_{i+1}^0) \X_{t_{i+1},t_{i+1}+},
\end{eqnarray*}
where we still denote $\tY^{\tilde{ \xi}_i^0}_t$ as the solution for the RDE driven by $\tBX^i$ on $[t_i,t_{i+1})$ with initial condition $\tilde{\xi}_i^0.$ Then define $Y$ on every $[t_i,t_{i+1})$ by
\[
Y_t:=\left\{
\begin{array}{ll}
\xi_{i}^0, & t=t_i, \\
\tY_t,& t\in (t_i,t_{i+1}),
\end{array}
\right.
\]
and check it defines a unique solution to the original RDE driven by $\BX.$
\end{proof}


\section{Branched rough paths} \label{sec:bRP}

(Continuous) branched rough paths were understood in \cite{Gub10,HK15} as the correct framework for paths of arbitrary ($1/p$-H\"older) roughness in absence of a chain-rule. The latter makes them immediately relevant for discontinuous paths, and it remains to revise the theory from a jump perspective.


\subsection{Recall on Hopf algebra formalism} We need some notation and take the opportunity to review some elements of the theory for the reaeder's convenience.

\begin{defi}(\textbf{algebra of forest})
Let $\mathcal{T}$ the set of all rooted trees with vertex labeled by a number from set $\{1,2,...,d\} $,
e.g.
\[
\bullet_i, \begin{tikzpicture}[scale=0.2,baseline=0.1cm]
        \node at (0,0)  [dot,label= {[label distance=-0.2em]below: \scriptsize  $ k $} ] (root) {};
         \node at (0,2)  [dot,label={[label distance=-0.2em]above: \scriptsize  $ j $}]  (up) {};
            \draw[kernel1] (root) to
     node {}  (up);
     \end{tikzpicture},
\begin{tikzpicture}[scale=0.2,baseline=0.1cm]
        \node at (0,0)  [dot,label= {[label distance=-0.2em]below: \scriptsize  $ k $} ] (root) {};
         \node at (1,2)  [dot,label={[label distance=-0.2em]above: \scriptsize  $ j $}]  (right) {};
         \node at (-1,2)  [dot,label={[label distance=-0.2em]above: \scriptsize  $ i $} ] (left) {};
            \draw[kernel1] (right) to
     node [sloped,below] {\small }     (root); \draw[kernel1] (left) to
     node [sloped,below] {\small }     (root);
     \end{tikzpicture},...,
    \ \ \  i,j,k\in\{1,...,d\}.
\]

Define a commutative product (or multiplication) ``$\cdot$'' on $\mathcal{T}$ and denote the set of forest as $\mathcal{F}:=\{\tau_1 \cdot \tau_2 \cdots \tau_n| \tau_i \in \mathcal{T},i=1,...,n, n\in \mathbb{N}\}.$ Pick a ``unit'' element $\mathbf{1}$ such that $\mathbf{1} \cdot f:=f \cdot \mathbf{1} :=f, \ \forall f\in \mathcal{F}$ and denote $\mathcal{H}$ the linear expansion of $\{\mathbf{1},\mathcal{F}\}.$ One may check that $\mathcal{H}$ is an algebra with the linear ``unit'' mapping
\[
\begin{array}{llll}
\eta: & \R  & \longrightarrow & \mathcal{H}  \\
      & 1 &  \longrightarrow & \mathbf{1}.
\end{array}
\]

\end{defi}

In the following we omit ``$\cdot $'' in elements of forest and denote any tree as $\tau=[\tau_1,...,\tau_k]_i$ where the letter $i$ represents the label on the root and $\tau_1,...,\tau_k$ are smaller labeled trees directly connected to the root $\bullet_i,$ and $[\tau_1,...,\tau_{k-1},\BI,\tau_k,...]_i=[\tau_1,...,\tau_{k-1},\tau_k,...]_i$ for latter convenience. Specially, $[\BI]_i=\bullet_i.$ Now we equip the algebra of labeled forest with a co-multiplication $\Delta:\MH \rightarrow \MH \otimes \MH ,$ which is also a morphism with respect to its multiplication(i.e. $\delta(h_1 h_2 )=\delta(h_1) \tilde{\cdot} \delta(h_2)$ with $\tilde{\cdot}$ the multiplication on $\MH \otimes \MH$ defined in the natural way) so our algebra is also a coalgebra(i.e. a vector space with co-multiplication and counit) with the following obvious linear ``counit'' mapping

\begin{eqnarray*}
&&\vep:  \MH   \longrightarrow  \R , \\
&&\vep(\tau)=  \left\{
 \begin{array}{ll}
  1, &\text{ if } \tau = \BI, \\
  0, & \text{ if } \tau \in \MF.
\end{array}
\right.
\end{eqnarray*}

\begin{defi}
The comultiplication $\Delta:\MH \rightarrow \MH \otimes \MH $ is defined recursively in the following way. For any $\tau=[\tau_1,...,\tau_n]_i \in \MT,$

\begin{enumerate}
\item  $\Delta \BI := \BI \otimes \BI  ,$\\
\item   $\Delta [\tau_1,...,\tau_n]_i := [\tau_1,...,\tau_n]_i \otimes \BI + \sum_{(\tau_1) \cdots (\tau_n)} (\tau_1^{(1)}\cdots \tau_n^{(1)}) \otimes [\tau_1^{(2)},\cdots, \tau_n^{(2)}]_i,$
\end{enumerate}

where we use the Sweedler notation $\Delta \tau = \sum_{(\tau)} \tau^{(1)} \otimes \tau^{(2)}$ and $(\tau_1^{(1)}\cdots \tau_n^{(1)})$ is the forest by multiplication of $\tau_1^{(1)},...,\tau_n^{(1)}.$ Then one extends $\Delta$ to $\MH$ by morphism and liniarity, and obtain the bialgebra(i.e. both an algebra and a co-algebra with co-multiplication as a morphism w.r.t. multiplication) $(\MH,\cdot,\Delta,\eta,\vep ).$
\end{defi}


\begin{example}\label{trees}
According to our definition of $\Delta,$
\begin{eqnarray*}
\Delta \bullet_i &=& \Delta[\BI]_i=\bullet_i \otimes \BI + \BI \otimes [\BI]_i = \bullet_i \otimes \BI + \BI \otimes \bullet_i,\\
\Delta \tikz[scale=0.2,baseline=0.1 cm]
        {\node at (0,0)  [dot,label= {[label distance=-0.2em]below: \scriptsize  $ k $}] (root) {};
         \node at (0,2)  [dot,label={[label distance=-0.2em]above: \scriptsize  $ j $}]  (up) {};
            \draw[kernel1] (root) to node [sloped,below] {\small }  (up);} &=& \ru{k}{j} \otimes \BI +\bullet_j \otimes [\BI]_k + \BI \otimes [\bullet_j]_k= \ru{k}{j} \otimes \BI +\bullet_j \otimes \bullet_k + \BI \otimes \ru{k}{j},\\
\Delta \rlr{k}{j}{i}&=& \rlr{k}{j}{i} \otimes \BI + (\bullet_i \BI) \otimes [\BI, \bullet_j]_k + (\BI \bullet_j) \otimes [\bullet_i,\BI]_k + (\bullet_i \bullet_j)\otimes \bullet_k + \BI \otimes \rlr{k}{j}{i}\\
&=& \rlr{k}{j}{i} \otimes \BI + \bullet_i\otimes \ru{k}{j} + \bullet_j \otimes  \ru{k}{i} + \bullet_i \bullet_j \otimes \bullet_k + \BI \otimes \rlr{k}{j}{i}.
\end{eqnarray*}
 (We see that $\Delta$ has an interpretation in terms of ``admissible cuts''.)
\end{example}

%
%
%

\begin{prop}(\textbf{Connes-Kreimer Hopf algebra of forest}) \label{hopf alg}
For the bialgebra $(\MH,\cdot,\Delta,\eta,\vep ), $ there exists an endomorphism(i.e. linear map of vector spaces) $S: \MH \longrightarrow \MH,$ such that
$$
\cdot \circ (Id \otimes S) \circ \Delta (h)=  \eta \circ \vep (h) ,\label{antipode}
$$
so our bialgebra is a Hopf algebra and $S$ is called the antipode of the Hopf algebra.
\end{prop}

\begin{proof}
By anti-homomorphism of the antipode(if exists, i.e. $S(x\cdot y)=S(y)S(x),$ in our commutative case, homomorphism), we only need to construct $S$ on $\MT.$ Notice that $\MT$ is a graded set w.r.t. the number of vertexes, so the antipode could be built inductively. According to the identity for the antipode, one needs
$$
\cdot \circ (Id \otimes S) \circ \Delta (\BI)= \BI \cdot S(\BI)=S(\BI)= \eta \circ \vep (\BI),
$$
so $S(\BI):=\BI.$ Suppose for any tree $\tau^k$ with $k$ vertexes or less, $S(\tau^k)$ is already defined. For any tree $\tau^{k+1}=[\tau_1,...,\tau_m]_i$ with $k+1$ vertexes, according to the definition of $\Delta,$ one obtains
$$
S(\tau^{k+1})= \eta \circ \vep (\tau^{k+1}) - \sum_{(\tau_1),...,(\tau_m)}  \tau_1^{(1)}\cdots \tau_m^{(1)} \cdot S([\tau_1^{(2)},...,\tau_m^{(2)}]_i),
$$
which is well defined since each $\tau_1,...,\tau_m$ has no more than $k$ vertexes. Then one may extend $S$ to $\MH$ by homomorphism and linearity.

\end{proof}

\begin{rem}
The antipode gives the inverse of a group-like element in a Hopf algebra. More precisely, the set $G(\MH):=\{h\in \MH| \vep(h)=1, \Delta h=h\otimes h\}$ is called group-like elements, and one can check that $(G(\MH),\cdot)$ is indeed a group with the inverse mapping given by $h^{-1}=S(h),\ \ \forall h\in G(\MH).$
\end{rem}

We are interested in the truncated Hopf algebra. Denote $\MT_n:=\bigcup_{k=0}^n \MT_{(k)},$ where $\MT_{(0)}=\BI$ and $\MT_{(k)}$
represents trees with exact $k$ vertexes. For $\tau\in \MT,$ denote $|\tau|$ as the number of vertex of $\tau.$ $\MF_{(k)}$ represents
forests $f=\tau_1\cdots \tau_m$ with vertex number $|f|=\sum_{i=1}^m |\tau_i|=k,$ and $\MH_{(k)}$ is the linear expansion of $\{  \MF_{(k)}\}.$ $\MF_n,\MH_n$ are defined similarly.
The truncated Hopf algebra gives its finite dual bialgebra(also a Hopf algebra) by duality, which can be viewed as the image space of driven signals.
More precisely, denote $\MM^n:=\bigoplus_{k=n+1}^\infty \MH_{(k)}$ a module of $\MH,$ and $ \MH^*_n:= \{l \in \MH^*| l(\MM^n)=0 \}$
which is a finite dimensional vector space and generated by $\{f^*| f\in \MF_n\}.$ Define

$$
\MH^\circ:=\{l\in \MH^*| \exists n\in \mathbb{N} ,l(\MM^n)=0\}= \bigcup_{n=0}^\infty \MH^*_n
$$

\begin{prop}(\textbf{the finite dual Hopf algebra })

Define operators $\star,\delta, S^*$ on $\MH^*$ dual operators of $\Delta, \cdot, S $ i.e., for any $l_1, l_2, l \in \MH^*$ and $f,f_1,f_2 \in \MH_n,$
\begin{eqnarray*}
  && \langle l_1 \star l_2, f \rangle:= \langle l_1  \otimes l_2  , \Delta f \rangle, \\
  &&\langle \delta l , f_1 \otimes f_2\rangle:=\langle l, f_1 f_2\rangle,\\
  &&\langle S^*(l),f\rangle : = \langle l, S(f)\rangle.
\end{eqnarray*}
Then one defines $\eta^* $ and $\vep^*$ as the unit map and counit map in the natural way similar as $\eta,\vep.$ Then $(\MH^\circ, \star, \delta, S^*, \eta^*, \vep^*)$ is also a Hopf algebra with an antipode $S^*.$

\end{prop}

\begin{proof}

Indeed, everything is easy to check by duality and definition except the closedness, i.e. for any $l_1,l_2,l \in \MH^\circ,$ $l_1 \star l_2 \in \MH^\circ$ and $\delta l \in \MH^\circ \otimes \MH^\circ.$ For any $l_1,l_2 \in \MH^\circ,$ by our representation of $\MH^\circ,$ there exists $\MH_{n_1}^*,\MH_{n_2}^*$ such that $l_1\in \MH_{n_1}^*, l_2 \in \MH_{n_2}^*.$ Then for any $f\in \MF_{(n)}$ with $n\geq n_1+n_2+3,$
$$
\langle l_1 \star l_2 , f \rangle = \langle l_1  \otimes l_2  , \Delta f \rangle=\sum_{(f)} l_1(f^{(1)}) l_2(f^{(2)})=0,
$$
which implies $l_1 \star l_2 \in \MH_{n_1+n_2+2}^*.$ For any $l\in \MH^\circ,$ there exists $\MH_k^*,$ such that $l\in \MH_k^*.$ According to the definition of $\delta,$ one has $\delta l \in (\MH \otimes \MH)^*$(notice $\MH^*\otimes \MH^* \subsetneq (\MH \otimes \MH)^*,$ which is why $\MH^*$ is not a bialgebra and one needs $\MH^\circ$). We claim that $\delta l \in \MH_{k}^* \otimes \MH_{k}^* \subseteq \MH^\circ \otimes \MH^\circ.$ Indeed, for any $f_1 ,f_2 \in \MM^k, f\in \MH,$ one has
$$
\langle \delta l , f_1 \otimes f \rangle =\langle  l , f_1   f \rangle=0,\ \ \ \  \langle \delta l , f \otimes f_2 \rangle =\langle  l ,    f f_2 \rangle=0,
$$
which implies our claim.

\end{proof}

\begin{rem}\label{star op}
As the dual operator of $\Delta,$ we see that $\tau_1^* \star \tau_2^*$, for dual trees $\tau_1^*,\tau_2^*$ can be given by $ (\tau_1 \tau_2)^* + \sum_\tau \tau^*,$ with sum is taken over all trees obtained by attaching $\tau_1$ to any vertex of $\tau_2.$
\end{rem}

\subsection{Branched rough paths with jumps}

As we implied before, the branched rough path takes values in $\MH^*.$ We are going to extend the classical ``Chen's relation'' to a group property. Denote $G(\MH^*):=\{l\in \MH^*| \delta l= l\otimes l, l(\BI)=1 \}$ as the group-like elements of $\MH^*,$ and Hom$(\MH,\R)$ the set of homomorphisms(characters) in $\MH^*,$ i.e. $\langle  l, h_1 h_2  \rangle=\langle  l, h_1  \rangle \langle  l, h_2  \rangle.$ Then one actually have $G(\MH^*)=\text{Hom}(\MH,\R).$ Indeed, $l\in \text{Hom}(\MH,\R)$ if and only if $\langle \delta l, h_1 \otimes h_2\rangle=\langle  l, h_1 h_2\rangle=\langle  l, h_1  \rangle \langle  l, h_2  \rangle=\langle l\otimes l, h_1 \otimes h_2\rangle.$ Furthermore, $(G(\MH^*),\star)$ is a group with inverse mapping $S^*,$ which is known as the Butcher group. Now we denote
$$
G_N(\MH^*):=G(\MH^*)/\{ \cup_{k=N+1}^\infty \MF^*_{(k)}\} \backsimeq G(\MH^*) |_{\MH_N},
$$
which means for any $g\in G_N,$ and any $f  \in \cup_{k=N+1}^\infty \MF_{(k)},$ $g(f)=0.$ In the following we may abuse the notation according to the above isomorphism.

\begin{defi}(\textbf{branched $p$-rough path})
A branched $p$-rough path is a mapping $\BX:[0,T] \rightarrow G_{[p]}(\MH^*),$ such that
for any $f\in \MF_{[p]} , $
$$
\|\BX^f\|_{\frac{p}{|f|},[0,T]}:= \| \langle \BX, f\rangle \|_{\frac{p}{|f|},[0,T]} :=\{ \sup_{\op} \sum_{(s,t)\in \op} |\langle \BX_{s,t}, f\rangle|^{\frac{p}{|f|}}\}^{\frac{|f|}{p}}<\infty,
$$
where $\BX_{s,t}:=\BX_s^{-1} \star \BX_t$ and $\op$ is any partition of $[0,T].$ Denote the following as the control of the branched rough path,
$$
\omega_{\BX,p}(s,t):= \sum_{f\in \MF_{[p]}  } \|  \BX^f   \|_{\frac{p}{|f|},[s,t]}^{\frac{p}{|f|}}.
$$
We also use the following norm,
$$
\|\BX \|_{p,[0,T]}:= \sum_{f\in \MF_{[p]} } \|\BX^f\|_{\frac{p}{|f|},[0,T]}.
$$
\end{defi}

\begin{example}[Canonical branched rough path lift]   \label{rem:cbrpl}
Suppose at first that $x$ is a smooth path in $\R^d.$ Then one may build a branched rough path $\BX_{t}\in \MH^*,\ t\in [0,T]$ over $x$ recursively by defining $\BX_{0}=\BI$ and $\BX_{s}:=\BX_{0,s}$ with
$$
\langle \BX_{s,t}, \bullet_i \rangle := x^i_{s,t}, \ \  \langle \BX_{s,t},[\tau_1\cdots \tau_n]_i \rangle:= \int_s^t \langle\BX_{s,r}, \tau_1 \rangle \cdots \langle\BX_{s,r}, \tau_n \rangle dx^i_r,
$$
where we imply $\BX_t$ is a character. Then one can apply relation $\int_s^t=\int_s^u + \int_u^t $ iteratedly and inductively check
\begin{eqnarray*}
\int_s^t \langle\BX_{s,r}, \tau_1 \rangle \cdots \langle\BX_{s,r}, \tau_n \rangle dx^i_r &=& \int_s^u \langle\BX_{s,r}, \tau_1 \rangle \cdots \langle\BX_{s,r}, \tau_n \rangle dx^i_r\\
 &+& \sum_{(\tau_1)\cdots (\tau_n)}\langle\BX_{s,u},\tau_1^{(1)} \rangle \cdots \langle\BX_{s,u},\tau_n^{(1)} \rangle \int_u^t \langle\BX_{u,r}, \tau_1^{(2)} \rangle \cdots \langle\BX_{u,r}, \tau_n^{(2)} \rangle dx^i_r.
\end{eqnarray*}
For instance, for a linear tree with two vertexes, one has
\begin{eqnarray*}
\langle \BX_{s,t} , \tikz[scale=0.2,baseline=0.1 cm]
        {\node at (0,0)  [dot,label= {[label distance=-0.2em]below: \scriptsize  $ i $}] (root) {};
         \node at (0,2)  [dot,label={[label distance=-0.2em]above: \scriptsize  $ j $}]  (up) {};
            \draw[kernel1] (root) to node [sloped,below] {\small }  (up);}
            \rangle &=& \int_s^t \langle \BX_{s,r}, \bullet_j \rangle dx^i_r = \int_s^t x^j_{s,r} dx^i_r
            =\int_s^u x_{s,r}^j dx_r^i + x_{s,u}^j x_{u,t}^i + 1 \cdot \int_u^t x_{u,r}^j dx^i_r\\
 &=&    \langle \BX_{s,u} \otimes \BX_{u,t}, \Delta  \tikz[scale=0.2,baseline=0.1 cm]
        {\node at (0,0)  [dot,label= {[label distance=-0.2em]below: \scriptsize  $ i $}] (root) {};
         \node at (0,2)  [dot,label={[label distance=-0.2em]above: \scriptsize  $ j $}]  (up) {};
            \draw[kernel1] (root) to node [sloped,below] {\small }  (up);} \rangle =  \langle \BX_{s,u} \star \BX_{u,t},    \tikz[scale=0.2,baseline=0.1 cm]
        {\node at (0,0)  [dot,label= {[label distance=-0.2em]below: \scriptsize  $ i $}] (root) {};
         \node at (0,2)  [dot,label={[label distance=-0.2em]above: \scriptsize  $ j $}]  (up) {};
            \draw[kernel1] (root) to node [sloped,below] {\small }  (up);} \rangle
\end{eqnarray*}
where we only apply additivity of the interval for integration, not the chain rule.\footnote{This construction is a special of the unique ``Lyons-Gubinelli'' extension \cite{Lyo98,Gub10} of a continuous geometric $p$-rough path, here in the case $p=1$.} For this reason, we can re-run the above computation replacing all classical integrals $\int (...) dx $ by
our left-point integral $\int (...)^\ell dx$, now allowing for any $x \in V^1$, i.e. any $x$ of bounded variation. It is worth pointing out that even if $t$ is a jump time (say, $\Delta^-_t x \ne 0$), one has
$$
\lim_{s \to t} \langle \BX_{s,t} , \tikz[scale=0.2,baseline=0.1 cm]
        {\node at (0,0)  [dot,label= {[label distance=-0.2em]below: \scriptsize  $ i $}] (root) {};
         \node at (0,2)  [dot,label={[label distance=-0.2em]above: \scriptsize  $ j $}]  (up) {};
            \draw[kernel1] (root) to node [sloped,below] {\small }  (up);}
            \rangle
\equiv
 \langle \Delta^-_t \BX , \tikz[scale=0.2,baseline=0.1 cm]
        {\node at (0,0)  [dot,label= {[label distance=-0.2em]below: \scriptsize  $ i $}] (root) {};
         \node at (0,2)  [dot,label={[label distance=-0.2em]above: \scriptsize  $ j $}]  (up) {};
            \draw[kernel1] (root) to node [sloped,below] {\small }  (up);}
            \rangle
= 0.
$$
To see this, revert to the simpler notation $\X^{ji}$ and note that
$$
       \Delta^-_t \X^{ji} = \lim_{s \to t} (\X^{ji}_{0,t} - \X^{ji}_{0,s} ) - \lim_{s \to t}  x^j_s ( x^i_t - x^i_s ),
$$
But the first limit is precisely $\Delta^-_t \int x^{j, \ell} dx^i = x^j_{t-} \Delta^-_t x^i$ (this follows from a ``BV miniature'' of Lemma \ref{preserve jumps}), and hence a precise cancellation takes place.
\footnote{What one sees here is that our canoncial lift is the {\it minimal jump extension} \cite{FS17}, here in the sense that $\langle \BX, f \rangle$ has no jumps whenever $|f| > [p]$. We will not develop this point further,
but note that (the uniqueness part of) the Lyons-Gubinelli extension theorem fails in presence of jumps. The notion of minimal jump extension then restores uniqueness, similar to \cite{FS17}.}
\end{example}

%
%
%
%

%
%

\subsection{Integration theory}

Now we introduce the integrand which is known as controlled rough paths. Recall that if $g:\R^d \rightarrow \R^d$ is smooth and $\BX$ is a geometric rough path (see e.g. \cite{LCL07,FV10}), one has the following Taylor's expansion, for any $k<[p]-1, s<t,$
$$
D^k g(x_t) = \sum_{i=k}^ {[p]-1} D^i g(x_s) (\BX^{i-k}_{s,t}) + R^{[p]-k}_{s,t}, \ \ \ |R^{[p]-k}_{s,t}|\leq |x_{s,t}|^{[p]-k},
$$
which is vital for the integration of geometric rough paths. We translate the above identity into tree language to give some insights of controlled rough paths. Denote $\tau_i^*$ as the dual element of a linear tree (i.e. no branch, always labelled by $1$ for simplicity) $\tau_i$ with $i$ vertexes and denote $\langle \tau_i^* , \BZ_s  \rangle := D^ig(x_s).$ Since $\Delta \tau_i$ equals to the summation of all possible pairs of smaller linear trees after an admissible cut, the above equation is encoded into
\begin{eqnarray*}
\langle  \tau_k^*, \BZ_t \rangle &=& \sum_{i=k}^{[p]-1} \langle\tau_i^* ,  \BZ_s\rangle \langle \BX_{s,t}, \tau_{i-k} \rangle + R^{[p]-k}_{s,t},  \\
&=& \sum_{i=k}^{[p]-1} \langle \tau_i^* ,  \BZ_s \rangle \langle \BX_{s,t}\otimes \tau_k^* , \Delta \tau_{i} \rangle + R^{[p]-k}_{s,t},  \\
&=& \sum_{i=0}^{[p]-1} \langle\tau_i^* ,  \BZ_s \rangle \langle \BX_{s,t}\star \tau^*_k, \tau_{i} \rangle + R^{[p]-k}_{s,t},
\end{eqnarray*}
More general, one has the following definition.

\begin{defi}(\textbf{controlled rough paths})
$\BX$ is a branched $p-$rough path. $\BZ:[0,T]\rightarrow \MH_{[p]-1}$ is called a $\BX-$controlled rough path if for any $f^* \in (\MF^0_{[p]-1})^*:=(\MF_{[p]-1})^* \cup\{\BI^*\},$
\begin{equation}\label{crp}
\langle f^* , \BZ_t \rangle= \langle \BX_{s,t} \star f^*, \BZ_s \rangle  + R^{\BZ,f}_{s,t},
\end{equation}
or in a explicit form,
\begin{eqnarray*}
\langle f^* , \BZ_t \rangle &=& \sum_{h \in \MF_{[p]-1}^0} \langle h^*,  \BZ_s  \rangle \langle \BX_{s,t}\star f^* , h \rangle + R^{\BZ,f}_{s,t}\\
&=& \langle f^* , \BZ_s \rangle + \sum_{[p]-1 \geq |h|>|f|} \langle h^*,  \BZ_s  \rangle \langle \BX_{s,t}\star f^* , h \rangle + R^{\BZ,f}_{s,t},
\end{eqnarray*}
where $R^{\BZ,f}$ are functions on the simplex and satisfy $\sum_{f\in \MF_{[p]-1}^0} \|R^{\BZ,f}\|_{ {p}/(\gamma-|f| ) ,[0,T]}^{ {p}/(\gamma-|f|)} \leq \omega_{\BZ,\gamma}(0,T),$ for some control $\omega_{\BZ,\gamma}$ with $\gamma \geq [p].$ In particular, denote $\BZ^h_s:=\langle h^* , \BZ_s \rangle$ and $z_t:=\langle \BI^*, \BZ_t \rangle,$ and one has
\begin{equation}\label{1crp}
z_{s,t}= \sum_{h \in \MF_{[p]-1}} \BZ^h_s \langle  \BX_{s,t} , h \rangle + R^{\BZ,\BI}_{s,t},
\end{equation}
and one calls $\BZ$ a controlled rough path above $z.$ In the following, we usually take $\gamma=[p]$ for simplicity.

\end{defi}

\begin{example}
Consider the case of $[p]=2 $ and assume $d=1$ so no decoration on vertexes. In this case, $\MF=\{\bullet, \bullet \bullet, \tikz[scale=0.2,baseline=0.1 cm]
        {\node at (0,0)  [dot,label= {[label distance=-0.2em]below:  }] (root) {};
         \node at (0,2)  [dot,label={[label distance=-0.2em]above:  }]  (up) {};
            \draw[kernel1] (root) to node [sloped,below] {\small }  (up);} \}.$
\eqref{crp} is equivalent to
\begin{eqnarray*}
z_t &:=& \langle \BI^*, \BZ_t \rangle = \langle \BI^*, \BZ_s \rangle + \langle \bullet^*, \BZ_s \rangle \langle \BX_{s,t} \star \BI^*, \bullet \rangle + R^1_{s,t}\\
                                 &= & z_s + Z'_s x_{s,t} + R^1_{s,t},\\
Z'_t &:=& \langle \bullet^*, \BZ_t \rangle = \langle \bullet^*, \BZ_s \rangle + R^2_{s,t} = Z'_s+R^2_{s,t},
\end{eqnarray*}
with $R^1_{s,t} \in V^{\frac p 2},R^2_{s,t} \in V^{ p } ,$ which is exactly our definition for controlled rough paths in the level-$2$ case.

\end{example}

\begin{rem}

It is natural to adapt the definition of controlled rough paths above $\R-$valued paths to $\R^e-$valued ones. Indeed, one only needs to consider $\BZ:[0,T]\rightarrow (\MH_{[p]-1})^e$ by taking every component as a $\MH_{[p]-1}-$ valued controlled rough path.

\end{rem}

It is obvious that the space of controlled rough paths is a vector space. One may equip it with the following norm
$$
\|\BZ \|_p:= |\BZ_0| + \sum_{f \in \MF_{[p]-1}^0} \|R^{\BZ,f}\|_{\frac{p}{[p]-|f|}},
$$
and indeed it is a Banach space, denoted as $\MV_{\BX}^p([0,T],(\MH_{[p]-1})^e).$ In particular, one has the following bound
\begin{equation}\label{bound for unif. norm}
\|\BZ^f_. \|_{\sup,[0,T]} \leq |\BZ_0^f|+ \|\BX\|_{p,[0,T]}\sum_{|h|>|f|} |\BZ_0^h| + \|R^{\BZ,f}\|_{\frac{p}{[p]-f},[0,T]}.
\end{equation}
We also use notation
$$
R^{\BZ,k}_{s,t}:=  \max_{|f|, l \leq k}  (|R^{\BZ,f}_{s,t}| + |z_{s,t}|^{[p]-l}), \ \ k=0,...,[p]-1.
$$
It is obvious that $R^{\BZ,k}  $ has finite $\frac{p}{[p]-k}-$variation, so we use notation
$$
\|R^{\BZ,k}_{s,t}\|_{\frac{p}{[p]-k},[0,T]}:= \max_{|f| \leq k} \|R^{\BZ,f}\|_{\frac{p}{[p]-|f|},[0,T]} + \|z\|_{p,[0,t]}.
$$

\begin{thm}(\textbf{integration for branched rough paths})\label{integ.for b.r.p}
Suppose $\BX$ is a branched $p-$rough path and $\BZ$ is a $\BX-$controlled rough path. Let
$$
\Xi^i_{s,t} := \sum_{h \in \MF^0_{[p]-1}} \BZ_s^h \langle \BX_{s,t}, [h]_i \rangle
$$
Then one has
$$
\int_0^T \BZ_s^\ell d \BX_s^i:= \I (\Xi^i)_{0,T} = RRS- \lim_{|\op|\rightarrow 0} \sum_{(s,t)\in \op} \Xi^i_{s,t}
$$
exists, and
$$
|\int_s^t \BZ_s^\ell d \BX_s^i -\Xi_{s,t}^i| \leq C \sum_{h \in \MF^0_{[p]-1}} \|R^h\|_{\frac{p}{\gamma-|h|},[s,t)} \|\BX^{[h]_i}\|_{\frac{p}{(|h|+1)},(s,t]},
$$
with $C$ depending on $\gamma $ and $p.$
In particular, if $\omega_{\BX,p}(s+,t)=\omega_{\BX,p}(s,t)$ or $\omega_{\BZ,\gamma}(s,t-)=\omega_{\BZ,\gamma}(s,t),$ the convergence holds in MRS sense.
\end{thm}

\begin{proof}

Indeed, according to the definition of controlled rough paths and $\BX_{s,t}=\BX_{s,u}\star \BX_{u,t}$, one has for any $s<u<t,$
\begin{eqnarray*}
&&|\Xi^i_{s,t}-\Xi^i_{s,u}-\Xi^i_{u,t}|\\
& =&  |\sum_{h \in \MF^0_{[p]-1}} \BZ_s^h \langle \BX_{s,t}, [h]_i \rangle -\sum_{h \in \MF^0_{[p]-1}} \BZ_s^h \langle \BX_{s,u}, [h]_i \rangle -\sum_{h \in \MF^0_{[p]-1}} \BZ_u^h \langle \BX_{u,t}, [h]_i \rangle|\\
&=& |\sum_{h \in \MF^0_{[p]-1}} \BZ_s^h \langle \BX_{s,t}, [h]_i \rangle -\sum_{h \in \MF^0_{[p]-1}} \BZ_s^h \langle \BX_{s,u}, [h]_i \rangle - \sum_{h \in \MF^0_{[p]-1}} \big(\sum_{f \in \MF^0_{[p]-1}} \BZ_s^f \langle \BX_{s,u} \star h^*, f  \rangle| + R^h_{s,u} \big) \langle \BX_{u,t}, [h]_i \rangle|\\
&\leq & |\sum_{f \in \MF^0_{[p]-1}} \BZ_s^f \big(\langle \BX_{s,t}, [f]_i \rangle- \langle \BX_{s,u}, [f]_i \rangle- \sum_{h \in \MF^0_{[p]-1}} \langle \BX_{s,u}\otimes h^*, \Delta f \rangle   \langle \BX_{u,t} , [h]_i \rangle \big)| +  \sum_{h \in \MF^0_{[p]-1}} |R^{h}_{s,u} \langle \BX_{u,t} , [h]_i \rangle|\\
&\leq & |\sum_{f \in \MF^0_{[p]-1}} \BZ_s^f \big(\langle \BX_{s,t}, [f]_i \rangle- \langle \BX_{s,u}, [f]_i \rangle- \sum_{(f)} \langle \BX_{s,u}, f^{(1)} \rangle \langle \BX_{u,t} , [f^{(2)}]_i \rangle \big) | \\
& & + \sum_{h \in \MF^0_{[p]-1}} \|R^h\|_{\frac{p}{\gamma-|h|},[s,u]} \|\BX^{[h]_i}\|_{\frac{p}{(|h|+1)},[u,t]}\\
&\leq& \sum_{h \in \MF^0_{[p]-1}} \omega_{\BZ,\gamma}^\frac{\gamma-|h|}{p}(s,u) \omega_{\BX,p}^\frac{|h|+1}{p}(u,t).
\end{eqnarray*}
Then according to Theorem \ref{general sew}, one has the convergence and local estimate.

\end{proof}

\begin{rem}(\textbf{stability of c.r.p. under integration})\label{integral as crp}
Similar as the level-$2$ rough path case, given a controlled rough path $\BZ,$ the integration $\I_{\BX^i}(Z)$ could be lifted to a controlled rough path $\BBI_{\BX^i}(\BZ)$. Let
\begin{eqnarray*}
&&\langle \BI , \BBI_{\BX^i}(\BZ)_t \rangle := \int_0^t \BZ_s^\ell d \BX_s^i, \ \  \langle [\tau_1,...,\tau_n]_i, \BBI_{\BX^i}(\BZ)_t \rangle := \langle \tau_1 \cdots \tau_n, \BZ_t \rangle,\\
 &&\langle \tau , \BBI_{\BX^i}(\BZ)_t \rangle := 0, \text{ otherwise },
\end{eqnarray*}
where we omit the obvious dual notation $* $ from here on. Then one can check $\BBI_{\BX^i}(\BZ)$ satisfies relation \eqref{crp}. Indeed, for any $\tau=[f]_i,$ one has
\begin{equation}\label{i.remainer}
\langle [f]_i, \BBI_{\BX^i}(\BZ)_t \rangle = \langle f,  \BZ_t \rangle = \langle \BX_{s,t} \star f,  \BZ_s \rangle + R_{s,t}^f=  \langle \BX_{s,t} \star [f]_i, \BBI_{\BX^i}(\BZ)_t \rangle + R_{s,t}^f,
\end{equation}
where the last identity follows form Remark \ref{star op}. In particular, one has $R^{\BBI_{\BX^i}(\BZ),f}=0$ for any $f\in \MF \setminus \MT,$ and
$$
\BBI_{\BX^i}(\BZ)_t^{\BI}= \sum_{h \in \MF_{[p]-2}^0} \BZ_s^h \BX_{s,t}^{[h]_i} + R^{\BZ,\BI}_{s,t}.
$$
More generally, for a $\BZ=(\BZ^1, ..., \BZ^d)$ with $\BZ^i$ a $\BX-$controlled rough path above $z^i \in \R^e,$ let $\int_0^t \BZ_r^\ell d \BX_r:= \sum_{i=1}^d \int_0^t (\BZ^i_r)^\ell d\BX^i_r,$ and one can lift $\int_0^t \BZ_r^\ell d \BX_r$ to a controlled rough path $\BBI_{\BX}(\BZ)_t$ by defining for any $[f]_i,$ $1\leq i \leq d,$
$$
\langle \BI , \BBI_{\BX}(\BZ)_t \rangle := \sum_{i=1}^d \int_0^t (\BZ^i_r)^\ell d\BX^i_r, \ \  \langle [f]_i, \BBI_{\BX}(\BZ)_t \rangle := \langle f,  \BZ_t^i \rangle,
$$
and null otherwise.

\end{rem}

\begin{rem}(\textbf{stability of c.r.p. under smooth function})\label{smooth func as crp}
 Given $F \in C^{[p]+1}(\R^e,  \R^e )$ and a controlled rough path $\BZ\in (\mathcal{H}_{[p]-1})^e,$ one can lift $F(z)$ to a controlled rough path $F(\BZ)\in (\mathcal{H}_{[p]-1})^e$. According to Taylor formula and \eqref{crp}, one has
\begin{eqnarray*}
&&F(z_t) - F(z_s)= \sum_{n=1}^{[p]-1} \frac{1}{n!} D^n F(z_s)(z_{s,t},...,z_{s,t}) + R^{F(\BZ),\BI}_{s,t}\\
               &=& \sum_{n=1}^{[p]-1} \frac{1}{n!} D^n F(z_s)(\sum_{h_1}\langle h_1,\BZ_s\rangle \langle \BX_{s,t}, h_1 \rangle,...,\sum_{h_n}\langle h_n,\BZ_s\rangle \langle \BX_{s,t}, h_n \rangle)+ R^{F(\BZ),\BI}_{s,t}\\
               &=& \sum_{n=1}^{[p]-1} \sum_{h_1\cdots h_n=h\atop h_i \in \MF_{[p]-1}} \frac{1}{n!} D^n F(z_s)(\langle h_1,\BZ_s\rangle,...,\langle h_n,\BZ_s \rangle )\langle \BX_{s,t}, h\rangle + R^{F(\BZ),\BI}_{s,t},
\end{eqnarray*}
with $|R^{F(\BZ),\BI}_{s,t}| \leq C (|z_{s,t}|^{[p]}+ R^{\BZ,\BI}_{s,t})=C R^{\BZ,0}_{s,t}.$ Compared with \eqref{1crp}, the natural choice for the lift is, for any $f \in \MF_{[p]-1},$
\begin{equation}
F(\BZ)_t^f= \langle f, F(\BZ)_t \rangle=  \sum_{n=1}^{[p]-1}\sum_{f_1\cdots f_n=f, \atop f_i \in \MF_{[p]-1}} \frac{1}{n!} D^n F(z_t)(\langle f_1,\BZ_t\rangle,...,\langle f_n,\BZ_t \rangle ).
\end{equation}
Then one can check $F(\BX)$ satisfies \eqref{crp}. Indeed, for any $f\in \MF_{[p]-1}$ with $|f|\geq 1,$ by applying $\eqref{crp}$ to $\BZ$, one has
\begin{eqnarray*}
&&F(\BZ)_t^f =   \sum_{n=1}^{[p]-1}\sum_{f_1\cdots f_n=f, \atop f_i \in \MF_{[p]-1}} \frac{1}{n!} D^n F(z_t)(\langle f_1,\BZ_t\rangle,...,\langle f_n,\BZ_t \rangle )\\
&= &  \sum_{n=1}\sum_{f_1\cdots f_n=f} \sum_{h_i \in \MF_{[p]-1}} \frac{1}{n!} D^n F(z_t)(\langle h_1, \BZ_s \rangle,...,\langle h_n, \BZ_s \rangle )\langle \BX_{s,t}\star f_1, h_1 \rangle \cdots \langle \BX_{s,t}\star f_n, h_n \rangle  + R^{F(\BZ),f}_{s,t}\\
&= & \sum_{n=1} \sum_{h_1\cdots h_n=h} \frac{1}{n!} D^n F(z_t)(\langle h_1, \BZ_s \rangle,...,\langle h_n, \BZ_s \rangle ) (\sum_{f_1\cdots f_n=f}\langle \BX_{s,t}\star f_1, h_1 \rangle \cdots \langle \BX_{s,t}\star f_n, h_n \rangle)  + R^{F(\BZ),f}_{s,t}.
\end{eqnarray*}
with $|R^{F(\BZ),f}_{s,t}| \leq C_{F,\|\BZ\|_p} R^{\BZ,|f|}_{s,t}.$ By Taylor's expansion, one has
\begin{eqnarray*}
&&D^n F(z_t)(  \BZ_s^{h_1}  ,..., \BZ_s^{h_n}  ) = \sum_{m=n}^{[p]-1} \frac{1}{(m-n)!} D^m F(z_s)(  \BZ_s^{h_1}  ,..., \BZ_s^{h_n}, z_{s,t},..,z_{s,t} ) + R^n_{s,t}\\
&=& \sum_{m=n}^{[p]-1} \frac{1}{(m-n)!} \sum_{g_1,...,g_{m-n} \in \MF} D^m F(z_s)(  \BZ_s^{h_1}  ,..., \BZ_s^{h_n}, \BZ_{s }^{g_1},..,\BZ_{s }^{g_{m-n}} )\BX_{s,t}^{g_1}\cdots\BX_{s,t}^{g_{m-n}}  + R^n_{s,t}
\end{eqnarray*}
with $|R^n_{s,t}| \leq C_F |z_{s,t}|^{[p]-n}$ which implies $R^n \in V^{\frac{p}{[p]-n}} \subseteq V^{\frac{p}{[p]-|f|}}$ since $n\leq|f_1|+\cdots+|f_n|= |f|.$ Then one obtains
\begin{eqnarray*}
F(\BZ)_t^f & =& \sum_{m=1}^{[p]-1} \sum_{n=1}^m \sum_{f_1\cdots f_n=f}\frac{1}{n!} \frac{1}{(m-n)!} \sum_{g_i,h_i \in \MF} D^m F(z_s)(  \BZ_s^{h_1}  ,..., \BZ_s^{h_n}, \BZ_{s }^{g_1},...,\BZ_{s }^{g_{m-n}} )\\
&&\BX_{s,t}^{g_1}\cdots\BX_{s,t}^{g_{m-n}} \langle \BX_{s,t}\star f_1, h_1 \rangle \cdots \langle \BX_{s,t}\star f_n, h_n \rangle + R^{F(\BZ),f}_{s,t}\\
&= & \sum_{m=1}^{[p]-1}\sum_{n=1}^m \left( m \atop n \right) \frac{1}{m!} \sum_{l_1,..., l_m \in \MF \atop (l_1),...,(l_m)} \sum_{f_1\cdots f_n=f} D^m F(z_s)(  \BZ_s^{l_1}  ,..., \BZ_s^{l_n}, ...,\BZ_{s }^{l_{m}})\\
 &&\langle \BX_{s,t}, l_1^{(1)} \cdots l_n^{(1)} l_{n+1}\cdots l_m \rangle \langle f_1, l_1^{(2)}\rangle \cdots \langle f_n, l_n^{(2)}\rangle + R^{F(\BZ),f}_{s,t}\\
&=&  \sum_{m=1}^{[p]-1} \sum_{n=1}^m \left( m \atop n \right) \frac{1}{m!}  \sum_{l_1,..., l_m \in \MF} D^m F(z_s)(  \BZ_s^{l_1}  ,..., \BZ_s^{l_n}, ...,\BZ_{s }^{l_{m}})\\
 && \sum_{l_1^{(2)},...,l_n^{(2)} \neq \BI} \langle \BX_{s,t}, l_1^{(1)} \cdots l_n^{(1)} l_{n+1}\cdots l_m \rangle \langle f, l_1^{(2)} \cdots l_n^{(n)} \rangle + R^{F(\BZ),f}_{s,t}\\
&=&  \sum_{m=1}^{[p]-1} \sum_{l_1,..., l_m \in \MF} \frac{1}{m!} D^m F(z_s)(  \BZ_s^{l_1}  ,..., \BZ_s^{l_n}, ...,\BZ_{s }^{l_{m}}) \langle \BX_{s,t}\otimes f, l_1^{(1)} \cdots l_m^{(1)} \otimes l_1^{(2)} \cdots l_m^{(2)}   \rangle+ R^{F(\BZ),f}_{s,t}\\
 &=&  \sum_{l\in \MF} \langle l, F(\BZ_s) \rangle \langle X_{s,t} \star f, l \rangle + R^{F(\BZ),f}_{s,t},
 \end{eqnarray*}
with $|R^{F(\BZ),f}_{s,t}| \leq C_{F,\|\BZ\|} R^{\BZ,|f|}_{s,t},$ where we apply identity $\langle f, l_1^{(2)}\cdots l_n^{(2)}\rangle = \sum_{f_1\cdots f_n=f} \langle f_1, l_1^{(2)}\rangle \cdots \langle f_n , l_n^{(2)}\rangle$ in the third equation and symmetry of $D^m F(z_s)$ in the forth equation.

\end{rem}

\subsection{Differential equations} 

In this part, finally, we solve equations driven by branched $p$-rough paths with jumps, essentially following the strategy of the level-$2$ case discussed earlier.

\begin{lem}(\textbf{invariance of the solution mapping})\label{inv. for crp}
Suppose $\BY \in \MV_\BX^p $ above $y,$ and $F=(F_1,...,F_d)$ with each $F_i \in C_b^{[p]+1}(\R^e,\R^e).$ Then there exists a natural lift for $\int_0^t F(Y_r)^\ell d\BX_r,$ denoted as $\BZ:=\BBI_{\BX}(F(\BY)),$ which is given by
\begin{eqnarray*}
&&\BZ_t^{[\tau_1\cdots \tau_n]_i}  = \sum_{j=1}^{[p]-1} \sum_{f_1\cdots f_j=\tau_1 \cdots \tau_n} \frac{1}{j!} D^j F_i(y_t)(  \BY_t^{f_1} ,..., \BY_t^{f_j} ), \text{ for any trees }\tau_1,...,\tau_n,\\
&&\BZ_t^f=0, \text{ if }f \in \MF \setminus {\MT}.
\end{eqnarray*}
Furthermore, one has the following local estimate
\begin{eqnarray*}
\|R^{\BZ,\BI}\|_{\frac{p}{[p]},[s,t]} &\leq&  C_{F,\|\BY\|}\sum_{i=1}^d  \sum_{h \in \MF^0_{[p]-1}} \left(\|R^{ \BY,h}\|_{\frac{p}{[p]-|h|},[s,t)}+1 \right) \|\BX^{[h]_i}\|_{\frac{p}{(|h|+1)},(s,t]}, \\
\|R^{\BZ,\tau} \|_{\frac{p}{[p]},[s,t]} &\leq&  C_{F,\|\BY\|} \|R^{\BY,|\tau|-1}\|_{\frac{p}{[p]-|\tau|+1},[s,t]}, \text{ for any }\tau \in \MT_{[p]-1}.
\end{eqnarray*}
\end{lem}

\begin{proof}

Let $\BZ$ defined as above. The fact $\BZ$ is a $\BX-$controlled rough path follows from Remark \ref{integral as crp} and Remark \ref{smooth func as crp}. Indeed, for any $[\tau_1\cdots \tau_n]_i \in \MT_{[p]-1},$ by one has
\begin{eqnarray*}
\langle [\tau_1\cdots \tau_n]_i , \BBI_\BX(F(\BY))_t \rangle &=& \langle \tau_1\cdots \tau_n , F_i(\BY)_t \rangle \\
                                                            &=& \sum_{j=1}^{[p]-1} \sum_{f_1\cdots f_j=\tau_1 \cdots \tau_n} \frac{1}{j!} D^j F_i(y_t)(\langle f_1, \BY_t \rangle,...,\langle f_j, \BY_t \rangle) .
                                                            \end{eqnarray*}
For the estimate part, by local estimate for rough integration, one has
\begin{eqnarray*}
|R^{\BZ,\BI}_{s,t}|&=& \sum_{i=1}^d |\int_s^t F_i (\BY_r)^\ell d\BX_r^i - \sum_{h \in \MF_{[p]-2}^0} F_i(\BY)^h_s \BX_{s,t}^{[h]_i} | \\
&\leq& C \sum_{i=1}^d \Big(\sum_{h \in \MF^0_{[p]-1}} \|R^{F_i(\BY),h}\|_{\frac{p}{[p]-|h|},[s,t)} \|\BX^{[h]_i}\|_{\frac{p}{(|h|+1)},(s,t]}+ \sum_{h\in \MF_{([p]-1)}} |F_i(\BY)^h_s \BX_{s,t}^{[h]_i}| \Big)\\
&\leq& C_{F,\|\BY\|}\sum_{i=1}^d  \sum_{h \in \MF^0_{[p]-1}} \left(\|R^{ \BY,h}\|_{\frac{p}{[p]-|h|},[s,t)}+1 \right) \|\BX^{[h]_i}\|_{\frac{p}{(|h|+1)},(s,t]}
\end{eqnarray*}
For any $\tau=[\tau_1\cdots \tau_n]_i \in \MT_{[p]-1},$ according to identity \eqref{i.remainer} and Remark \ref{smooth func as crp}, one has
\begin{eqnarray*}
|R^{\BZ,\tau}_{s,t}|=|R^{F(\BY),\tau_1\cdots \tau_n}_{s,t}| \leq C_{F,\|\BY\|} |R^{\BY,|\tau|-1}_{s,t}|.
\end{eqnarray*}

\end{proof}

\begin{lem}(\textbf{contraction of the solution mapping})\label{contr.for crp}
Suppose $\BY \in \MV_{\BX}^p,$ $\tBY \in \MV_{\tBX}^p,$ and $F\in C_b^{[p]+1}.$ Let $\BZ:=\BBI_{\BX}(F(\BY)) , \tBZ:=\BBI_{\tBX}(F(\tBY))$ defined as in the above lemma. Then one has the following estimates
\begin{eqnarray*}
&&\|R^{\BZ,\BI} - R^{\tBZ,\BI}\|_{\frac{p}{[p]},[s,t]}  \leq   C   \Big( \| \Delta \BX \|_{p,[s,t]}+ \|\BX \|_{p,[s,t]}( |  \BY_s- \tBY_s| + \sum_{f\in \MF_{[p]-1}^0} \|R^{\BY,f}-R^{\tBY,f} \|_{\frac{p}{[p]-|f|},[s,t]} ) \Big)\\
&&\|R^{\BZ,\tau} - R^{\tBZ,\tau} \|_{\frac{p}{[p]-|\tau|},[s,t]}  \leq  C \Big( |\BY_s- \tBY_s|+ \sum_{|h|\leq |\tau|-1}\| R^{\BY,h}-R^{\tBY,h} \|_{\frac{p}{[p]-|h|},[s,t]}\\
&& + \sum_{|h|\geq |\tau| }\|R^{\BY,h}-R^{\tBY,h}  \|_{\frac{p}{[p]-|h|},[s,t]} \|\BX\|_{p,[s,t]} + \|  \BX-\tBX \|_{p,[s,t]}  \Big),\ \text{for any }|\tau|\geq 1,
\end{eqnarray*}
where $C$ depends on $F, \|\BY\|_p, \|\tBY\|_p,\|\BX\|_p.$

\end{lem}

\begin{proof}
For the first estimate, similar to the level-$2$ rough path case, denote $\Xi^i_{u,v}:= \sum_{h \in \MF^0_{[p]-1}} F_i(\BY)_u^h \BX_{u,v}^{[h]_i} $, $\tilde{\Xi}^i_{u,v}:= \sum_{h \in \MF^0_{[p]-1}} F_i(\tBY)_u^h \tBX_{u,v}^{[h]_i} ,$ and one has
\begin{eqnarray*}
&&|R^{\BZ,\BI}_{u,v}- R^{\tBZ,\BI}_{u,v}| \\
&\leq& \sum_{i=1}^d  |\int_u^v F_i(Y_r)^\ell d\BX_r^i - \sum_{h\in \MF_{[p]-2}^0 }F_i(\BY)_u^h \BX_{u,v}^{[h]_i} - \int_u^v F_i(\tY_r)^\ell d\tBX_r^i - \sum_{h\in \MF_{[p]-2}^0 }F_i(\tBY)_u^h \tBX_{u,v}^{[h]_i} |\\
&\leq& \sum_{i=1}^d\Big(|\sum_{h\in \MF_{([p]-1)}}( F_i(\BY)_u^h \BX_{u,v}^{[h]_i}- F_i(\tBY)_u^h \tBX_{u,v}^{[h]_i})|+ |\I(\Delta \Xi^i)_{u,v}- \Delta \Xi^i_{u,v} |\Big)\\
&\leq& C_p \sum_{i=1}^d \Big( \sum_{h\in \MF_{([p]-1)}} | (F_i(\BY)_u^h-F_i(\tBY)_u^h)\BX_{u,v}^{[h]_i} +(\BX_{u,v}^{[h]_i}-\tBX_{u,v}^{[h]_i})F_i(\tBY)_u^h|+ \sup_{\tau,r,\nu \in[u,v]} |\delta (\Delta \Xi^i)_{\tau,r,\nu} |   \Big),
\end{eqnarray*}
where the last inequality follows by Young's argument and $\delta (\Delta \Xi^i)_{\tau,r,\nu}=\sum_{h \in \MF^0_{[p]-1}}( R^{F_i(\BY),h}_{\tau,r} \BX^{[h]_i}_{r,\nu}-R^{F_i(\tBY),h}_{\tau,r} \tBX^{[h]_i}_{r,\nu} )$. Note that for any $h\in \MF_{([p]-1)},$ $F_i(\BY)_v^h= F_i(\BY)_u^h + R^{F_i(\BY),h}_{u,v}.$ By applying identity $ab-a'b'=a(b-b')+ b'(a-a') $ and estimate like \eqref{bound for unif. norm}, one has
\begin{eqnarray*}
\|R^{\BZ,\BI} - R^{\tBZ,\BI} \|_{\frac{p}{[p]},[s,t] } &\leq& C  \Big( \| \Delta \BX \|_{p,[s,t]} + \|\BX \|_{p,[s,t]} \|\Delta \BY \|_{p,[s,t]}        \Big)\\
                                                       &\leq& C   \Big( \| \Delta \BX \|_{p,[s,t]}+ \|\BX \|_{p,[s,t]}( |\Delta \BY_0| + \sum_{f\in \MF_{[p]-1}^0} \|\Delta R^f \|_{\frac{p}{[p]-|f|},[s,t]} ) \Big)
\end{eqnarray*}
with $C$ depending on $F, \|\BY\|_p, \|\tBY\|_p,\|\BX\|_p.$ For any $\tau= [\tau_1 \cdots \tau_m]_i,$ by Remark \ref{smooth func as crp}, one has
\begin{eqnarray*}
&&|R^{\BZ,\tau}_{s,t}- R^{\tBZ,\tau}_{s,t}|= |R^{F_i(\BY),\tau_1\cdots \tau_m}_{s,t}- R^{F_i(\tBY),\tau_1\cdots \tau_m}_{s,t}| \\
&\leq& C \Big( \sum_{|h|,|g|\leq |\tau|-1 }|G^h(\BY^h_s,h)R^{\BY,g}_{s,t}- G^h(\tBY^h_s,h)R^{\tBY,g}_{s,t} | + \sum_{|h|,|g| \geq |\tau|\atop |f|\geq 1 }|G^h(\BY^h_s,h)R^{\BY,g}_{s,t}\BX_{s,t}^f\\
 &&- G^h(\tBY^h_s,h)R^{\tBY,g}_{s,t}\tBX_{s,t}^f | + \sum_{n \leq |\tau|-1} |G^n(y_s,y_t)y_{s,t}^{\otimes[p]-n}- G^n(\ty_{s},\ty_t)\ty_{s,t}^{\otimes[p]-n}  |    \Big),
\end{eqnarray*}
where $G^h$ and $G^n$ are functions depending on derivatives of $F$ and $C$ depending on $F, \|\BY\|_p, \|\tBY\|_p,\|\BX\|_p.$ By applying identity \eqref{bound for unif. norm}, it follows that
\begin{eqnarray*}
\|R^{\BZ,\tau}- R^{\tBZ,\tau}\| \leq C \Big( |\BY_0- \tBY_0|+ \sum_{|h|\leq |\tau|-1}\|\Delta R^h \|_{\frac{p}{[p]-|h|}}+ \sum_{|h|\geq |\tau| }\|\Delta R^h \|_{\frac{p}{[p]-|h|}} \|\BX\|_{p} + \|\Delta \BX \|_p  \Big),
\end{eqnarray*}
where $C$ depends on $F, \|\BY\|_p, \|\tBY\|_p,\|\BX\|_p.$

\end{proof}

\begin{defi}(\textbf{solutions for D.E. driven by branched rough paths})
We say $Y:[0,T] \rightarrow \R^e$ solves
\begin{equation}\label{d.e. by brp}
Y_t= y_0 + \int_0^t F(\BY_r)^\ell d \BX_r
\end{equation}
if $Y_t= y_0 + \int_0^t F(\BY_r)^\ell d \BX_r,$ where $\BY$ is a $\BX$-controlled rough path above $Y,$ such that
$$
\BY_t- \BY_s= \BBI_{\BX}(F(\BY))_t - \BBI_{\BX}(F(\BY))_s.
$$

\end{defi}

Before we show the fixed point theorem for differential equations driven by branched rough paths, we need some argument about initial values. Suppose $Y$ solves \eqref{d.e. by brp}, then we claim that the solution lift $\BY_t$ depends only on $Y_s$ and derivatives of $F.$ Indeed, one obviously has $\BY_t^{\bullet_i}= F_i(Y_t), $ and furthermore, for any tree $\tau=[\tau_1 \cdots \tau_n]_i $ with each $\tau_j$ a smaller tree, noticing $\BY^f=0$ for any nontrivial forest $f,$ one has
\begin{eqnarray*}
\BY_s^\tau &=& \sum_{j=1}^{[p]-1} \sum_{f_1\cdots f_j=\tau_1 \cdots \tau_n} \frac{1}{j!} D^j F_i(y_t)(\langle f_1, \BY_t \rangle,...,\langle f_j, \BY_t \rangle) \\
&=& \sum_{\sigma \in \mathfrak{G}_n} \frac{1}{n!} D^n F_i(Y_t)(  \BY_t^{\tau_{\sigma(1)}}  ,...,  \BY_t^{\tau_{\sigma(n)}}  )=D^n F_i(Y_t)(  \BY_t^{\tau_1}  ,...,  \BY_t^{\tau_{n}}  ),
\end{eqnarray*}
where $\mathfrak{G}_n$ means the permutation group of $\{1,...,n\}.$ For any $\tau=[\tau_1 \cdots \tau_n]_i,$ let $F^\tau:\R^e \rightarrow \R^e$ the function inductively defined such that $F^{\bullet_i}:= F_i,$ and satisfies
$$
F^\tau(Y_t):= D^n F_i(Y_t)(  \BY_t^{\tau_1}  ,...,  \BY_t^{\tau_{n}} )
$$
where $\BY$ is defined by $\BY_t^{\bullet_i}= F_i(Y_t)$ and $\BY_s^\tau=D^n F_i(Y_t)(  \BY_t^{\tau_1}  ,...,  \BY_t^{\tau_{n}}  ).$ It is obvious that $F^\tau $ depends only on derivatives of $F$ up to order $n.$ For simplicity, denote $F^f=0$ for any nontrivial forest $f.$ In particular, if $\BY$ solves the above equation, one has $\BY^f_0 = F^f(y_0).$

\begin{thm}(\textbf{local solution for D.E. driven by branched rough paths})\label{local sol. for brp}
Suppose $\BX$ is a branched rough path with finite $p$-variation and right-continuous at $t=0.$ Then for any $F \in C^{[p]+1},$ there exists $t_1>0,$ such that the following equation has a unique solution on $[0,t_1],$
\begin{equation}
Y_t= y_0 + \int_0^t F(\BY_r)^\ell d \BX_r.
\end{equation}
Furthermore, if $\tY$ solves the same equation driven by $\tBX,$ then there exists $t_2>0,$ such that the following local Lipschitz estimate holds
$$
\|\BY-\tBY\|_{p,[0,t_2]} \leq C (|y_0- \tilde{y}_0|+ \|\BX;\tBX\|_{p,[0,t_2]} ),
$$
with $C$ depending on $F, y_0, \|\BX\|_{p}.$
\end{thm}

\begin{proof}
For the existence and uniqueness part, consider the following closed set
$$
\Omega_t:= \{\BY \in \MV_{\BX}^p,: \BY_0^f= F^f(y_0), \|\BY\|_{p,[0,t]} \leq |\BY_0|+1, \|R^{\BY,k}\|_{\frac{p}{[p]-k},[0,t]}< \delta_k ,k=0,...,[p]-1  \},
$$
with $t,\delta_k$ to be determined. Define a mapping on $\Omega_t,$
$$
\begin{array}{lclcl}
\mathcal{M}_t:& \Omega_t& \longrightarrow & \Omega_t&\\
& \BY &&  \BZ := \BBI_{\BX}(F(\BY)) .
\end{array}
$$
According to estimates in Lemma \ref{inv. for crp}, for any $f\in \MF_{[p]-1},$ noticing that
$$
\|R^{\BZ,f}\|_{p/([p]-|f|)} \leq C \|R^{\BY,|f|-1}\|_{p/([p]-|f|+1)}\leq C \delta_{|f|-1},
$$
with $C$ uniform over $f\in \MF_{[p]-1}^0,$ one may choose $\{\delta_k\}_{k=1}^{[p]-1}$ such that $C \delta_{k} < \delta_{k+1}.$ For $f=\BI,$ one has $\|R^{\BZ,\BI}\|_{p/[p],[0,t]} \leq C_{\delta_k}\|\BX\|_{p,[0,t]}. $ Since $\BX$ is right continuous at $t=0,$ so one may choose $t$ small such that $C_{\delta_k}\|\BX\|_{p,[0,t]} < \delta_0,$ which complete the invariance part. For the contraction part, according to Lemma \ref{contr.for crp}, indeed one has, for any $\BY,\tBY \in \Omega_t,$
\begin{eqnarray*}
&&\|\Delta R^{\BZ,\BI}  \|_{\frac{p}{[p]},[0,t]}\leq C \|\BX\|_{p,[0,t]} \sum_{f\in \MF_{[p]-1}^0 }\|\Delta R^{\BY,f} \|_{\frac{p}{[p]-|f|},[0,t]}\\
&&\| \Delta R^{\BZ,\tau}  \|_{\frac{p}{[p]-|\tau|},[0,t]} \leq C ( \|\Delta R^{\BY,|\tau|-1}\|_{\frac{p}{[p]-|\tau|+1},[0,t]}+ \|\BX\|_{p,[0,t]}\sum_{|f|\geq|\tau|}\|\Delta R^{\BY,f}\|_{\frac{p}{[p]-|f| },[0,t]}  )
\end{eqnarray*}
Define the following equivalent norm on $\Omega_t,$
$$
\| \BY \|_{p,[0,t]}^{(\alpha)}:=|\BY_0| + \sum_{k=0}^{[p]-1} \alpha_k \| R^{\BY,k} \|_{\frac{p}{[p]-k},[0,t]},
$$
with $\alpha=(\alpha_0,...,\alpha_{[p]-1})$ a $(\R^+)^{[p]} -$valued vector to be determined. One obtains that
\begin{eqnarray*}
\| \Delta \BZ \|_{p,[0,t]}^{(\alpha)}   &\leq& C_\alpha \|\BX\|_{p,[0,t]}\|\Delta \BY \|_{p,[0,t]}^{(\alpha)} +  C\sum_{k=1}^{[p]-1} \alpha_k \| \Delta R^{\BY,k-1}\|_{\frac{p}{[p]-k+1},[0,t]}\\
&\leq& C_\alpha \|\BX\|_{p,[0,t]}\|\Delta \BY \|_{p,[0,t]}^{(\alpha)} + C\max_{k=1,...,[p]-1} \frac{a_k}{a_{k-1}} \|\Delta \BY \|_{p,[0,t]}^{(\alpha)}.
\end{eqnarray*}
To obtain the contraction, one only needs to choose $\alpha$ such that $C\max_{k=1,...,[p]-1} \frac{a_k}{a_{k-1}}<1,$ and $t$ small such that $C_\alpha \|\BX\|_{p,[0,t]}< 1- C\max_{k=1,...,[p]-1} a_k/a_{k-1}.$ For the inequality part, according to the argument before this theorem, i.e. $\BY^\tau_s=F^\tau(Y_s)$ if $Y$ solves the equation, one only needs to prove
$$
\|Y-\tilde{Y}\|_{p,[0,t]} \leq C(|y_0-\ty_0|+\|\BX;\tBX\|_{p,[0,t]}).
$$
Since the natural lift $\BBI_\BX(F(\BY))$ is also a controlled rough path, one has for any $s,t,$
\begin{eqnarray*}
|\Delta Y_{s,t}| = |\sum_{i=1}^d (\sum_{h\in \MF_{[p]-2}^0 }F_i(\BY_s)^h \BX_{s,t}^{[h]_i}+ R^{F_i(\BY),\BI}_{s,t}- \sum_{h\in \MF_{[p]-2}^0 }F_i(\tBY_s)^h \tBX_{s,t}^{[h]_i}+ R^{F_i(\tBY),\BI}_{s,t}  )|.
\end{eqnarray*}
According to Lemma \ref{contr.for crp}, it follows that
\begin{eqnarray*}
\|\Delta Y \|_{p,[0,t]} \leq C (\|\Delta y_0\| + \|\Delta  \BX \|_{p,[0,t]} + \|\Delta Y \|_{p,[0,t]} \|\BX\|_{p,[0,t]} ),
\end{eqnarray*}
which implies our estimate by choosing $t$ small.

\end{proof}

\begin{thm}(\textbf{global solution for D.E. driven by branched rough paths})\label{sol for eq.by crp}
Suppose $\BX$ is a branched rough path with finite $p$-variation, $F\in C_b^{[p]+1}.$ Then the following equation has a unique solution on $[0,T],$
\begin{equation}
Y_t= y_0 + \int_0^t F(\BY_r)^\ell d \BX_r.
\end{equation}
Furthermore, if $\tY$ solves the same equation driven by $\tBX,$ the following local Lipschitz estimate holds
$$
\|\BY-\tBY\|_{p,[0,T]} \leq C (|y_0- \tilde{y}_0|+ \|\BX ; \tBX\|_{p,[0,T]} ),
$$
with $C$ depending on $F, y_0, L$ where $\|\BX\|_{p}, \|\tBX\|_{p} < L$.
\end{thm}

\begin{proof}
The proof is not very different from the level-$2$ rough path case, since we already solve the equation locally, and hence left to the reader.

\end{proof}


\section{Cadlag RDE stability under Skorokhod type metrics}

We now rephrase the hard analytical estimates obtained in the last two sections into a user-friendly format (Skorohod $J1$ type rough path metrics), with some immediate applications (even in absence of jumps).

\subsection{The $p$-variation Skorohod rough path metric} 

We recall the Skorokhod topology for c\`adl\`ag paths space in some metric space $E$. Denote $\Lambda_{[0,T]}$ the set of increasing bijective functions from $[0,T]$ to $[0,T].$ For any $x,y\in D([0,T],E),$ the Skorokhod metric is given by
$$
d (x,y):= \inf_{\lambda\in \Lambda} \{|\lambda| \vee \sup_{t\in[0,T]}d_E(x(\lambda(t)) ,y(t))\},
$$
where $|\lambda|:=\sup_{t\in[0,T]}|\lambda(t)-t|.$
We can define a $p$-variation variant of this metric. To this end, let $E$ be the Butcher group $G_{N}(\mathcal{H}^*)$ as introduced in Section \ref{sec:bRP}, equipped with left-invariant metric (see appendix). For any c\`adl\`ag branched rough paths $\BX,\BZ   $ with finite $p$-variation,
$$
\sigma_{p,[0,T]} (\BX,\BZ):= \inf_{\lambda\in \Lambda} \{|\lambda| \vee  \|\BX\circ \lambda ;\BZ \|_{p,[0,T]}\},
$$
where we recall that
\begin{equation}   \label{equ:inhomogRPnorm}
\|\BX \|_{p,[0,T]}:= \sum_{f\in \MF_{[p]} } \|\BX^f\|_{\frac{p}{|f|},[0,T]}.
\end{equation}
In particular, for the level-$2$ rough path case, one has $E \cong \R^d \oplus \R^{d\times d},$ and
\begin{eqnarray*}
\sigma_{\infty,[0,T]} (\BX,\BZ)&:=& \inf_{\lambda\in \Lambda} \left\{|\lambda| \vee  \left(\|X \circ \lambda -Z\|_{\infty,[0,T]}+ \| \X \circ (\lambda,\lambda ) -\Z \|_{\infty,[0,T]}\right)\right\}\\
\sigma_{p,[0,T]} (\BX,\BZ)&:=& \inf_{\lambda\in \Lambda} \left\{|\lambda| \vee  \left(\|X \circ \lambda -Z\|_{p,[0,T]}+ \|\X \circ (\lambda,\lambda ) -\Z\|_{\frac{p}{2},[0,T]}\right)\right\},
\end{eqnarray*}
with
$
\| \X-\Z \|_{\infty,[0,T]}:= \sup_{ 0 \le s<t \le T}     |\X_{s,t}-\Z_{s,t}|.
$
We then have

\begin{coro}(\textbf{local estimate for RDEs under Skorokhod topology})
Given $F \in C^{[p]+1}_b,$ and $\BX,\tBX $ branched rough paths with finite $p-$variaiton. Let $Y,\tY$ be solutions for RDEs driven by $\BX,\tilde{\BX}$ respectively. Suppose $\|\BX\|_{p,[0,T]},\|\tBX\|_{p,[0,T]}<L,$ Then one has
$$
\sigma_{p,[0,T]} ( Y, \tY )   \leq  C_{p,F,L} (\sigma_{p,[0,T]} (\BX,\tBX ) + |Y_0-\tY_0| ) .
$$

\end{coro}

\begin{proof}
For simplicity of notation only, we spell out the level-$2$ case. We claim that for any $\lambda \in \Lambda,$ $Y^\lambda := Y \circ \lambda$ solves
$$
Z_t= Y_0 + \int_0^t F(Z_s)^\ell d\BX_s^{\lambda},
$$
where $\BX^{\lambda}:=\left( X\circ \lambda,\X \circ (\lambda,\lambda)\right).$ Indeed,
\begin{eqnarray*}
Y(\lambda(t))&=& Y_0 + \int_0^{\lambda(t)} F(Y_s)^\ell d \BX_s \\
             &=& Y_0 + \lim_{|\op|\rightarrow 0} \sum_{[u,v] \in \op|_{[0,\lambda(t)]}} F(Y_u) X_{u,v} + DF(Y_u)F(Y_u)\X_{u,v}\\
             &=&  Y_0 + \lim_{|\op|\rightarrow 0} \sum_{[u,v] \in \op|_{[0,t]}} F(Y_{\lambda(u)}) X_{{\lambda(u)},{\lambda(v)}} + DF(Y_{\lambda(u)})F(Y_{\lambda(u)})\X_{{\lambda(u)},{\lambda(v)}}\\
             &=&  Y_0 + \int_0^{t} F(Y_s^\lambda)^\ell d \BX_s^\lambda.
\end{eqnarray*}
For any $\vep>0,$ there exists a $\lambda \in \Lambda, $ such that $|\lambda| \vee \|\BX^\lambda ;\tBX \|_{p,[0,T]} < \sigma_{p,[0,T]} (\BX,\tBX)+ \vep.$ Note that $p$-variation of $\BX$ remains the same after a time change of $[0,T]$. Then  according to Theorem \ref{global sol.} (level-$2$) resp. Theorem \ref{sol for eq.by crp} in the general case, one has
\begin{eqnarray*}
\sigma_{p,[0,T]} ( Y, \tY ) &\leq & |\lambda| + \|Y^\lambda - \tY \|_{p,[0,T]}\\
                               &\leq &  |\lambda|  + M^2(C_{p,F}(1\vee L))^{M+1} (\|\BX^\lambda ; \tBX \|_{p,[0,T]} + |Y_0-\tY_0| )\\
&\leq &\left(1 + M^2(C_{p,F}(1\vee L))^{M+1} \right) \left(  \sigma_{p,[0,T]} (\BX,\tBX ) + \vep + |Y_0-\tY_0| ) \right),
\end{eqnarray*}
which implies our estimate.

\end{proof}


\subsection{Interpolation and convergence under uniform bounds}

Recall that our ``$\infty$-norms'' aways involve a supremum over all $s < t$. For instance,
$
\| X \|_{\infty,[0,T]} \equiv  \sup_{s<t \in [0,T]} | X_{s,t} | \
$
and similarly in the branched rough path case, say in the level $2$ setting
$$
  \| \BX^n; \BX \|_{\infty, [0,T]}:= \| X^n-X \|_{\infty,[0,T]} + \|\X^n  -\X \|_{\infty, [0,T]}.
$$

\begin{lem}(\textbf{interpolation for c\`adl\`ag rough paths}) \label{interpolation}
Suppose $\BX, \BX^n$ are 
c\`adl\`ag $p$-rough paths 
with uniformly bounded $p$-variation,
$$
      \sup_n || \BX^n ||_{p,[0,T]} =: L < \infty
$$
Then, for any $p'>p$, there exists $C=C(p,p')$ such that
$$
\|\BX^n;\BX\|_{p',[0,T]} \leq C  L^{\frac {p}{p'}} \|\BX^n; \BX\|_{\infty, [0,T]}^{1-\frac {p}{p'}} \ ,  
$$
As a consequence, if $\BX^n$ converge to $\BX$ uniformly or in Skorokhod topology, with a uniform $p$-variation bound, then it also converges in $p'$-variation or its Skorokhod variant.

\end{lem}

\begin{proof}
Since $p$-variation norm pays no special attention on continuity, the argument is exactly the same as in the continuous case. We spell out the level-$2$ case,
the extension to the general branched case is immediate. Indeed, by basic inequalities one has
\begin{eqnarray*}
\sum_{[s,t]\in \op}|X^n_{s,t}-X_{s,t}|^{p'} &\leq& \sup_{u,v\in [0,T]} |X^n_{u,v} -X_{u,v}|^{p'-p} \sum_{[s,t]\in \op}|X^n_{s,t}-X_{s,t}|^{p},\\
\sum_{[s,t]\in \op}|\X^n_{s,t}-\X_{s,t}|^{\frac {p'} 2} &\leq& \sup_{u,v\in [0,T]} |\X^n_{u,v} -\X_{u,v}|^{\frac {p'-p} {2} } \sum_{[s,t]\in \op}|\X^n_{s,t}-\X_{s,t}|^{\frac p2},
\end{eqnarray*}
which implies
\begin{eqnarray*}
\|X^n - X\|_{p',[0,T]} &\leq&  \|X^n - X \|_{\infty,[0,T]}^{1-\frac{p}{p'}} \|X^n-X \|_{p,[0,T]}^{\frac{p}{p'}}, \\
\|\X^n -\X\|_{\frac{p'}{2},[0,T]} &\leq& \|\X^n - \X \|_{\infty,[0,T]}^{1-\frac{p}{p'}} \| \X^n - \X \|_{\frac{p}{2},[0,T]}^{\frac{p}{p'}},
\end{eqnarray*}
where $ |\X^n - \X|_{\infty,[0,T]}:= \sup_{u < v\in [0,T]} |\X^n_{u,v} -\X_{u,v}|.$ Then the inequality follows from some basic inequalities, such as $a^\beta + b^\beta \leq 2(a+b)^\beta$ for $\beta\in (0,1).$

\end{proof}


\begin{thm}(\textbf{convergence for RDEs under uniform or Skorokhod topology})\label{converge for RDEs}
Suppose $\BX, \BX^n$ are 
c\`adl\`ag $p$-rough paths 
with, as above, $   \sup_n || \BX^n ||_{p,[0,T]} = L < \infty$.
%
Let $Y^{n}$ be the (unique) solution to
\[
dY^{n}=F \left( Y^{n}\right) d\mathbf{X}^{n},\,Y^{n}=y_{0} \ .
\]
Then, for any $p'>p$ with $[p']=[p],$ one has the following estimate in the uniform version,

\begin{eqnarray*}
\|Y-Y^n \|_{p',[0,T]} &\leq&    M^{2+\frac{1}{p'}} (C_{p,F}(1\vee L))^{M+1} ||\BX^n; \BX ||_{\infty, [0,T]}^{1-\frac {p}{p'}} ,
\end{eqnarray*}
with $ M=C_{p,F} L^p$ as before, and then also in Skorokhod rough metric,
\begin{eqnarray*}
\sigma_{p',[0,T]}(Y,Y^n) &\leq&  M^{2+\frac{1}{p'}} (C_{p,F}(1\vee L))^{M+1} \sigma_{\infty, [0,T]}^{1-\frac {p}{p'}}(\BX,\BX^n) \ .
\end{eqnarray*}

In particular, if $\BX^n$ converge to $\BX$ uniformly or in Skorokhod topology, with a uniform $p$-variation bound, the RDE solutions
converge uniformly or in Skorokhod topology, also with uniform $p$-variation bounds.
%
%
%
%

\end{thm}


\begin{proof}

For the uniform version, the proof follows from the local estimate for RDEs, i.e. Theorem \ref{global sol.} and the interpolation property. Indeed, for any $p'>p$, a partition $\op$ and $\BZ \in \{\BX, \BX^n |n=1,...,\}$, one has
\begin{eqnarray*}
\sum_{[s,t]\in \op } |Z_{s,t}|^{p'} ( \text{ or } \sum_{[s,t]\in  {\op}} |\Z_{s,t}|^{\frac{p'}{2}}) &\leq & \sup_{s,t}|Z_{s,t}|^{p'-p} \sum_{[s,t]\in \op } |Z_{s,t}|^{p} ( \text{ or } \sup_{s,t}|\Z_{s,t}|^{\frac{p'-p}{2}} \sum_{[s,t]\in  {\op} } |\Z_{s,t}|^{\frac{ p}{2}} )\\
&\leq &   \|Z\|_{p,[0,T]}^{p'-p} \|Z\|_{p,[0,T]}^p ( \text{ or } \| \Z \|_{\frac{p}{2},[0,T]}^{\frac{p'-p}{2}} \|\Z\|_{\frac p 2,[0,T]}^{\frac p 2}) \\
&\leq& \|Z\|_{p,[0,T]}^{p'} (  \text{ or } \| \Z \|_{\frac{p}{2},[0,T]}^{\frac{p' }{2}})
\end{eqnarray*}
which implies $\|\BZ \|_{p',[0,T]} \leq  L.$ According to Theorem \ref{global sol.} and the interpolation inequality, one has
\begin{eqnarray*}
\|Y-Y^n \|_{p',[0,T]} &\leq&  M^2(C_{p,F}(1\vee L))^{M+1} \|\BX;\BX^n \|_{p',[0,T]}\\
                        &\leq&  16 L^{\frac {p}{p'}} M^2 (C_{p,F}(1\vee L))^{M+1} \|\BX^n; \BX \|_{\infty, [0,T]}^{1-\frac {p}{p'}} .
\end{eqnarray*}

Note that if $\BX^n$ converges to $\BX$ in the Skorokhod topology, for any fixed $\vep>0,$ there exists $\lambda^n \in \Lambda$ such that $ \| \BX^n \circ \lambda^n ; \BX \|_{\infty,[0,T]} < \frac{\vep}{2}$ and $|\lambda^n|< \frac{\vep}{2}.$ Note that the $p$-variation of $\BX^n \circ \lambda^n$ is the same as $\BX^n.$ One has
\begin{eqnarray*}
\sigma_{p',[0,T]}(Y,Y^n) &\leq& |\lambda^n| + \|Y-Y^n \circ \lambda^n \|_{p',[0,T]},\\
                           &\leq& |\lambda^n| + 16 L^{\frac {p}{p'}} M^2 (C_{p,F}(1\vee L))^{M+1} \|\BX^n \circ \lambda^n ; \BX \|_{\infty, [0,T]}^{1-\frac {p}{p'}}
\end{eqnarray*}

\end{proof}


\subsection{Discrete approximation for c\`adl\`ag ODEs/RDEs } \label{sec:discrete}

As a first application, we discuss discrete-time approximations (``higher-order Euler schemes'') to rough differential equations first discussed by Davie \cite{Dav07} in the level-$2$ setting and then \cite{FV08a} (see also \cite{FV10}) in case of geometric $p$-rough paths, always continuous. Having a discontinuous theory at hand, such approximations are readily written as RDEs driven by piecewise constant rough paths. The just obtained stability theorems for these equations then make convergence
statements of such higher order schemes basically immediate.

For any partition $\op$ of $\left[ 0,T\right] $, define the piecewise constant c\`adl\`ag rough path: for any $[s,t)\in \op,$
\[
\mathbf{X}_{u}^{\op}=\mathbf{X}_{s}\text{ when }u\in \lbrack s,t), \ \ \ \BX_T^n\equiv \BX_T.
\]
For instance, in the level-$2$ rough path case, $
\mathbf{X}_{0,u}^{\op}= (X_{0,s},\X_{0,s} )   \text{ when }u\in \lbrack s,t), \ \ \ \BX_{0,T}^\op  \equiv (X_{0,T},\X_{0,T}). $
As a consequence of Chen's relation, we note that $X_{u,v} = 0, \ \X_{u,v} = 0$ if $ u,v \in [s,t)$. This extends directly to the branched rough path case where one has $\BX_{u,v}^f=0$ for any $f\in \MF$.

\begin{lem}(\textbf{Convergence of discretized rough paths})\label{converge of discrete rp}
Suppose $\BX  $ is a c\`adl\`ag $p$-rough path. For any partition $\op$ of $[0,T]$ and $\BX^\op$ defined as above, one has
$$
\lim_{|\op|\rightarrow 0} \BX^\op = \BX,
$$
under the Skorokhod metric in the MRS sense. Moreover, one has $\sup_{\op} \|\BX^\op \|_{p,[0,T]} \leq \|\BX\|_{p,[0,T]}.$

\begin{rem} Simple examples show that one cannot replace Skorokhod - by uniform convergence.
\end{rem}

\end{lem}

\begin{proof}

Recall that $\omega_{\BX }(s,t):=\sum_{f\in \MF_{[p]}} \|\BX^f\|_{\frac{p}{|f|},[s,t]}^{\frac{p}{|f|}}.$ Let $x(t):=\omega_\BX(0,t).$ According to Lemma \ref{regular path}, for any $1>\vep>0,$ there exists a partition $\op_\vep,$ such that for any $[s,t)\in \op_\vep,$ $\omega_\BX(s,t-) \leq x_{s,t-} \leq \vep.$ Then for any partition $\op=\{0=s_0< s_1 <s_2< \cdots < s_N< s_{N+1}=T \}$ with $|\op|< |\op_\vep| \wedge \frac{\vep}{2} ,$ without loss of generality, suppose $\op_\vep=\{0=t_0< t_1 <t_2< \cdots < t_{N-1}< t_N=T \}$ with $t_i \in [s_i,s_{i+1}), $ $i=1,...,N-1$(otherwise make a refinement $\op'_\vep$ of $\op_\vep$ and work with $\op'_\vep$ from here on).

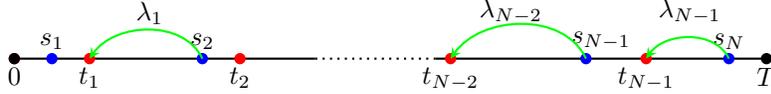
\begin{tikzpicture}[scale = 2]
\draw [thick, -] (2,0) --(4,0);
\draw [thick, dotted] (4,0) --(4.8,0);
\draw [thick, -] (4.8,0) --(7,0);

\node [below] at (2.5,0) {$t_1$};
\node [below] at (3.5,0) {$t_2$};
\node [below] at (6.2,0) {$t_{N-1}$};
\node [below] at (4.9,0) {$t_{N-2}$};
\node at (2,0) [reddot] {};
\node at (2.5,0) [reddot] {};
\node at (3.5,0) [reddot] {};
\node at (4.9,0) [reddot] {};
\node at (6.2,0) [reddot] {};

\node [above] at (2.25,0) {$s_1$};
\node [above  ] at (3.25,0) {$s_2$};
\node [above] at (5.9,0) {$s_{N-1}$};
\node [above] at (6.75,0) {$s_{N}$};
\node at (2.25,0) [bluedot] {};
\node at (3.25,0) [bluedot] {};
\node at (5.8,0) [bluedot] {};
\node at (6.75,0) [bluedot] {};

\draw [thick, ->, green] (3.25,0) to [out=115, in=55] (2.5,0);
\draw [thick, ->, green] (5.8,0) to [out=115, in=55] (4.9,0);
\draw [thick, ->, green] (6.75,0) to [out=115, in=55] (6.2,0);

\node [above] at (2.9,0.17) {$\lambda_1$};
\node [above] at (5.3,0.17) {$\lambda_{N-2}$};
\node [above] at (6.49,0.17) {$\lambda_{N-1 }$};

\node [below] at (2,0) {$0$};
\node [below] at (7,0) {$T$};
\node at (7,0) [dot] {};
\node at (2,0) [dot] {};
\end{tikzpicture}

Define $\lambda_0:= Id \big|_{[0,s_1]},$ $\lambda_1:[s_1,s_2]\rightarrow [s_1,t_1]$ linear, and for $i=2,...,N,$
$$
\begin{array}{lcll }
\lambda_i:& [s_i,s_{i+1}]& \longrightarrow & [t_{i-1},t_i ],\\
& t & \longrightarrow & t_{i-1} + \frac{t-s_{i}}{s_{i+1}-s_{i}} (t_i- t_{i-1}).
\end{array}
$$
Then let $\lambda:=\lambda_i$ on $[s_i,s_{i+1}], i=0,...,N.$ One obtains $\lambda \in \Lambda|_{[0,T]}$ and
\begin{eqnarray*}
|\lambda|&:=&\sup_{t\in [0,T]} |\lambda(t)-t| \leq \max_{i=0,...,N} |\lambda_i(t)-t|\\
     &\leq & \max_{i=0,...,N}|s_{i+1}-t_{i-1}| \leq 2|\op| \leq \vep.
\end{eqnarray*}
Then, by Remark \ref{equivalent norms} in the appendix and our choice of $\op_\vep,$ it follows that
\begin{eqnarray*}
\vertiii{\BX^\op ; \BX \circ \lambda}_{\infty,[0,T]} &\leq  & C \max_{i=0,...,N} \max_{f\in \MF_{[p]}} \sup_{t\in[s_i,s_{i+1})} |(\BX^{\op})^f_t - \BX^f_{ \lambda_i(t)}|^{\frac{1}{|f|}}  \\
&\leq  & C \max_{i=1,...,N} \left( \vep \vee \max_{f\in \MF_{[p]}} \sup_{t\in[s_i,s_{i+1})} |(\BX^{\op})^f_t - \BX^f_{ \lambda_i(t)}|^{\frac{1}{|f|}}\right)  \\
&\leq  & C \max_{i=1,...,N}\left( \vep \vee  \omega_\BX(t_{i-1},t_i-)\right) \leq C \vep.
\end{eqnarray*}
The boundedness part follows by the definition of $\BX^\op.$ Indeed, for any partition $\op' $ and $f\in \MF,$
\begin{eqnarray*}
\sum_{[s,t] \in \op' }|(\BX^\op)^f_{s,t}|^p &= &\sum_{\stackrel{[s,t] \in \op'}{s\in[s_i,s_{i+1}),t\in[s_j,s_{j+1})} } |\BX^f_{s_i,s_j} |^p \leq \|X^f\|_{\frac{p}{|f|},[0,T]}^\frac{p}{|f|},
\end{eqnarray*}
where we apply $\BX_{u,v}^f=0$ if $u,v\in [s_k,s_{k+1}).$

\end{proof}

%
%
%
%
%
%
%

\begin{thm}(\textbf{Higher order Euler schemes for c\`adl\`ag RDEs})\label{appro. to rde}
Given a c\`adl\`ag $p$- rough path $\mathbf{X}$, consider the RDE
$$
Y_t= y_0+ \int_0^t F(Y_s)^\ell d\BX_s.
$$
Suppose $\left( \op^{n}\right) $ is a sequence of partitions of $\left[ 0,T\right] $, with
vanishing mesh-size. For any $[s^n,t^n) \in \op^n,$ define piecewise constant path $Y^n$ by
$$
Y_{t^n}^{n} := Y_{s^n}^{n} + \sum_{\tau \in \MT_{[p]-1}} F^\tau(Y_{s^n}^{n}) \BX_{s^n,t^n}^\tau,
$$
where $F^\tau : \R^e \rightarrow \R^e $ is defined in the argument before Theorem \ref{local sol. for brp}. In particular, for the level-$2$ rough path case,
$$
Y_{t^n}^{n} := Y_{s^n}^{n}+ F \left( Y_{s^n}^{n}\right) X_{s^n,t^n}+ DF\left( Y_{s^n}^{n} \right)  F   \left(Y_{s^n}^{n}\right) \mathbb{X}_{s^n,t^n}.
$$
Then $Y^n$ converges to $Y$ in the Skorokhod sense, with uniform bounded $p$-variation. In particular, the convergence holds in $p'$-variation metric of Skorokhod type, any $p'>p.$

\end{thm}

\begin{proof}
For simplicity of notation only, we write the proof for the level-$2$ rough path case. For a partition $\op^n,$ define $\BX^{\op^n}$ as before. Then according to Lemma \ref{converge of discrete rp}, $\mathbf{X}^{n} \rightarrow \BX$ in the
Skorokhod sense (in uniform sense if $\BX$ continuous), and for $\left[ s,t\right) \in \op^{n}$, (we omit the obvious superscript $n$ for $s,t$)
\begin{eqnarray*}
\Delta _{t}\mathbf{X}^{n} &:= &\lim_{u\uparrow t}\mathbf{X}_{u,t}^{n}:= \lim_{u\uparrow t}\left( \left( \mathbf{X}_{u}^{n}\right)
^{-1}\otimes \mathbf{X}_{t}^{n}\right)  \\
&=&\lim_{u\uparrow t} \left( X^n_{u,t}  ,\X_{u,t}^n \right) \equiv  \left( X_{s,t},\mathbb{X}_{s,t}\right).
\end{eqnarray*}
Now consider c\`adl\`ag RDEs $dY^{n}=f\left( Y^{n}\right) d\mathbf{X}^{n}$ on $[0,T]$ with initial condition $y_0.$ For any interval $[s,t]\in \op^n,$ by our very definition of rough integral,
$$
\int_s^r f(Y_r) d\BX^n_r \equiv 0, \ \ r\in [s,t),
$$
which implies $Y^n$ is constant on $[s,t),$ and
\begin{eqnarray*}
Y_{t}^{n}&=& Y_{t-}^{n}+ f \left( Y_{t-}^{n} \right) \Delta_{t} X^n + Df\left( Y_{t-}^{n}\right) f   \left( Y_{t-}^{n}\right) \Delta_{t} \mathbb{X}^n \\
         &=& Y_{s}^{n}+ f \left( Y_{s}^{n}\right) X_{s,t}+ Df\left( Y_{s}^{n} \right)  f   \left(Y_{s}^{n}\right) \mathbb{X}_{s,t}.
\end{eqnarray*}
Then by our continuity results, i.e. Theorem \ref{converge for RDEs}, $Y^{(n)}\rightarrow Y.$

\end{proof}

%
%


\section{Random RDEs and SDEs} \label{sec:SDE}

\subsection{Weak convergence for random RDEs}

Suppose $(\BX^n) = (\BX^n (\omega))$ is a sequence of random $p$-rough path. The It\^o-lift of semimartingales, with $p \in (2,3)$ is a natural level$-2$ example. But one can also construct  random rough path
directly, e.g. as L\'evy processes, with values in the $[p]$-truncated Butcher group, 
 provided the triplet satisfied certain structural assumption. (For instance, in a level-$4$
setting, there must not be a Brownian component on the third and forth level.) A full characterization of admissible triplets is found in \cite{FS17, Che17}. Other well-known examples are the Stratonovich lift of certain Gaussian processes, including fractional Brownian motion with $H>1/4$, since geometric rough paths canonically embed in (branched) rough paths. Construction via non-local Dirichlet forms (in the spirit of \cite[Ch.16]{FV10}) are clearly also possible.

\medskip
Having made the point that random RDEs include, yet go (far) beyond SDEs driven by semimartingales, the following limit theorem is of central importance and constitutes a significant generalizing of classical limit theorem for It\^o-SDEs, as
established by Kurtz-Protter \cite{KP91}, Jacubowski et al. \cite{JMP89}. In a sense (made very precise in Section \ref{sec:UCVUT} below), their ``UCV/UT condition'' is replaced by a tightness condition for rough path norms. Given our preparations, the proof is immediate.

\begin{thm}(\textbf{convergence for random RDEs})\label{r.rde}
Let $1 \le p < \infty$ and $F\in C_b^{[p]+1}$. Consider random, c\`adl\`ag $p$-rough paths $\BX^n \rightarrow \BX$ weakly (or in probability) under the uniform (or Skorokhod) metric, with $\{\|\BX^n\|_{p,[0,T]}(\omega)\}$ tight. 
Let $Y^n$ solve random RDEs
\begin{equation}
dY^n= F(Y^n)^- d\BX^n    \label{equ:rRDE}
\end{equation}
and $Y$ solve the same one driven by $\BX$ with the same initial value $y_0.$ Then the random RDE solution $Y^n$ converges weakly (or in probability) to $Y$ under the uniform (or Skorokhod) sense. Moreover, $\{ \left\Vert Y^{n}\right\Vert _{p,\left[ 0,T%
\right] }\left( \omega \right) :n\geq 1 \} $ is tight and one also has the weakly (or in probability) convergence in $p^{\prime }$-variation uniform (or Skorokhod) metric for any $p^{\prime }>p.$

\end{thm}

\begin{rem} (\textbf{Adding a drift vector field})  \label{rem:drift}
We could have, in the above theorem and throughout the paper, studied equations with explicit drift term, say $dY= F_0(Y^-) dt + F(Y^-) d\BX$. The cheap way to do this is to rewritte the right-hand side as $(F_0,F) (Y^-) d(t, \BX )$, in terms of a time-space rough path, which can be canonically defined. This however, requires $F_0$ to have the same regularity as $F$, see next remark. A direct analysis, which amounts to add a first order ``Euler'' term in all expansions used in the analysis, actually shows that $F_0 \in C_b^1$ will be sufficient. (In the geometric rough path case, this was spelled out in \cite{FV10}.)
\end{rem}

\begin{rem} (\textbf{Regularity of coefficients})
The regularity assumptions in Section 3,4 on $F$ could be mildly sharpened to $F \in C^{p+}$, in the rough path literature \cite{LQ02, LCL07, FV10} this is written $F \in Lip^\gamma, \gamma > p$.
\end{rem}

\begin{rem} (\textbf{Marcus canonical RDEs}) A Marcus version of this limit theorem was shown in \cite{CF17x}. (Neither implies the other and the required techniques are different.)
\end{rem}

\begin{proof}

We only prove the convergence in probability version; the weak convergence version is quite similar. For any $p'>p,$ because of the continuity of solution map, one only needs to show the convergence of $\BX^n$ to $\BX$ in probability under $p'$-variation or its the corresponding Skorokhod metric, which follows by the tightness of $\|\BX^n\|_{p,[0,T]}$ and interpolation. Indeed, one has for any $\vep>0,$
\begin{eqnarray*}
P(\{\|\BX^n;\BX \|_{p'} > \vep\} ) &\leq& P(\{\|\BX^n;\BX \|_{p} > \vep, \|\BX^n \|_{p} \leq L\} ) + P(\{\|\BX^n \|_{p} > L \})\\
&\leq& P( \{ \|\BX^n;\BX \|_{\infty}^{1-\frac{p}{p'}} > \vep/(L^{\frac{p}{p'}}) \} )+ P(\{\|\BX^n \|_{p} > L \})
\end{eqnarray*}
where we apply the interpolation inequality (Lemma \ref{interpolation}) in the last inequality. Then one only needs to choose $L$ large for the second term to be small and $n$ large to make the first term small.

\end{proof}

%
%



\subsection{C\`adl\`ag semimartingales as rough paths}

Given a semimartingale $X,$ we want to lift it to a rough path $\BX$.  To be consistent with SDEs in It\^{o}'s sense, want
$$\BX = \left(X, \int X^- \otimes dX \right) \equiv \left(X, \int X \diamond dX - \tfrac{1}{2} [X] \right), $$
which exhibits the It\^o lift as (harmless, level-$2$) perturbation of the Marcus lift studied in \cite{CF17x}. There it was seen that the Marcus
lift of a general semimartingale is a.s. a geometric $p$-rough paths, any $p>2$. It is now immediate that the same (short of geometricity)
is true for the It\^o lift. To summarize

\begin{thm} With probability one, the It\^o-lift $\BX$ of $\R^d$-valued c\`adl\`ag semimartingale $X$ is a c\`adl\`ag $p$-rough path, $p \in (2,3)$.
\end{thm}

\begin{rem} For the special case of L\'evy processes, see also \cite{Wil01, FS17}.  \end{rem}

%
%
%

We also note the validity of a BDG estimates as follows.

\begin{thm}(\textbf{BDG inequality for It\^o local martingale rough paths})\label{BDG}
Let $X$ be a $\R^d$-valued local martingale with It\^o-lift $\BX$, as above. Fix $p>2$ and a convex, moderate function $\phi$ (Example: $\phi(x)=x^q, q \ge 1$). 
Then there exist $c,C$ such that for any $q\geq1,$
$$
c\E \left[ \phi (  |[X]_\infty|^{\frac 1 2}) \right] \leq \E \left[ \phi ( ||| \BX |||_{p} ) \right] \leq C\E \left[ \phi (  |[X]_\infty|^{\frac 1 2}) \right].
$$
\end{thm}

\begin{rem} This estimates extends simultaneously L\'epingle's classcial BDG inequality (with $ \| X \|_{p} $ instead of $ ||| \BX |||_{p}$) and a BDG inequality for
{\it continuous} semimartingale rough paths \cite{FV08b}.
\end{rem}


\begin{proof}
The conclusion follows from the Marcus version, i.e. Theorem 4.7 in \cite{CF17x}. Indeed, by the relation between It\^{o} and Marcus lift,
and the fact that we work with the {\it homogenous} norm
$$
      ||| \BX |||_{p} \asymp \| X \|_{p} + \| \X \|^{1/2}_{p/2} \ ,
$$
the proof is readily finished.

%
%
%

\end{proof}

%
%
%
%
%
%
%

Furthermore, one has the following equivalence between RDEs and SDEs driven by semimartingales. The Marcus version of this result is built in \cite[Prop. 4.15]{CF17x}.

\begin{prop}(\textbf{equivalence of RDEs and SDEs})\label{equivalent DE}
Suppose $X$ is a semimartingale and $\BX$ is its It\^{o}'s lift as above. $F \in C_b^3.$ Then the solution for the random RDE
$$
dY_t = F(Y_t)^\ell d\BX_t, \ \ \ \ Y_0= y_0,
$$
agrees, with probability one, with the It\^{o}'s SDE
$$
dY_t = F(Y_t^-)  dX_t, \ \ \ \ Y_0= y_0 \ .
$$
\end{prop}

\begin{proof}

According to the proof of Proposition 4.15 in \cite{CF17x}, one has
$$
\int_0^t F(Y_s^-)^\ell d\BX_t = \int_0^t F(Y_s^-) dX_s, \ \ \  a.s.,
$$
which, by the right-continuity of $\BX$ and the sewing lemma for pure jumps, implies
$$
\int_0^t F(Y_s)^\ell d\BX_t = \int_0^t F(Y_s^-) dX_s, \ \ \  a.s..
$$
Then the statement follows by (strong) uniqueness of It\^o SDE solutions.

\end{proof}


\subsection{Manual for checking tightness} \label{sec:cat}

It\^o-lifted general semimartingales are natural random rough paths, however there are many other stochastic processes (with jumps) that admit a rough path interpretation. According to our convergence result for random RDEs, the key assumption is $$ (*) \ \ \ \ \{\|\BX^n\|_{p,[0,T]} \} \text{ is tight.} $$ This invites comparison with the UCV/UT (``uniformly controlled variation/uniform tightness'') conditions familiar from semimartingale theory (and more specifically the stability theory for It\^o SDEs). Remarkably, the two theories match perfectly in the sense that UCV/UT precisely implies $(*)$ and we detail this in the first subsection below. This in turn allows to avoid the UCV/UT condition altogether, which is of important for instance in homogenization theory (see Remark \ref{rem:KM}). In essence, one needs to check $(*)$ by hand, but -  fortunately -  we have a variety of tools available that can be used, also in a non-semimartingale context, as discussed in subsequent parts of this section.



\subsubsection{Semimartingale rough paths under UCV/UT condition} \label{sec:UCVUT}

Now we recall some classical SDEs convergence results from \cite{KP96}, and then we show how our result above consistent with the classical one. First, recall uniformly controlled variation(UCV) condition for a sequence of semimartingales $(X^n)_{n \geq 1}.$ For a path $X \in D([0,T],\R^d)$ and a constant $\delta>0,$ define
$$
X_t^\delta:= X_t - \sum_{s\leq t} (1-\delta/|X_{s-,s}|)^+ X_{s-,s}.
$$
We remark that $X^\delta$ is a version of $X$ with jumps bounded by $\delta,$ which implies the mapping $X \rightarrow X^\delta $ is continuous in both uniform and Skorokhod sense.

\begin{defi}(\textbf{UCV condition})
A sequence of semimartingales $\left(X^n \right)_{n\geq 1}$ satisfies uniformly controlled variation(UCV) condition if there exists a constant $\delta $ such that for any $\alpha>0,$ $X^{n,\delta }$ defined by the above formula has a martingale decomposition $X^{n,\delta } = M^{n, \delta } + A^{n, \delta },$ along with stopping times $\tau^{n, \alpha },$ such that for any $t>0,$

$$
\sup_{n \geq 1} P(\{ \tau^{n, \alpha } \leq \alpha \} ) \leq \frac1\alpha ,\ \ \  \sup_{n \geq 1} \E\left[[ M^{n,\delta } ]_{t\wedge \tau^{n, \alpha }}  + |A^{n, \delta } |_{1-v,[0,t\wedge \tau^{n, \alpha } ]}  \right] < \infty.
$$

\end{defi}

If $\left(X^n \right)_{n\geq 1}$ is a set of semimartingales converges to $X, $ then the lifted rough path $\X^n$ also converge to $\X$ in the same topology by the continuity of stochastic integral, i.e. Theorem 7.10 in \cite{KP96}. Precisely, one has the following proposition.

\begin{prop}(\textbf{convergence of lifted semimartingales})\label{conver. of lifted semimar}
Suppose $\left(X^n \right)_{n\geq 1}$ is a set of semimartingales converges in probability(or weakly) to a semimartingale $X $ under Skorokhod topology. Furthermore, suppose $\left(X^n \right)_{n\geq 1}$ satisfies UCV condition. Then the lifted rough path $\BX^n$ also converges in probability(or weakly) to $\BX$ under the Skorokhod topology, and moreover, the set $\{\|\BX^n \|_{p,[0,T]} | n \geq 1 \}$ is tight for any $p>2.$

\end{prop}

\begin{proof}

The convergence part follows from the continuity of stochastic integral under Skorokhod topology, i.e. Theorem 7.10 in \cite{KP96}. The tightness part follows from the same argument as Theorem 4.11 in \cite{CF17x}. Indeed, notice for any semimartingale $Z,$ one has the identity
$$
\Z^{\mathbf{M}}_{s,t}= \Z_{s,t} + \frac12 [ Z ]_{s,t},
$$
where $\Z^{\mathbf{M}} $ is defined by the Marcus kind integration and $\Z$ is the It\^{o}'s version. It is proved in Theorem 4.11 \cite{CF17x} that the lifted rough path of Marcus kind $\{  \|(\BX^n)^M\|_{p,[0,T]} \}_{n}$ is tight. Since $[ X^n ]$ has bounded variation, one only needs to show $\{ |[ X^n ]_T | \}_n $ is tight, which is implied by the fact $X^n$ converges to $X$ in probability(or weakly) under Skorokhod topology.

\end{proof}

By the convergence of random RDEs and the convergence of lifted semimartingales, now we give a pathwise interpretation of the classical convergence result about SDEs solutions (e.g. Theorem 8.1 in \cite{KP96}).

\begin{thm}(\textbf{continuity for solutions of SDEs under UCV/UT})\label{continu.under UCV}
Suppose $F\in C_b^3$ and $\left(X^n \right)_{n\geq 1}$ is a set of semimartingales converges in probability(or weakly) to a semimartingale $X $ under Skorokhod topology which satisfy UCV condition. Denote $Y^n$ as solutions to SDEs
$$
dY^n= F(Y^n )^- d X^n,
$$
and $Y$ solves the same one driven by $X.$ Then $Y^n$ converges in probability(or weakly) to $Y$ under the Skorokhod topology. Moreover, $\{ \left\Vert Y^{n}\right\Vert _{p,\left[ 0,T\right] }\left( \omega \right) :n\geq 1 \} $ is tight and one also has convergence in probability(or weakly) under $p^{\prime }$-variation norm of Skorokhod type for any $p^{\prime }>p.$

\end{thm}

\begin{proof}

According to Proposition \ref{equivalent DE}, $Y^n$ solves $dY^n= F(Y^n )^- d \BX^n, \ a.s..$ Then by applying Proposition \ref{conver. of lifted semimar} and Theorem \ref{converge for RDEs}, one finishes the proof.

\end{proof}

\bigskip

The following example shows that Theorem \ref{continu.under UCV} yields sharp results, using \cite{KP91,KP96} as benchmark.

\begin{example}(\textbf{random walk benchmark example})\label{random walk}
Given a sequence of partitions $\op_n :=\{0=t_0< t_1^n  <t_2^n  < \cdots < t_{N_n -1}^n< t_{N_n}^n=T \},$ for a fixed $n,$ suppose for any $t\in [t_k^n,t_{k+1}^n), k =0,...,N_n-1,$
$$
X^n_{0,t}=\sum_{i=0}^{k-1} \xi_i^n ,
$$
where $\{\xi^n_i\}_{i=0}^{N_n-1}$ are $\R^d-$valued independent random variables with null means and finite quadratic variation. The It\^o lift is clearly given by, for $t\in [t_k^n,t_{k+1}^n)$,
$$
\BX^n_{0,t }:= (X^n_{0,t}, \X^n_{0,t}) := (\sum_{i=0}^{k-1} \xi_i^n,\sum_{i=0}^{k-1} \sum_{j=0}^{i-1} \xi_i^n \xi_j^n),
$$
In particular, one has
\[
\X_{s,t}^n =\left\{
\begin{array}{ll}
0, & s, t \in [t_k^n,t_{k+1}^n], \\
\xi^n_k \xi^n_{k+1},& s=t_k^n, t=t_{k+2}^n.
\end{array}
\right.
\]
Furthermore, suppose $\sup_{n} \sum_{i=0}^{N_n-1} \E(|\xi_i^n|^2) < \infty,$ which implies $\{X^n\}_n$ satisfies the UCV condition. Indeed, one has $\E([X^n]_t) \leq \sum_{i=0}^{N_n-1} \E(|\xi_i^n|^2) < \infty $ uniformly over $n.$ Then one has the tightness of $\{\|\BX^n \|_{p,[0,T]}   \}_n$ for any $p>2,$ and the continuity of solutions for SDEs driven by $\{X^n\}_n.$ In particular, if $\{\zeta_i \}_{i \geq 0}$ be i.i.d. $\R^d-$valued random variables with $\E[\zeta_i]=0$ and $\E[\zeta_i \zeta_j]=Q $ and
$$
X^n(t):= n^{-\frac12} \sum_{i=0}^{\lfloor nt \rceil-1} \zeta_i, \ \ \ \
\X^n(t):= n^{-1} \sum_{i=0}^{\lfloor nt \rceil-1}\sum_{j=0}^{i-1} \zeta_j \zeta_i= \int_0^t (X^n_{r})^- dX^n_r.
$$
Note that $\X^n_{s,t}=0, $ if $s,t \in [t_i^n, t_{i+1}^n].$ Define
\begin{eqnarray}\label{Y_n}
Y_{t}^{n} := Y_{t^n_{i-1}}^{n}+ F \left( Y_{t^n_{i-1}}^{n}\right) \xi^n_{i-1} , \ \ \ t\in[t^n_{i}, t^n_{i+1}),
\end{eqnarray}
which solves random RDEs/SDEs $dY^n = F(Y^n)^-   d X^n$ and $Y$ solves the differential equation driven by Brownian motion,
$$
dY= F(Y) Q^{\frac12} dB.
$$
By classical results (e.g. Donsker's theorem plus UCV/UT criterion) one has $(X^n, \X^n) \implies (B, \B)=(B, \int B dB)$. 
Then by our Theorem \ref{continu.under UCV}, one
obtains $Y_n \implies Y$, meaning  weak convergence under uniform topology. This is exactly what classical ``UCV/UT theory'' yields, applied to the above sequence of SDEs.  The point is that, with regard to integrability of $\zeta$'s,
application of our Theorem \ref{continu.under UCV} requires only $2$ moments, as does classical theory, and this is optimal.  On the other hand, the very general Theorem \ref{continu.under UCV} remains  fully operational without any UCV/UT-type assumptions; see also Theorem \ref{thm:alaK} below for a corollary which improves on \cite{Kel16}.

\end{example}

%
%

\subsubsection{Strong Markov process under Manstavi\v{c}ius criterion}

A (branched) rough path takes values in a (truncated) Butcher group equipped with a left-invariant, homogenous metric as discussed in Section \ref{sec:not} and the appendix. As seen in \cite[Ch.16]{FV10} (uniformly elliptic symmetric diffusions) and also \cite{FS17} (L\'evy processes and jump diffusions) such lifted Markov processes can be viewed as Markov processes with values in such a  Lie group, equipped with a metric that precisely allows to express the $p$-variation rough path regularity.

\medskip

It follows that $p$-variation criteria for Markov processes with values in a metric space become an immediate source of the desired tightness estimates. Specifically, we introduce the criterion for the a.s. finite $p$-variation of a strong Markov process developed by Manstavi\v{c}ius \cite{Man04}, first employed in a L\'evy rough path setting by \cite{FS17}, and then the tightness criterion of step-function version developed by Chevyrev \cite{Che17}. Let $\BX_t,$ $t\in [0,T]$ be a c\`adl\`ag strong Markov process with values in a Polish space $(E,d)$.
Write $P_{s,t}(x,dy), x\in E$ for the transition probability. For any $h>0,$ $a >0,$ denote
$$
P^{s,x}(d(\BX_s, \BX_t)> a):= P_{s,t}(x,\{y: d(x,y)\geq a \}), \ \ \ \alpha(h,a):= \sup_{ {x\in E}\atop { s \leq t \leq (s+h) \wedge T} }P^{s,x}(d(\BX_s, \BX_t)\geq a).
$$
Consider the oscillations of $\BX(\omega)$ by defining
$$
\nu_b(\BX)(\omega):= \sup_{t_i \in [0,T], i=1,...,2k} \{k | t_1 < t_2 \leq t_3 < \cdots < t_{2k}, d(\BX_{t_{2k-1}}, \BX_{t_{2k}} ) > b \},
$$
i.e. the number of oscillations of size larger than $b.$ Since $\BX$ c\`adl\`ag, one has $\nu_b < \infty , a.s.$ for any $b>0.$ Denote
$$
M(\BX):= \sup_{t\in[0,T]} d(\BX_0,\BX_t)),
$$

and the contribution to $\|\BX\|_{p,[0,T]}$ from oscillations larger than $b$ is controlled by $(2M(\BX))^p \nu_b(\BX).$ Furthermore, by taking $2^{-r}$ as bounds for oscillations, one has
$$
\|\BX\|_{p,[0,T]}^p \leq \sum_{r>r_0}^\infty 2^{-rp+p } \nu_{2^{-r}}(\BX) + (2M(\BX))^p \nu_{2^{-r_0}}(\BX),
$$
for any $r_0 \geq 0.$ We have the following criterion for tightness.

\begin{prop}(\textbf{tightness for strong Markov under Manstavi\v{c}ius condition})
Suppose $\{\BX\}_{\BX \in \mathcal{M}}$ is a set of strong Markov processes with each element $\BX$ defined as above, and there exist constants $a_0 >0, K>0$ such that for all $h \in [0,T]$ and $a \in (0,a_0],$
$$
\sup_{\BX \in \mathcal{M}} \alpha(h,a) \leq K \frac{h^\beta}{a^\gamma}.
$$
where $\gamma>0, \beta>1.$ Furthermore, assume $\{M(\BX)  \}$ is tight. Then for any $p > \frac{\gamma}{\beta},$ $\{\|\BX\|_{p,[0,T]} \}_{\BX \in \mathcal{M}}$ is tight.
\end{prop}

\begin{proof}
According to the proof of Theorem 1.3 in \cite{Man04}, one has
$$
 \|\BX \|_{p,[0,T]}^p  \leq \sum_{r>r_0} 2^{-rp}  \nu_{2^{-r-1}}(\BX) + (2M(\BX))^p \nu_{2^{-r_0-1}}(\BX), \ \ \ \sup_{\BX\in \mathcal{M}} \left(\E[\sum_{r >r_0} 2^{-rp}  \nu_{2^{-r-1}}(\BX)] + \E[\nu_{2^{-r_0-1}}(\BX)]\right) < \infty,
$$
where $r_0$ is the largest integer such that $2^{-r_0-2} \geq  a_0.$ Then the tightness follows from our assumption on the tightness of large oscillation and classical Markov's inequality.

\end{proof}

\begin{rem}
Concerning the tightness of $M(\BX),$ one may check a basic criterion for a sequence of c\`adl\`ag processes $\{\BX^n\}_n$ in Theorem 4.8.1 of \cite{K11}, or the so called Aldous condition, i.e.
$$
d( \BX^n_{\tau_n} ,\BX^n_{\tau_n+h_n} ) \rightarrow 0 \text{ in probability,}
$$
for any sequence of bounded $\MF^n$-stopping times $\tau_n$ and any sequence of positive numbers $h_n$(see Theorem 4.8.2 in \cite{K11}). For example, one has the following criterion according to \cite{Che17}.
\end{rem}

\begin{prop}(\textbf{Corollary 4.9 in \cite{Che17}})
Suppose $(\BX^n)_{n \geq 1}$ a sequence of c\'adl\'ag strong Markov processes defined as above. Assume that
\begin{enumerate}
\item for any rational $h \in [0,T],$ $(\BX_h^n)_n$ is tight.\\
\item there exist constants $K , \beta, \gamma, b>0,$ such that for all $\delta \in (0,b], h>0,$
$$
\sup_n \sup_{s \in R_{\BX^n}} \sup_{x\in E} \sup_{t\in [s,s+h]} P^{s,x}(d(\BX^\op_s, \BX^\op_t)> \delta) \leq K \frac{h^\beta}{\delta^\gamma},
$$
\end{enumerate}
where $R_{\BX^n}=[0,T]\Z_{\BX^n},$ with $\Z_{\BX^n}$ denoting the union of so called stationary intervals, see Definition 4.4 in \cite{Che17}.
Then for any $p> \gamma/\beta,$ $(\|\BX^n\|_{p,[0,T]})_n$ is tight.

\end{prop}


\subsubsection{Discrete random rough paths under Besov-type condition} 
\label{sec:Kelly}

Suppose $\op  :=\{0=t_0< t_1   <t_2   < \cdots < t_{N  -1} < t_{N } =1 \}$ is a partition and $(X , \X )$ is a piecewise constant random path, i.e. $(X_{0,t}, \X_{0,t})$ constant on $[t_i,t_{i+1}),$ which takes values on $\R^d \oplus \R^{d\times d}$, satisfies ``Chen's relation'', and have the following moment bounds
\begin{equation}\label{kolmog moment jump}
\left( \E[  | X_{t_k, t_j}|^{q} ] \right)^{\frac 1 {q}} < C |t_k - t_j|^{\beta}, \ \ \ \left( \E[  | \X_{t_k, t_j}|^{\frac {q} 2} ] \right)^{\frac 2 {q}} < C |t_k - t_j|^{2\beta},
\end{equation}
with $q>1, \beta> \frac 1 {q}.$ for all $t_k, t_j \in \op$.


Then one has the following criterions for the tightness of such piecewise constant processes.

\begin{prop}\textbf{(tightness for discrete rough paths)}\label{discrete criterion}
Let $(X, \X)$ be a c\`adl\`ag piecewise constant random path according to $\op$, i.e. for any $t \in [t_i,t_{i+1}),$
$$
(X_{0,t},\X_{0,t})(\omega)= (X_{0,t_i}, \X_{0,t_{i}})(\omega).
$$
Suppose $(X, \X)$ satisfies condition \eqref{kolmog moment jump} with $\beta\leq \frac12$. Then for any $   p \in ( \frac 1 \beta,q),$ $(X, \X)$ has finite $p$-variation a.s., and
$$
\E[\|X\|_{p,[0,1]}^q] + \E[\|\X\|_{\frac p 2,[0,1]}^{\frac q 2 } ]  < M
$$
with $M$ depending only on $C, p , q, \beta .$ In particular, if $\{\op^\alpha \}_{\alpha\in \mathcal{A}}$ is a set of partitions and for each $\op^\alpha,$ $(X^\alpha, \X^\alpha)$ is a c\`adl\`ag piecewise constant random rough path satisfying \eqref{kolmog moment jump} according to $\op^\alpha$ with uniform $C.$ Then one has
\begin{equation}
\sup_{\alpha \in \mathcal{A}} \left(\E[\|X^\alpha\|_{p,[0,1]}^q] + \E[\|\X^\alpha\|_{\frac p 2,[0,1]}^{\frac q 2 } ]\right)  < M
 \end{equation}
with $M$ depending only on $C, p , q, \beta ,$ and $\{  \|\BX^\alpha \|_{p,[0,1]} \}_\alpha$ is tight.
\end{prop}

\begin{proof}

Define a continuous random rough path $(\tX, \tilde{\X})$ which agrees with $(X,\X)$ at $t_j$ and linear on $[t_j, t_{j+1}],$ i.e.
\begin{eqnarray*}
\tX_t &:= & \left\{
\begin{array}{ll}
X_{t_j} & \text{ if } t=t_j\\
 X_{t_j} + \frac{t-t_j}{t_{j+1} - t_j} X_{t_j, t_{j+1}} & \text{ if } t\in [t_j,t_{j+1}].
\end{array}
\right.\\
\tilde{\X}_{0,t} & :=&  \left\{
\begin{array}{ll}
\X_{0,t_j}, & \text{ if } t=t_j\\
 \X_{0,t_j} + \frac{t-t_j}{t_{j+1} - t_j} (\X_{0,t_{j+1}}-\X_{0,t_j}) & \text{ if } t\in [t_j,t_{j+1}].
\end{array}
\right.
\end{eqnarray*}

One obviously has that $\|\BX\|_{p,[0,T]} \leq \|\tBX\|_{p,[0,T]},$ which implies one only needs to show the conclusion for $\tBX.$ Note that $\tBX$ is a continuous and piecewise linear. (The construction
is more involved in a geometric setting (keyword ``log-linear interpolation'' \cite{CF17x}). The reason we get away with a seemingly simple construction here is that $\R^d \oplus \R^{d \times d}$ can be identified with the Lie algebra of the (level-$2$ truncated) Butcher group.) In any case, the computations are not difficult: we claim that for any $s,t \in [0,T],$ one has
$$
\left(\E[|\tX_{s,t}|^q]\right)^{\frac 1 q} \leq M |t-s|^{\beta}, \ \ \ \left(\E[|\tilde{\X}_{s,t}|^{\frac q 2}]\right)^{\frac 2 q} \leq M |t-s|^{2\beta},
$$
Indeed, for any $s\in[t_j, t_{j+1}), t\in [t_k, t_{k+1}),$ for the first estimate, one has
\begin{eqnarray*}
\E[ |\tX_{s,t}|^q] &\leq & 3^q \E[(|\tX_{s,t_{j+1}}|^q + |\tX_{t_{j+1} ,t_k}|^q  + |\tX_{ t_k, t}|^q)]\\
                    &\leq & 3^q \left( \big(\frac{ t_{j+1}-s}{t_{j+1} - t_j}\big)^q \E[ |X_{t_j, t_{j+1}}|^q] + C^q |t_k- t_{j+1}|^{q\beta} + \big(\frac{t-t_k}{t_{k+1} - t_k}\big)^q \E[ |X_{t_k, t_{k+1}}|^q] \right)\\
                    &\leq & 3^{q+1} C^q |t-s|^{q \beta}.
                    \end{eqnarray*}
For the second estimate, note that by ``Chen's relation,'' one has
\begin{eqnarray*}
\tilde{\X}_{s,t_{j+1}} &=& \tilde{\X}_{0,t_{j+1}}-  \tilde{\X}_{0,s}- \tX_{0,s}\tX_{s,t_{j+1}}\\
&=& \frac{t_{j+1}-s}{t_{j+1}-t_j}(\X_{0,t_{j+1}}- \X_{0,t_j})- (X_{0,t_j}+ \tX_{t_j,s}) \frac{t_{j+1}-s}{t_{j+1}-t_j}X_{t_j,t_{j+1}}\\
&=& \frac{t_{j+1}-s}{t_{j+1}-t_j} \X_{t_j,t_{j+1}}   +  \tX_{t_j,s} \frac{t_{j+1}-s}{t_{j+1}-t_j}X_{t_j,t_{j+1}}\\
&=& \frac{t_{j+1}-s}{t_{j+1}-t_j} \X_{t_j,t_{j+1}} + \frac{t_{j+1}-s}{t_{j+1}-t_j}\frac{ s-t_j}{t_{j+1}-t_j}X_{t_j,t_{j+1}}X_{t_j,t_{j+1}}.
\end{eqnarray*}
Then one obtains that
\begin{eqnarray*}
\E[|\tilde{\X}_{s,t_{j+1}}|^\frac q 2] &\leq& M \Big( (\frac{t_{j+1}-s}{t_{j+1}-t_j})^{\frac q 2} |t_{j+1}-t_j|^{q\beta}+ (\frac{t_{j+1}-s}{t_{j+1}-t_j}\frac{ s-t_j}{t_{j+1}-t_j})^{\frac q 2} |t_{j+1}-t_j|^{q\beta}\Big)\\
&\leq& M (t_{j+1}-s)^{q\beta} .
\end{eqnarray*}
By similar argument for $\tilde{\X}_{t_k, t},$ one has
\begin{eqnarray*}
\E[|\tilde{\X}_{s,t}|^\frac q 2] &\leq & C_p \E\left[ \left( |\tilde{\X}_{s, t_{j+1} } |^{\frac q 2} + |\tilde{\X}_{  t_{j+1},t_k } |^{\frac q 2} + |\tilde{\X}_{t_k, t}|^{\frac q 2} + |\tX_{s, t_{j+1}}\tX_{ t_{j+1},t_k}|^{\frac q 2} + |\tX_{s, t_{k}} \tX_{ t_{k},t }|^{\frac q 2} \right)\right]\\
&\leq & M |t-s|^{q \beta}.
\end{eqnarray*}
We recall from last section that $\tBX$ indeed takes values in a Butcher group equipped with metric $d(\tBX_{0,s},\tBX_{0,t}) \asymp |\tX_{s,t}|  + |\tilde{\X}_{s,t}|^\frac12.$ Then for any $  p\in ( \frac 1 \beta,q),$ one has
$$
\E \left( \iint_{[0,T]^2} \frac{d(\tBX_{0,u},\tBX_{0,v})^q}{|v-u|^{\frac q p +1}} du dv\right) \asymp \iint_{[0,T]^2}  \frac{|v-u|^{q\beta}}{|v-u|^{\frac q p +1}} du dv  \leq M ,
$$
which implies, by a Besov variation embedding theorem (\cite{FV06}, see also \cite[Cor. A.3]{FV10}),
$$
\E[\|\tBX\|_{p,[0,T]}^q] \leq  \E \left( \iint_{[0,T]^2} \frac{d(\tBX_{0,u},\tBX_{0,v})^q}{|v-u|^{\frac q p +1}}du dv\right) \leq M.
$$
The tightness follows from the classical Chebyshev inequality.

\end{proof}

For the sake of comparison with previous works \cite{Kel16}, \cite[Thm 9.3]{KM16} let us fully spell out the combination of the above tighness criterion (in case $\beta = 1/2$), stability of random RDEs and their interpretation as It\^o SDEs in a semimartingale setting.
For notational simplicity only, we take $[0,T]=[0,1]$ and assume each $X^n$ is pure jumps, with jump times $\{ j/n: j=1,...,n \}$. We have
\begin{thm} \label{thm:alaK}

Suppose that $(X_n,\X_n) \implies (X,\X)$ in $D([0,T], \R^d \oplus \R^{d\times d} ) $
where $X$ is a semimartingale started at $0$, and $\X$ can be written, by assumption, as
\[
\X(t) = \int_0^t X^-(s)  \,dX(s) + \Gamma (t),
\]
for some, possibly random, $\Gamma \in D^1 ([0,T], \R^{e\times e})$. Suppose there exist $C>0$ and $q>1$ such that
%
\begin{equation}\qquad
\label{e:disc_kolm} \bigl\|X_n(j/n,k/n) \bigr\|_{2q} \leq C \bigl\lvert
\frac{j-k}{n}\bigr\rvert ^{1/2} \quad\mbox{and} \quad\bigl\|\X_n(j/n,k/n)
\bigr\|_{q} \leq C \biggl\lvert \frac
{j-k}{n}\biggr\rvert ,
\end{equation}
hold for all $n\ge1$ and $j,k = 0, \ldots, n$.
Given $F : \R^d\to\R^{d\times e}$  we define
$Y_n\in D([0,T],\R^d)$ by $Y_n(t)=Y_{[nt],n}$ where
\[
Y_{j+1,n}=Y_{j,n}+F (Y_{j,n})
\biggl(X_n\biggl({\frac
{j+1}{n}}\biggr)-X_n\biggl({
\frac{j}{n}}\biggr) \biggr),\qquad Y_{0,n}=\xi.
\]
Then, provided $F \in C^{2+}$, we have $Y_n \implies Y$ in $ D([0,T],\R^d)$, where

\[
dY =  \sum_{\alpha,\beta,\gamma} \Gamma ^{\beta\gamma
}
\biggl(
\partial ^\alpha F^{\beta}(X) F^{ \alpha\gamma}(X) \biggr) \,dt +
F(X) \,dW, ,\qquad Y_{0}=\xi.
\]

\end{thm}

\begin{rem} One readily includes a drift vector field $F_0 \in C^1$, cf. Remark \ref{rem:drift}.

\end{rem}

\begin{proof} Combine Proposition \ref{discrete criterion} (attention, switch $q \to 2q$ so that $q>1/\beta=2$ becomes $q>1$), Theorem \ref{r.rde} and Proposition \ref{equivalent DE}.
\end{proof}


\begin{rem} \label{rem:KM}

The result \cite[Thm 9.3]{KM16}) was restricted to the case when $X$ is a Brownian motion and $\Gamma$ is deterministic of constant speed; we make no such assumption. However, the limiting rough path $\BX$ has $p$-variation, any $p>2$, but also a H\"older exponent $\frac{1}{2} - \frac{1}{2q}>0$ so that (in the limit) we do deal with continuous $p$-rough paths -
this is also the reason that \cite{Kel16} ultimately was able to rely on (H\"older, geometric) rough path theory. The $p$-variation/jump setting conveniently avoids any Stratonovich detours. More importantly, even in the limiting absence of jumps, it significantly pays off to work with $p$-variation: an underlying H\"older regularity always comes with a price in terms of moments (think: Kolmogorov criterion), whereas $p$-variation comes at no extra price (think: L\'epingle BDG). Having that explanation in mind, the above theorem indeed improves from $2q>6$ moments (\cite[Thm 9.3]{KM16}) to $2q> 2$ which allows for some direct improvement for the validity of homogenization for certain discrete fast-slow systems, see Thm 2.2 and Cor 7.4. in \cite{KM16}.
\end{rem}

\subsubsection{Stability under perturbations}

Now consider adding a c\`adl\`ag perturbation process $\Gamma^n$ to a lifted c\`adl\`ag random rough path $(\tX^n, \tilde{\X}^n)$ by $(X^n, \X^n):= (\tX^n, \tilde{\X}^n + \Gamma^n).$ We remark that
perturbation on $\tX^n$ can be handled by translation argument, i.e. $(\tX', \tilde{\X}'):=(\tX+ h , \tilde{\X}+ \int h dX + \int X dh + \int h dh).$ One may consider the perturbation $\Gamma^n$ as correction terms
for different lifts(e.g. It\^{o} v.s. Marcus), or some noise for the system, or additional information from the system(e.g. Euler v.s. Milstein). Then one has the following obvious tightness conclusion.

\begin{prop}\textbf{(tightness for random rough paths under perturbation)}\label{tightness for perturbation}
Suppose that $\{(\tX^n, \tilde{\X}^n)\}_n$ are random rough paths with $\{  \|\tBX^n\|_{p,[0,T]} \}_{n}$ tight for some $p \in [2,3)$. Let
$$
(X^n_{ t}, \X^n_{u,v}):= (\tX^n_{ t}, \tilde{\X}^n_{u,v} + \Gamma^n_{u,v}),\ \ \ t \in[0,T], \ \{u,v\}\in \{0 \leq u \leq v \leq T\},
$$
where $\Gamma^n_t$ is a c\`adl\`ag $\R^{d\times d}-$valued process with $\{\|\Gamma^n \|_{\frac p 2,[0,T]}\}$ tight. Then $\{  \|\BX^n \|_{p,[0,T]} \}_n$ is tight.

\end{prop}

Then one needs to consider the tightness of $\{\|\Gamma^n \|_{\frac p 2,[0,T]}\}$ which may be obtained by checking all the above conditions. Furthermore, we give another case for the tightness of $\{\|\Gamma^n \|_{\frac p 2,[0,T]}\}$ which may be helpful in a numerical framework. Suppose $\op_n :=\{0=t_0< t_1^n  <t_2^n  < \cdots < t_{N_n -1}^n< t_{N_n}^n=T \}$ is a sequence of partitions. For each $n,$ $\{\eta^n_i\}_{i=0}^{N_n-1}$ are $\R^{d\times d}-$valued independent random variables with null means.

\begin{prop}\textbf{(tightness for step-function like perturbation)}\label{tightness for step functions}
Suppose $\Gamma^n_t$ is a c\`adl\`ag $\R^{d\times d}-$valued process and has the form
$$
\Gamma^n_{0,t} = M_{0,t}^n + A_{0,t}^n,\ \ \ t\in [t_k^n,t_{k+1}^n),
$$
with $M_{0,t}^n = \sum_{i=0}^{k-1} \eta_i^n , t\in [t_i^n,t_{i+1}^n)$ and $\{A^n\}_n$ tight bounded variation processes satisfying UCV condition. Furthermore, suppose for $i=0,...,N_n-1,$
$$
\E [ | \eta_i^n|  ]  \leq  C |t_{i+1}^n- t_i^n| ,
$$
with $C$ uniform over $n.$ Then $\{\|\Gamma^n \|_{\frac p 2,[0,T]}\}$ is tight.

\end{prop}

\begin{proof}

Since $\Gamma^n$ is indeed a random path with finite variation a.s., it has finite $p$-variation a.s.. For the tightness, by basic inequality $(a+b)^{\frac p 2} \leq 2^{p/2}(a^\frac p 2 + b^{\frac p 2}),$ one only needs to show the tightness of $\{\|A^n\|_{\frac p 2,[0,T]}\}_n$ and $\{\|M^n\|_{\frac p 2,[0,T]}\}_n.$ The tightness of $\{\|A^n\|_{\frac p 2,[0,T]}\}_n$ is implied in the proof of Proposition \ref{conver. of lifted semimar}. Indeed, for a path $x$ with finite variation, one has $\|x\|_{p,[0,T]}^p \leq \|x\|_{1,[0,T]}^p$ for any $p\geq 1.$ By applying the definition of UCV condition and this inequality, one has
\begin{eqnarray*}
\sup_n P\left( \|A^{n,\delta}\|_{\frac p 2,[0,T]}> K  \right) \leq  \sup_n P (\tau^{n,\alpha} \leq \alpha) + \sup_n K^{-1} \E[\|A^{n,\delta}\|_{1,[0,T\wedge \tau^{n,\alpha}]}] < \vep,
\end{eqnarray*}
with $\alpha > T \vee \frac{2}{\vep}$ and $K$ chosen to be large. Since $\|A^n\|_{1 } \leq \|A^{n,\delta}\|_1 + \sum_{|\Delta_t A|> \delta}|\Delta_t A|,$ the tightness of $\{ \|A^n\|_{1 }\}_n$ follows from the tightness of $\sum_{|\Delta_t A|> \delta}|\Delta_t A|,$ which is implied by the continuity of the mapping $x \rightarrow \sum_{|\Delta_t x|> \delta}|\Delta_t x|$ and the tightness of $A^n.$ Now we prove $\{\|M^n\|_{\frac p 2,[0,T]}\}_n$ is tight. Indeed, since $\|M^n\|_{\frac p 2} \leq \|M^n \|_{1}= \sum_{i=0}^{N_n-1} |\eta^n_i|,$ one has
\begin{equation*}
P(\|M^n\|_{\frac p 2} > K) \leq K^{-1} \E(\|M^n\|_{\frac p 2}) \leq C K^{-1} \sum_{i=0}^{N_n-1} |t_{i+1}^n- t_i^n|=C K^{-1} T.
\end{equation*}

\end{proof}


\section{Appendix}

\subsection{Regularity of controls}

If a path $x$ on some metric space is continuous, it is well-known that its $p$-variation function $\omega_x(s,t)$ is continuous in both variables $s,t.$ It turns out that if $x$ is right-continuous(or left continuous), then $\omega_x$ shares similar continuities. Indeed, the proof is the same as the continuous case. We provide a short proof for the convenience of readers.

\begin{lem}(\textbf{continuity of paths gives continuity of its $p$-variation})\label{vari.continuity}
Suppose $x:[0,T] \rightarrow (E,d)$ is right-continuous with finite $p$-variation. Denote its $p$-variation as $ \omega^\frac1p(s,t) $ for $s<t$ and $\omega(s,s)=0$ for convenience. Then for any fixed $s,$ the function $\omega(s,t), t\in [s,T]$ is right-continuous, and for any fixed $t,$ $\omega(s,t), s\in[0,t)$ is also right-continuous. Specially, $$\omega(s,s+):=\lim_{t \downarrow s} \omega(s,t)=0.$$

\end{lem}

\begin{proof}
The second right-continuity is simple so we only show the first continuity.
Thanks to the super-additivity of $\omega(s,t),$ one only needs to show $\omega(s,t+):=\lim_{u \downarrow t}\omega(s,u) \leq \omega(s,t),$ for any $t\in [s,T).$ Fix a $t\in [s,T).$ For any $0< \vep <1,$ choose $H>0$ such that $|x_{u,v}|:=d(x_u,x_v)<\vep,$ for any $u,v\in[t,t+H].$ According to the definition of $p$-variation, there exists a partition $\op$ of $[s,t+H]$ such that $\omega(s,t+H)< \sum_\op|x_{u,v}|^p + \vep.$ We add $t$ into the partition $\op,$ and denote $t'\in \op$ the point closest to $t$ from the right side. Then one has
$$
\sum_\op|x_{u,v}|^p \leq \sum_{\op\cup \{t\}|_{[s,t]}} |x_{u,v}|^p +\vep  C_{\omega(0,T),p}  + \sum_{\op|_{[t',t+H]}} |x_{u,v}|^p,
$$
where the triangular inequality and the basic inequality $(a+b)^p\leq a^p +  \vep C_{\omega(0,T),p}$ for $0<a<\omega(0,T), 0<b< \vep$ are applied. Then for any $0<h<H,$ one has
\begin{eqnarray*}
\omega(s,t+h) &\leq& \omega(s,t+H)- \omega(t+h,t+H)\\
&\leq &  \sum_\op|x_{u,v}|^p + \vep - \omega(t+h,t+H)\\
&\leq & \sum_{\op\cup \{t\}|_{[s,t]}} |x_{u,v}|^p +\vep  C_{\omega(0,T),p}  + \sum_{\op|_{[t',t+H]}} |x_{u,v}|^p + \vep - \omega(t+h,t+H)\\
&\leq & \omega(s,t) + C \vep + \omega(t',t+H) - \omega(t+h,t+H)
\end{eqnarray*}
Let $h$ decrease to $0,$ and one has $\omega(s,t+)\leq \omega(s,t) + C \vep.$

\end{proof}

\subsection{Sup-norm vs. $\infty$-norm for branched rough paths}\label{app:leftinvmetric}

Recall the ``homogenous'' infinity and sup-norm for paths in a truncated Butcher group $G$.
$$
||| \BX |||_{\infty , [0,T]} \equiv  \sum_{f\in \MF_{[p]} } \sup_{0\le s<t \le T} | \BX^f_{s,t} |^{1/|f|}$$
and

$$
 ||| \BX |||_{\sup , [0,T]} \equiv  \sum_{f\in \MF_{[p]} } \sup_{0 \le t \le T} | \BX^f_{0,t} |^{1/|f| }.
$$
\begin{lem} \label{equivalent norms}
In the sense of a two-sided estimate, $ ||| \BX |||_{\infty , [0,T]} \asymp ||| \BX |||_{\sup , [0,T]}$.

\end{lem}

\begin{proof} It suffices to show the $\lesssim$ direciton. Write $\BX_{0,t} \equiv \BX_t$. In terms of the antipode, we can write
$$
\BX^f_{s,t} = \langle \BX_s^{-1} \otimes \BX_t , \Delta f \rangle = \sum_{(f)} \langle \BX_s, S (f^1) \rangle \langle \BX_t, f^2) \rangle.
$$
Note that the antipode preserve the degree of forests(number of nodes), which could be seen from the proof of Proposition \ref{hopf alg} by induction, so that each forest appearing in $S (f^1) = \sum_i f^1_i$ (finite sum) has the same degree, i.e. $|f^1_i| + |f^2| =  |f^1| + |f^2| = |f|$. Hence
$$
   | \BX^f_{s,t} | \lesssim \sum_{(f)} |  \langle \BX_s, S (f^1) \rangle \langle \BX_t, f^2) \rangle |  \lesssim  ||| \BX |||_{\sup , [0,T]}^{|f|}
$$
\end{proof}

\begin{rem}

As a homogenous group, $G$ admits a smooth subadditive, homogenous norm \cite{HS90}, say $N$, so that $d(g,h) := N (g^{-1} h)$ defines a left-invariant metric. Since all continuous homogenous norm are equivalent (see as in Thm 7.44 \cite{FV10}), we have equivalence of $N$ with  $\tilde N (g) := \sum_{f\in \MF_{[p]} }  | g^f |^{1/|f|}$.
This gives another proof of the above lemma,
$$
       ||| \BX |||_{\infty , [0,T]}  \asymp  \sup_{0\le s<t \le T} \sum_{f\in \MF_{[p]} } | \BX^f_{s,t} |^{1/|f|}       \asymp  \sup_{0\le s<t \le T} d (\BX_s, \BX_t) \le 2 \sup_{0\le t \le T} d (\BX_0, \BX_t)  \asymp  ||| \BX |||_{\sup , [0,T]} .
$$
The advantage of the earlier proof, however, is that it also shows immediately that
$$
           ||| \BX; \BX^n |||_{\infty , [0,T]} \to 0 \text{ iff } ||| \BX; \BX^n |||_{\sup , [0,T]} \to 0
$$
and this is further equivalent to, now in terms of the inhomogenuous distance,
$$
            || \BX ; \BX^n ||_{\infty , [0,T]} \to 0 .
$$
\end{rem}

\end{document}